\documentclass[reqno,11pt]{amsart}
\usepackage{amsthm,amsmath,amsfonts,amssymb,amscd,mathrsfs,graphics}
\usepackage{txfonts}
\usepackage{hyperref,supertabular}
\usepackage[all]{xy}
\usepackage{verbatim}

\usepackage{color, xcolor}
\usepackage{ulem}

\DeclareMathOperator{\Conf}{Conf}
\DeclareMathOperator{\loc}{loc}
\DeclareMathOperator{\per}{per}
\DeclareMathOperator{\Pro}{Pr}
\DeclareMathOperator{\Ve}{Ve}
\DeclareMathOperator{\Ed}{Ed} 
\DeclareMathOperator{\Fl}{Fl}
\DeclareMathOperator{\Int}{Int}
\DeclareMathOperator{\WI}{WI}

\newcommand{\A}{{\mathscr A}}
\newcommand{\B}{{\mathscr B}}

\newcommand{\CP}{{\mathbb{CP}}}
\newcommand{\C}{{\mathbb{C}}}

\newcommand{\E}{{\mathscr E}}

\newcommand{\G}{{\mathscr G}}

\newcommand{\I}{{\mathbf I}}

\newcommand{\J}{{\mathscr J}}

\newcommand{\LL}{{\mathscr L}} 

\newcommand{\MM}{\overline{\M}} 
\newcommand{\M}{{\mathscr M}}

\renewcommand{\Pro}{{Pr}}

\newcommand{\R}{{\mathbb{R}}}

\newcommand{\W}{{\mathscr W}}

\newcommand{\Z}{{\mathbb{Z}}}

\newcommand{\aut}{\operatorname{Aut}}

\newcommand{\cC}{\mathscr C} 

\newcommand{\cO}{\O}
\newcommand{\cR}{\mathscr R}

\newcommand{\cZ}{{\mathscr{Z}}}

\newcommand{\gt}{{\tilde{g}}}

\newcommand{\cCh}{\widehat{\cC}}

\newcommand{\ib}{{\bar{i}}}

\newcommand{\im}{\operatorname{Im}}
\newcommand{\ind}{\operatorname{index}}

\newcommand{\jb}{{\bar{j}}}

\newcommand{\lb}{{\bar{l}}}

\newcommand{\om}{{\omega}}

\newcommand{\re}{\operatorname{Re}}

\newcommand{\tM}{\widetilde{\W}}
\newcommand{\ub}{{\bar{u}}}

\newcommand{\ve}{{\varepsilon}}

\newcommand{\z}{{\mathbf z}}
\newcommand{\zb}{{\bar{ z}}}
\renewcommand{\O}{{\mathscr{O}}}
\renewcommand{\hom}{\operatorname{Hom}}

\newcommand{\Hb}{\mathbb{H}}

\newcommand{\pat}{{\partial}}
\newcommand{\bpat}{{\bar{\partial}}}

\newtheorem{result}{Main Result}
\newtheorem{thm}{Theorem}[section]
\newtheorem{lm}[thm]{Lemma}
\newtheorem{prop}[thm]{Proposition}
\newtheorem{crl}[thm]{Corollary}

\theoremstyle{definition}
\newtheorem{rem}[thm]{Remark}

\newtheorem{df}[thm]{Definition}
\newtheorem{ex}[thm]{Example}
\theoremstyle{remark}
 
\title[Fukaya category of LG model]
{Fukaya category of Landau-Ginzburg model}
\author{Huijun Fan}
\thanks{}
\address{School of Mathematical Sciences, Peking University, Beijing 100871, China.}
\thanks{The first author is supported by NSFC (11271028), NSFC (11325101), NSFC (11831017), and NSFC (11890661) }
\email{fanhj@math.pku.edu.cn}

\author{Wenfeng Jiang}
\thanks{The second author is partially supported by Chinese Universities Scientific Fund (74120-31610002)}
\address{School of Mathematics (Zhuhai), Sun Yat-Sen University, Zhuhai, 519082, China.}
\email{wen\_feng1912@outlook.com}

\author{Dingyu Yang}
\thanks{The third author is partially supported by ERC via the Starting Grant ERC StG-259118 by CNRS at Jussieu, National Science Foundation under agreement No. DMS-1128155 and Grant No. DMS-1638352 at IAS, and by ERC grant TRANSHOLOMORPHIC at Humboldt-Universit\"at zu Berlin.}
\address{Humboldt-Universit\"at zu Berlin, Rudower Chaussee 25, Room 1.303, 12489 Berlin, Germany.}
\email{yangding@math.hu-berlin.de, dingyu\_yang@cantab.net}
\date{\today}
\begin{document}
\maketitle

\begin{abstract}

This article introduces and provides the mathematical foundation of the open string Floer theory of Landau-Ginzburg model via Witten equation. We introduce the concept of regular tame exact Landau-Ginzburg system on a noncompact K\"ahler manifold, and define the notion of Landau-Ginzburg branes, as the objects of the Fukaya category. The study on Witten equation in our context provides the construction of the Fukaya category of Landau-Ginzburg model which was conjectured to be existed in Gaiotto-Moore-Witten's  work \cite{GMW} and Kapranov-Kontsevich-Soibelman's work \cite{KKS}. This is the first paper in the series of studies on Landau-Ginzburg models in the contexts of the mirror symmetry and other topics.

\end{abstract}
\maketitle
\tableofcontents

\section{Introduction}\
Fukaya category of Landau-Ginzburg model is an indispensable part of the general framework of homological mirror symmetry. Homological mirror symmetry (HMS) predicts the categorical quasi-equivalence between (split-completed derived) Fukaya category (A-model) of one K\"ahler manifold and bounded derived category of coherent sheaves (B-model) of its mirror partner, and vice versa for the same pair. This open string version of mirror symmetry was first proposed by Kontsevich \cite{k1} for Calabi-Yau pairs and has been established for many cases, for example, elliptic curve \cite{PZ}, Abelian varieties \cite{Fu2}, SYZ fibrations\footnote{In the presence of (singular) SYZ, Kontsevich's homological mirror symmetry from A-side to B-side can be viewed as a construction of the mirror as a non-commutative space via (instanton-corrected) family Floer theory of moduli of Lagrangian torus fibers equipped with local systems.} \cite{KS}, \cite{Fu1}, \cite{Ab4}, \cite{Tu}, Quartic surfaces \cite{Si11}, Products \cite{AI}, and Calabi-Yau projective hypersurfaces  \cite{Si-ICM,Sh1}.

Mirror symmetry conjecture can be extended beyond Calabi-Yau. In these cases, the Landau-Ginzburg (LG) models usually arise. For example, Kontsevich proposed to extend HMS to Fano-LG pairs \cite{k2}.  The framework of homological mirror symmetry was realized to further include general type-LG pairs, and (toric) LG-LG pairs \cite{GKR}, \cite{AAK}.

There are not only mirror symmetries between A-models and B-models, but also Landau-Ginzburg/Calabi-Yau (LG/CY) correspondence between a Calabi-Yau hypersurface and a LG model (with group action) whose description is essentially the defining equation for the Calabi-Yau. Open string LG B-model theory was constructed by Orlov (\cite{or1,or2}), and he also established the open string B-model LG/CY correspondence \cite{or3}. One also expects LG/CY correspondence on the A-side and so far efforts have been only concentrated on the closed string version, e.g. \cite{CIR}.

The geometric input of a LG model theory is a holomorphic Morse function (or holomorphic Morse fibration) $W$ on a noncompact K\"ahler manifold $(M, h)$.  Its open string A-model theory is expected to capture the information of the singularity of $W$ in the nonlinear background $M$ via Fukaya category constructed from the geometric input. It originates from the physics theory of (2,2) supersymmetry, e.g. theories of Kapustin-Li (\cite{ay}), Hori-Iqbal-Vafa (\cite{HV}).

In this paper, given a (regular) holomorphic Morse function $W$ on $M$, we construct a Lefschetz thimble from each critical point of $W$ as the stable manifold of the gradient flow of $\re W$. Each Lefschetz thimble is a smooth Lagrangian submanifold of $(M, \omega)$, where $\omega$ is the symplectic form associated to the K\"ahler metric $h$. We lift these Lefschetz thimbles $L$ to the universal Abelian cover $\widetilde{\mathcal{G}}M$ of the Lagrangian Grassmannian bundle $\mathcal{G}M$ using lifting maps $i^{\#}$ and with a pin structure $P^\#$ we get a set $\{L^{\#}=(L,i^{\#},P^{\#})\}$ called Landau-Ginzburg Lagrangian branes. For each pair $(L_0,L_1)$ of Lefschetz thimbles, we define a set of Lagrangian chords using (rescaled) Hamiltonian flow of $\re W$, denoted by $S_W(L_0,L_1)$. We can perturb the K\"ahler potential such that the set  $S_W(L_0,L_1)$ is finite (ref. Theorem \ref{transversal}).

For a pointed-boundary disc $\cC$ with $d+1$ punctures, we introduce the notion of directed Lagrangian labels, that is, any component $C\subset \pat \cC$ is labeled by a Lefschetz thimble $L_C$ and satisfies some order criterion. For each puncture $\zeta$ at the boundary of the disc, there are two adjacent components $C_{\zeta,0}$ and $C_{\zeta,1}$ with the corresponding Lagrangian labels $L_{\zeta,0}>L_{\zeta,1}$ where $L_{\zeta, i}:=L_{C_{\zeta, i}}$, and choose a Hamiltonian chord in $l_{\zeta}\in S_W(L_{\zeta,0},L_{\zeta,1})$ for each $\zeta$. We denote this choice $(\{L_C\}, \{l_\zeta\})$ of labels and Hamiltonian chords by $(L,l)$ for short.

We consider the (perturbed) Witten equation
$$
\pat_\zb u^j= \sum_{\ib}h^{\ib j}\pat_\ib \overline{W}\sigma+\mu^j.
$$
for maps $u=u^j$ from a pointed-boundary disc $\mathscr{C}$ to $M$. Here, $\sigma$ is the  the auxiliary section $\sigma$ of the $\log$-canonical bundle, and $\mu^j$ is some perturbation. The domain punctured disc is allowed to vary in the open Deligne-Mumford-Stasheff moduli, with a consistent choice of strip-like ends near punctures; and this universal curve is still denoted by $\mathscr{C}$. We require that each boundary component $C$ of the disc away from the punctures are mapped to Lefschetz thimbles $L_C$; and at each boundary puncture $\zeta$, the limit of such a solution (in the local coordinate of the strip-like end) is required to be $l_{\zeta}$. We denote by $\W^{d+1}(\cC, L, l;K)$ the space of such solutions (allowing domain variations), where $K$ denotes a consistent choice of perturbation, which is defined in Subsection \ref{subsection8.4}.

We assume that $M$ is exact and has bounded geometry, and $W$ satisfies a set of mild conditions called ``tame condition'', see Definition \ref{sec:df-tame-1}. Under these conditions, we study the compactness property of moduli spaces $\W^{d+1}(\cC, L, l;K)$. For $\phi\in \W^{d+1}(\cC, L, l;K)$, we provide a $C^0 -C^1$ interlocking compactness argument to get both the $C^0$ and $C^1$ bounds for $\phi$.

\begin{lm}[Lemma \ref{control-first-direction}]\
Let $\nu>0$. For any $(\phi: \mathscr{C}_b\to M)\in \mathscr{W}^{d+1}(\mathscr{C}, L, l; K)$ where $\mathscr{C}_b$ is a fiber in the universal curve $\mathscr{C}:=\mathscr{C}^{d+1}$, we have
$$
\sup_{z\in \mathscr{C}_b }|\pat_M W(\phi(z))|\leq C(\max_{z\in\mathscr{C}_b}|d\phi(z)|+1)^{1/\nu},
 $$
 where the constant $C$ depends only on $q$, $W$ and bounds in the bounded geometry of $M$.
\end{lm}

\begin{lm}[Lemma \ref{anti-control}]\label{L1.2}
There exists some $k>0$ such that
$$
\max_{z\in\mathscr{C}_b}|d\phi(z)|\le C(\sup_{z\in \mathscr{C}_b }|\pat_M W(\phi(z))|^k+1),
$$
for any $(\phi:\mathscr{C}_b\to M) \in \mathscr{W}^{d+1}(\mathscr{C}, L, l; K)$, where the constant $C$ depends only on $W$ and bounds in the bounded geometry of $M$.
\end{lm}
To get Lemma \ref{L1.2}, we use an approach based on isoperimetric inequality instead of bubbling-off analysis that is often used in studying the $J$-holomorphic curves. The tame condition for $W$ is crucial for this mechanism to work.

Combining these two lemmata, together with the tame condition, we get
\begin{result}[Theorem \ref{c_0_polytope}]
$$
d(\phi,q_0)<C,\; |d\phi|<C,
$$
where $q_0\in M$ is the chosen point in the tame condition.
\end{result}

This result, together with the standard Deligne-Mumford-Stasheff compactifications of pointed-boundary discs, leads to the compactness of the moduli spaces. And it pave the way towards defining an $A_{\infty}$ category, called the Fukaya category of the LG model.

The Fukaya category $\text{Fuk}(M,h,W)$ of a Landau-Ginzburg model $(M,h,W)$ consists of the following data:

\begin{itemize}
\item An object set $\text{Ob}(\text{Fuk}(M,h,W))$, consisting of all Landau-Ginzburg branes,
\item For each pair $(L_0^{\#},L_1^{\#})$ of objects, a morphism $\text{Hom}(L_0^{\#},L_1^{\#})=CF(L_0^{\#},L_1^{\#})$ as a module generated by the set $S_W(L_0,L_1)$ of Hamiltonian chords, and
\item Composition (linear) maps
$$
\mu^d:\text{Hom}(L_0^{\#},L_1^{\#})\otimes\cdots\otimes \text{Hom}(L_{d-1}^{\#},L_d^{\#}) \mapsto \text{Hom}(L_0^{\#},L_d^{\#}),
$$
defined using $\mathscr W^{d+1}(\mathscr C, L, l;K)$ and whose details are stated in Subsection \ref{subsection9.2}, which satisfy the $A_\infty$ relation.
\end{itemize}
\begin{result}[(\ref{a-infty-relation})]\

For each $(a_1, \cdots, a_{i+j+k})\in \text{Hom}(L_0^{\#},L_1^{\#})\otimes\cdots\otimes \text{Hom}(L_{i+j+k-1}^{\#},L_{i+j+k}^{\#})$, we have
$$
\sum_{i+j+k=d}(-1)^{\dag}\mu^{i+k+1}(a_{1}\cdots a_{i},\mu^j(a_{i+1},\cdots,a_{i+j}),a_{i+j+1},\cdots,a_{i+j+k})=0.
$$
\end{result}

This result together with the grading result obtained in Subsection \ref{subsection9.2} ensures that $\mu^d$ satisfies the $A_{\infty}$ relation. In particular, we get an $A_\infty$ category from a regular tame LG model on an exact K\"ahler manifold with bounded geometry.

There are already rich studies on LG A-models. One of them is Fukaya-Seidel category (\cite{Si1,Si2,Si3}) for holomorphic Morse fibration $\pi:E \mapsto D$. Here a regular fiber $F=\pi^{-1}(z_0)$ was chosen together with paths from $z_0$ to critical points of $\pi$ called vanishing paths. Using parallel transport along the vanishing paths, a set of Lefschetz thimbles was obtained, and they intersect $F$ and give a set of Lagrangian spheres called vanishing cycles.  By studying the Floer equation in $F$, a Fukaya category of vanishing cycles was defined under some conditions of the geometry of the fibration, and a directed subcategory was extracted. In \cite{Ab1}, Abouzaid studied the Fukaya category of compact Lagrangians possibly with boundaries in a regular reference fiber $\pi^{-1}(z_0)$ of $\pi:E\mapsto \C$ which project and intersect in a specific way near $z_0$ in $\C$. Seidel also studied the Fukaya category of Lefschetz fibrations $E\mapsto \C$, in his series of papers (\cite{Si4,Si5,Si6,Si7,Si8,Si9,Si10}), where he used Lefschetz thimbles instead of vanishing cycles. There are many other models, e.g., Haydys' theory for LG equations on the complex plane (\cite{h1}) and Sylvan's model using ``stops'' in Liouville domains (\cite{Sy}).

The differences and contributions of our study are as follows. First, unlike theories of Seidel, Abouzaid and Sylvan, our model is exactly the same as in the physics literatures of LG models, which explicitly involve $W$ in the equation. This LG theory is also expected to be equivalent to Gaiotto-Moore-Witten's web-based theory on the algebra of the infrared (\cite{GMW}).  Second,  the noncompactness of the manifold in LG model gives rise to the problem of compactness, especially, the $C^0$ estimate. Usually, technical convenience of Liouville domains has been used, and $C^0$ compactness is achieved vertically using holomorphic convexity and horizontally using holomorphic projection $W$ and maximum principle (\cite{Si1,Si2, {Si3}, Si4,Si5,Si6,Si7,Si8,Si9,Si10,Ab1,Sy}). In \cite{h1} the function $W$ was assumed to be proper. But in applications, it is natural to have $M$ admitting compactification with a normal crossing divisor or worse asymptotic geometry, and non-proper $W$. We formulate a general class of such $(M,h,W)$ satisfying so-called tame condition. Under the tame condition together with bounded geometry and exactness\footnote{We choose to focus on our new idea in the exact setting first instead of full generality, and we expect to be able to remove exactness in the future.} of $M$, we provide an entirely different compactness argument via Lemmata 1.1 and 1.2 above. This new framework can be expected to be applicable and generalizable to other contexts. Third, we do not need to convert $W$ into a Lefschetz fibration with nice boundary behavior and our Lefschetz thimbles are canonically chosen, rather than just well-defined up to exact deformations, both of which are convenient when considering group action $\text{Aut}(W)$; and intersections of thimbles are naturally realized via Witten equation without specific recipe.

There are many motivations for this study, as well as  expected relations to other existing theories. We will elaborate below in three aspects.

(1) The quantum singularity theory via Witten equation developed in \cite{FJR1,FJR2,FJR3} (commonly referred to as FJRW theory), is a closed string invariant about singularity of a (non-degenerate) quasi-homogeneous polynomial $W$ on $\C^N$ with action of a choice of an appropriate subgroup of the symmetry group $\text{Aut}(W) \subset (\C^*)^N$ of $W$. The setting in \cite{FJR3}, where FJRW invariant was constructed, is an example that satisfies the tame condition. This $W$ gives  rise to the moduli spaces of $W$-structures covering Deligne-Mumfold spaces of pointed orbicurves, which are spaces of choices to identify two sides of Witten equation on sections of the associated orbifold line bundles. In FJRW theory, a $W$-curve (a $W$-structure with a solution of the associated Witten equation) converges at any interior puncture $x$ to critical points of $W$ restricted to the fixed locus of the stabilizer of $x$. We can now consider $W$-curves with pointed boundaries (where boundary punctures necessarily have trivial stabilizers and hence are broad), then the collection of line bundles over the W-curve trivialize into $\C^N$ and we can impose thimble boundary conditions as in this paper. Note that the asymptotic convergence to LG chords at boundary punctures is consistent to the aforementioned asymptotics in FJRW. Wall-crossing behaviors (under different Morsifications) in both open and closed settings also perfectly match, as we will follow up in the future. This gives an enriched version of Fukaya category defined in this paper and open-closed/closed-open maps with FJRW invariants, which under the conjectural LG/CY correspondence will be intertwined with the counterparts between Gromov-Witten invariants and the usual Fukaya category.

(2) The LG-LG Mirror symmetry:

In principle, from a LG model, one can construct another LG as its mirror partner, so that the LG A-model of one should be mirror symmetric to LG B-model of the other, vice versa. There are already some examples involving LG models in homological mirror symmetry.

(i) (Fano-LG) The A side of Fano projective hypersurface homologically mirror symmetric to LG B-model was proved in \cite{Sh2}; and A side of toric LG A-model to B-side of toric Fano was established in \cite{Ab1,Ab2,Ab3}. Please see \cite{Au} for a comprehensive survey and many other results cited therein.   Of interest to our approach, in the paragraph after \cite[Definition 5.3]{Au}, it was said that `one should place the reference fiber ``at infinity'', i.e. either consider a limit of this construction as $t\to +\infty$ ...'; during this process the compactness control not obvious, whilst in our approach, the ``reference fiber'' is at infinity (using Witten equation instead of small perturbation of holomorphic curve equation).

(ii) (general type-LG) In \cite{ AAK}, from a hypersurface $H$ of general type in an affine toric variety, a Fano $X$ was constructed and equipped with a Morse-Bott holomorphic function with critical manifold $H$. A (twisted) A-model of the resulting LG was shown to be homologically mirror symmetric to a LG B-model (first from Fano $X$ to LG as in (i) which is in turn partially compactified to get a LG on $X$).

(iii) LG-LG homological mirror symmetry:

(a) Seidel in a series of papers \cite{Si4,Si5,Si6,Si7,Si8,Si9,Si10} among other interesting key observations and constructions, constructed from an anticanonical pencil of hypersurfaces in a suitable ambient space, a non-commutative anticanonical pencil, which manifests an enchanced version of LG-LG mirror symmetry.

(b) Item (ii)  above is an instance of a LG-LG mirror symmetry, even though LG A-side is of a simple kind, a Morse-Bott holomorphic function with a single critical manifold. In general, if open-string A-model LG/CY correspondence alluded to in item (1) can be established and generalized to LG/beyond-CY (as the CY condition plays a less crucial role here unlike in mirror symmetry), this together with items (i) and (ii) above should give LG-LG mirror symmetry.  As an evidence of LG/beyond-CY, \cite[Conjecture 1.1]{Au2} should follow from item (ii) by applying a version of LG/beyond-CY on the A-side and Orlov's equivalence \cite{or1} on the B-side.

(c) Generalization of HMS beyond Calabi-Yau in essence is to first generalize to open Calabi-Yau, then consider (possibly normal crossing) divisor partial (de)compactification or restrict to its hypersurface, and iterate and analyze the relations. So far the literature involving LG model is mostly based on toric duality. For general $W$ with bounded critical values, one can require Lagrangians whose images under $W$ in $\mathbb{C}$, if noncompact, are mapped horizontally to the left near infinity (Kontsevich often uses this version, for example \cite{KKP}), and near infinity the situation is same as in this paper. We expect our theory to be generalized to weakly unobstructed Lagrangians of this form. In the presence of SYZ (singular) fibration in the complement of normal crossing anticanonical divisor, this version should provide a Family Floer tool for constructing a mirror; and for the holomorphic function on the mirror side, one has to virtually count (Maslov 2) disc configurations with boundaries on SYZ fibers with extra data. Many steps are not fully worked out and we expect our theory to be fruitful towards establishing general LG-LG Mirror symmetry.

(3)
In Gaiotto-Moore-Witten's web-based formalism (\cite{GMW}) of LG theory for holomorphic Morse $W$, with the underlying algebraic structure mathematically recasted by Kapranov-Kontsevich-Soibelman (\cite{KKS}), one considers polytopes generated by singular values of $W$ in $\mathbb{C}$ and secondary fans of all its possible regular polyhedral subdivisions (the dual of which is the space of webs), one can cook up a $L_\infty$-algebra $\mathcal{R}_\infty$ from this. From a regular polyhedral subdivision, one can consider a soliton connecting critical points lying over each boundary of this subdivision, and consider moduli spaces of solutions of Witten equations on polygons asymptotic to these boundary conditions and stitched up. Degenerations of moduli spaces into moduli spaces stitched over finer subdivisions give rise to the algebra structure, and rigid counts are conjectured to provide a Maurer-Cartan element for $\mathcal{R}_\infty$ and the associated deformation is expected to recover LG category (\cite[Conjecture 14.10]{KKS}). In this paper, we use the same Witten equation as in GMW and KKS, and techniques and viewpoints, especially compactness etc, should pave the way to rigorously construct this algebra of infrared and approach this conjecture.

This article is organized as follows. In Section \ref{sec-2} we introduce the regular tame exact Landau-Ginzburg system, which is the background for our theory.  In particular, we define the tame condition for a Landau-Ginzburg system $(M,h,W)$.

In Section \ref{sec-3} we investigate the Lefschetz thimbles of a Landau-Ginzburg system $(M,h,W)$, which are stable manifolds of equation (\ref{sec2:equa-4}). We use Sard-Smale method to show that for generic $h$ and time $T$, the Lefschetz thimbles intersect each other transversally (at each integer speed), in the sense of Definition \ref{trans-dfn}. Once these are established, we can define the Witten equation for maps from the strip $\R\times[0,1]$ to $M$.

Section \ref{sec-5} gives a brief review of Deligne-Mumford-Stasheff compactifications for pointed discs. We also introduce the auxiliary section of the $\log$-canonical bundle. All these are necessary in defining the Witten equation for maps from pointed discs to $M$, in Section \ref{sec-6}.

In Section \ref{sec-7} we study the compactness of the moduli spaces consisting of solutions of the Witten equation.  The tame condition, together with the bounded geometry and exactness of $M$, will lead to a $C^0$ bound for solutions, which is crucial in the compactness, as the manifold $M$ is noncompact. 

Section \ref{sec-8} studies Fredholm theory of the equation.  In particular, we introduce Landau-Ginzburg branes in Subsection \ref{index-orientation section}, which are objects of the Fukaya category of Landau-Ginzburg model.  We also introduce the grading datum, which gives  gradings for Hamiltonian chords. Note that our datum is slightly different from the one in Sheridan's paper \cite{Sh1}, as we encode the speed of the chords into the datum, in addition.

The last section is devoted to the definition of the Fukaya category of the LG model. We adopt the algebraic framework of Sheridan \cite{Sh1} and give a brief review of it in Subsection \ref{subsec-9.1}. This together with the compactness result in Section \ref{sec-7} will make up an $A_{\infty}$ category of the Landau-Ginzburg model.

\subsection*{Acknowledgement} The idea of this paper can be traced to some small notes written by Y. Ruan between 2006-2007 when he discussed this problem with the first author. All of three authors thanks Y. Ruan for his very helpful suggestions and encouragement. The first author also wants to thank K. Fukaya for his many help and support.

The second author thanks G. Tian and Y.-G. Oh for their discussion and support. Part of this work was done when he was a postdoc in School of Mathematics Science, Peking University.

The third author thanks H. Hofer, A. Oancea and C. Wendl for their support, encouragement and helpful discussions and appreciates the excellent research atmosphere and interactions at Jussieu, IAS and HU-Berlin where part of the work was done. D.Y. is grateful to have learned much from articles, lectures and talks of and discussions with K. Fukaya, P. Seidel, D. Auroux and M. Abouzaid at IAS.

\section{Regular tame exact Landau-Ginzburg system}\label{sec-2}\

In this paper, we always assume that $(M,h)$ is an $n$-complex dimensional complete noncompact K\"ahler manifold with bounded geometry and $W$ is a nontrivial holomorphic function on $M$, which is called a \textit{potential function}. We call the tuple $(M,h,W)$ as a \textit{Landau-Ginzburg (LG) system}. If the K\"ahler form $\om$ of the K\"ahler metric $h$ is exact, then we call $(M,h,W)$ an \textit{exact LG system}.

Let $z^j=x^j+iy^j, j=1,\cdots,n,$ be the local complex coordinates of $M$. Then locally the K\"ahler metric $h$ has the associated K\"ahler form $\om=\frac{i}{2}\sum_{ij}h_{i\jb}dz^i\wedge d z^\jb$. Let $J$ be the (almost) complex structure defined on $TM$, then the complexified tangent bundle $T_\C M$ has the orthogonal decomposition $T_\C M=T^{1,0} M\oplus T^{0,1}M$ with respect to $h$. The K\"ahler metric $h$ when restricted to $T^{1,0}M$ has the form $h(\;,\;)=\;\frac{1}{2}(g(\;,\;)-i\om(\;,\;))$. Here $g$ is the Riemannian metric on the underlying Riemannian manifold $(M,g)$ which is compatible with the complex structure $J$. Let $\nabla$ be the Chern connection of the holomorphic tangent bundle $T^{1,0} M$, then it is torsion free and the corresponding connection $D$ on $TM$ is just the Levi-Civita connection of $(M,g)$ such that $DJ=0$.   We write the Cauchy Riemann operator on $M$ as $\bpat_M$, whose complex conjugate operator is denoted as $\pat_M$. $\bpat_M$ is the $(0,1)$-part of the Chern connection $\nabla$.

We denote by $\Lambda^{p,q}(M)$ the bundle of $(p,q)$ forms on $M$.

\begin{df}\label{sec:df-tame-1} Given a LG system $(M,h,W)$, we say that it satisfies the \textit{tame condition} if there exist a base point $q_0\in M$ and constants $C_1,C_2,\delta>0$ such that for any point $\phi=(u, \bar u)=(u_1,\cdots,u_n,\ub_1,\cdots\ub_n)\in M$
\begin{equation}\label{sec3:tame-cond-1}
\begin{cases}
(i)\quad |W(u)|+|\nabla_M^2 W(u)|\le C_1 d(\phi,q_0)|\pat_M W(u)|+C_2,\\
(ii)\quad d(\phi,q_0)\le C_1|\pat_M W(u)|+C_2,\\
(iii)\quad |\pat_M W(u)|\le C_1e^{\delta d(\phi,q_0)}+C_2.
\end{cases}
\end{equation}
\end{df}

Note that $W(\phi)=W(u, \bar u)=W(u)$ here, because $W$ is holomorphic.

From the definition, we immediately get the estimate
\begin{equation}\label{sec2:iden-tame-1}
|W(u)|+|\nabla_M^2 W(u)|\le C'_1|\pat_M W(u)|^2+C'_2
\end{equation}
for some constants $C'_1,C'_2$.

By (ii) of Definition \ref{sec:df-tame-1}, we know that there are no critical points of $W$ outside a compact set $K$ and there are finitely many isolated critical points of $W$ inside $K$. Denote by $C_W$ the set of critical points of $W$ on $M$. For each critical point $p$ of $W$, we can define the Milnor number $\mu_W(p)$ as the dimension of the $\C$-algebra $\O_p/J_p$, where $J_p$ is the Jacobi ring of $W$. We can define the multiplicity of $W$ on $M$ as the sum of Milnor numbers:

\begin{df}[Multiplicity]If $(M,h,W)$ is a tame LG system, then the \textit{multiplicity} of $W$ on $M$ is defined by
\begin{equation}
\mu_W(M)=\sum_{p\in C_W}\mu_W(p).
\end{equation}
\end{df}

\begin{df}A LG system $(M,h,W)$ is called a \textit{Morse} LG system, if the potential function $W$ is a Morse function. $W$ is called \textit{regular Morse}, if it is Morse and for any pair $p,q\in C_W, p\neq q,$ we have $\im W(p)\neq \im W(q)$. If $W$ is regular Morse, then the LG system $(M,h,W)$ is called a \textit{regular Morse LG system}. If $(M, h, W)$ is a tame LG system, and $W$ is regular Morse function, then $(M,h,W)$ is called a \textit{regular tame LG system}. A \textit{regular tame exact LG system} $(M,h,W)$ is a regular tame LG system where the K\"ahler form for $h$ is exact.
\end{df}

There are two important examples of tame LG systems.

Let $W \in \mathbb{C}[x_1, \dots, x_N]$. If for any $\lambda\in \C^*$, there exist positive fractional numbers $q_1,\cdots, q_n, \gamma$ such that the following equality holds:
$$
W(\lambda^{q_1}x_1,\cdots, \lambda^{q_n}x_n)=\lambda^\gamma W(x_1,\cdots,x_n),
$$
then $W$ is called a \textit{quasi-homogeneous polynomial} and each $q_i$ is called the \textit{weight} of the variable $x_i$ and $\gamma$ is called the \textit{total weight} of $W$.  We say $W$ is \textit{non-degenerate} if (1) the choices of weights $q_i$ with the same $\gamma$ (which can be assumed to be 1 without loss of generality) are unique, and (2) $W$ has a singularity only at zero.

\begin{lm}\cite[Theorem 5.8]{FJR1}\label{sec3:lm-1} Let $W \in \mathbb{C}[x_1, \dots, x_N]$ be a non-degenerate,
quasi-homogeneous polynomial with weights $q_i<1$ for each
variable $x_i,i=1, \dots,N$, and with total weight $1$. Then for any tuple $(u_1, \dots,
u_N) \in \mathbb{C}^N$ we have

\[ |u_i| \leq C \left(\sum^N_{i=1}\left|\frac{\partial W}{\partial x_i}(u_1, \dots,
u_N)\right|+1 \right)^{\delta_i},\] where
$\delta_i=\frac{q_i}{\min_j(1-q_j)}$ and the constant $C$ depends
only on $W$. If $q_i\leq 1/2$ for all $ i \in \{1, \dots, N\}$, then
$\delta_i\le 1$ for all $ i \in \{1, \dots, N\}$. If $q_i<1/2$ for
all $ i \in \{1, \dots, N\}$, then $\delta_i<1$ for all $ i \in \{1,
\dots, N\}$.
\end{lm}

\begin{prop}Let $W(u_1,\cdots,u_n)$ be a non-degenerate quasi-homogeneous polynomial on $\C^n$ with each weight $q_i\le 1/2$ and the total weight $1$. Let $W_0(u_1,\cdots,u_n)$ be the sum of any monomials with weight (which is the sum of the weights of all factors) no more than $1$. Assume that $h$ is a K\"ahler metric on $\C^n$ which is a compact perturbation of the Euclidean metric.  Then $(\C^n, h, W+W_0)$ is a tame LG system.
\end{prop}

\begin{proof}Take the base point $q_0=0$ in $\C^n$. Without loss of generality, we assume that $h$ is the Euclidean metric and let $W_0$ be a monomial with weight no more than $1$. If $W_0$ is of weight $1$, then $W+W_0$ is of weight $1$ and by quasi-homogeneity, we have
$$
(W+W_0)(u_1,\cdots,u_n)=\sum q_i u_i \pat_{u_i}(W+W_0)(u_1,\cdots,u_n).
$$
If the weight of $W_0$ is less than $1$, then Lemma \ref{sec3:lm-1} implies that
$$
|\pat W_0|\le \frac{1}{2}|\pat W|+C_0
$$
and
$$
|W_0|\le C (|u||\pat W|+1).
$$
Then there is
$$
|W+W_0|\le C (|u||\pat W|+1)\le C(|u||\pat (W+W_0)|+1).
$$
So in either cases, we get the control of $|W|$ in (i) of (\ref{sec3:tame-cond-1}). Similarly, we can get the control of $|\nabla^2_M W|$.

By Lemma \ref{sec3:lm-1}, it is easy to see that
$$
|u|\le C (|\pat (W+W_0)|+1).
$$
This proves (ii) of (\ref{sec3:tame-cond-1}). Since $|\pat_M W|$ is polynomial growth, (iii) of (\ref{sec3:tame-cond-1}) naturally holds.
\end{proof}

\begin{ex} $(\C^5, \sum_i \frac{\sqrt{-1}}{2}dz^i\wedge dz^\ib, W(z_1,\cdots,z_5)=z_1^5+\cdots z_5^5+a_1z_1+\cdots+a_5 z_5 )$ is a tame exact LG system. Here $a_1,\cdots, a_5\in \C$ are perturbation parameters. For generic parameters $a_i's$, the system is a Morse tame exact LG system. If $(a_1,\dots, a_5)$ lies in some chambers in $\C^5$ separated by ``walls'' consisting of real hypersurfaces, then the system is a regular tame exact LG system (ref. \cite{FJR3}).
\end{ex}

Another important example is the system $((\C^*)^n,\frac{i}{2}\sum_l \frac{dz^l\wedge dz^\lb}{|z|^2}, f)$, where $f$ is a Laurent polynomial of the form
$$
f(z_1,\cdots,z_n)=\sum_{\alpha=(\alpha_1,\cdots,\alpha_n)\in \Z^n}a_\alpha z^\alpha.
$$
The Newton polyhedron $\Delta=\Delta(f)$ of $f$ is the convex hull of the integral points $\alpha=(\alpha_1,\cdots,\alpha_n)\in \Z^n$ with nonzero $a_\alpha$'s. $f$ is said to be \textit{convenient} if $0$ is in the interior of the Newton polyhedron. Let $\Delta'\subset \Delta$ be an $l$-dimensional face of $\Delta$. Define the Laurent polynomial with the Newton polyhedron $\Delta'$
$$
f^{\Delta'}(z)=\sum_{\alpha'\in \Delta'}a_{\alpha'}z^{\alpha'}.
$$
For any Laurent polynomial $g$, denote by $g_i,1\le i\le n$ the logarithmic derivatives of $g$:
$$
g_i(z)=z_i\frac{\pat}{\pat z_i}g(z).
$$
\begin{df}
A Laurent polynomial $f$ is called \textit{non-degenerate} if for every $l$-dimensional face $\Delta'\subset \Delta(l>0)$ the polynomial equations
$$
f^{\Delta'}(z)=f^{\Delta'}_1(z)=\cdots=f^{\Delta'}_n(z)=0.
$$
has no common solutions in $T$.
\end{df}

\begin{prop}If $f$ is a convenient and non-degenerate Laurent polynomial on $((\C^*)^n,\frac{i}{2}\sum_l \frac{dz^l\wedge dz^\lb}{|z|^2})$, then $((\C^*)^n,\frac{i}{2}\sum_l \frac{dz^l\wedge dz^\lb}{|z|^2}, f)$ is a tame exact LG system.
\end{prop}

\begin{proof}The proof of the tameness is much like the proof in \cite[Proposition 2. 48]{Fa}. Let $z_i=e^{t_i}$,  then in the new coordinates the metric $h=\frac{i}{2}\sum_i dt^i\wedge dt^\ib$ and there is
$$
f(z_1,\cdots,z_n)=f(t_1,\cdots,t_n)=\sum_{\alpha\in \Delta}a_\alpha e^{\langle \alpha,t\rangle}.
$$
Write $t=(t_1,\cdots,t_n)=\beta|t|=(\beta_1,\cdots,\beta_n)|t|$, where $\beta_i=t_i/|t|,|\beta|\le 1$. The real part $R(\beta)=(\re(\beta_1),\cdots,\re(\beta_n))$ defines a line in $\R^n$ through the origin. Since $f$ is convenient, there exist at least two directions $\alpha_\pm$ such that
$\langle R(\beta),\alpha_-\rangle<0$ and $\langle R(\beta),\alpha_+\rangle>0$. We arrange all the vertices $\alpha$ such that
\begin{equation}
\langle R(\beta),\alpha_0\rangle\le \langle
R(\beta),\alpha_1\rangle\cdots<\cdots\le 0\le \cdots<\cdots \langle
R(\beta),\alpha_s\rangle,
\end{equation}
where $s+1$ is the number of the vertices of the Newton polyhedron $\Delta$. Denote by $M(\beta)=\langle R(\beta),\alpha_s\rangle$, $\M$ the set
of all of $\alpha$ such that $\langle R(\beta),\alpha\rangle=M(\beta)$ and $\M^+$ the set of all $\alpha$ such that $0<\langle R(\beta),\alpha\rangle<M(\beta)$. Hence we have along the line $t$ that
\begin{equation}\label{sec2:ineq-laurent-1}
|f(t)|\le \sum_\alpha |a_\alpha| e^{\langle R(\beta),\alpha\rangle|t| }\le C_1 e^{M(\beta)|t|}+C_2.
\end{equation}
Choose $\delta=\max_{\beta}M(\beta)$, then we get (iii) of Definition \ref{sec:df-tame-1}.

On the other hand, along the complex line $t$ there is
\begin{align}\label{sec2:inequ:laurent-2}
&|\pat_t f|^2=\sum_i\left|\sum_\alpha a_\alpha \alpha_i e^{\langle
\alpha,t\rangle}\right|^2\\
&\ge \sum_i \left|\sum_{\alpha\in \M}
a_\alpha \alpha_i e^{\langle \alpha,\beta\rangle|t|}\right|^2-
C\sum_{\alpha\in \M^+}|e^{\langle \alpha,\beta\rangle|t|}|^2-C\\
&=\left(\sum_i \left|\sum_{\alpha\in \M} a_\alpha \alpha_i e^{i Im(\langle
\alpha,\beta\rangle)|t|}\right|\right) e^{2 M(\beta)|t|}-
C\sum_{\alpha\in \M^+}|e^{\langle \alpha,\beta\rangle|t|}|^2-C
\end{align}

 Claim: for any $\theta \in \R$, there holds

\begin{equation}\label{inequ-laurent-2.5}
\sum_i\left|\sum_{\alpha\in \M} a_\alpha \alpha_i e^{i Im(\langle
\alpha,\beta\rangle)\theta}\right|>0.
\end{equation}
If the claim is true, then there exists constants $C_1,C_2$ depending only on the parameters of $f$ such that
$$
|\pat f|\ge C_1e^{M(\beta)|t|}-C_2.
$$
Combining (\ref{sec2:ineq-laurent-1}), we obtain the control of $|f|$ in (i) of Definition \ref{sec:df-tame-1}. Similarly, we can get the control of $|\nabla^2_M f|$.  The inequality (ii) of Definition \ref{sec:df-tame-1} is obvious. Therefore, we have proved that $((\C^*)^n,\frac{i}{2}\sum_l \frac{dz^l\wedge dz^\lb}{|z|^2}, f)$ is a tame system

To prove the Claim, we prove by contradiction. Suppose that
(\ref{inequ-laurent-2.5}) is not true, then there exists a
$\theta_0$ such that for any $i=1,\cdots,n$, there is
\begin{equation}
\sum_{\alpha\in \M} a_\alpha \alpha_{i} e^{i Im(\langle
\alpha,\beta\rangle)\theta_0}=0,i=1,\cdots,n,
\end{equation}
Multiplying the above equality by $R(\beta_i)$ and taking the sum,
noticing that for any $\alpha\in \M$, $\sum_i \alpha_i
Re(\beta_i)=M(\beta)$, we obtain
\begin{equation}
\sum_{\alpha\in \M} a_\alpha  e^{i Im(\langle
\alpha,\beta\rangle)\theta_0}=0.
\end{equation}
This contradicts with the fact that $f$ is non-degenerate.
Therefore, we proved the Claim.
\end{proof}

\section{Lefschetz thimble and transverse intersection}\label{sec-3}

\subsubsection*{Hamiltonian vector field}\

The K\"ahler manifold $(M,h)=(M, \omega, g)$ is a (real) symplectic manifold $(M, \omega)$ equipped with the compatible complex structure $J$. If $H$ is a (real) smooth function on $M$, then the Hamiltonian vector field $X_H$ is defined by $\iota_{X_H}\om=dH$, i.e, $X_H$ is induced by the isomorphism
$$
\tilde{\om}: TM\to T^*M.
$$
On the other hand, if we denote by $\nabla_g H$ the gradient vector field of $H$ with respect to the metric $g$, then we can show
$$
X_H=-J\cdot \nabla_g H,
$$
since we have the identities that $(\iota_{X_H}\om)(Y)=\om(X_H, Y)=dH(Y)=g(\nabla_g H,Y)$ and $\om=g\circ (J\times Id)$. The Hamiltonian equation is defined as
\begin{equation}
\dot{\phi}(t)=X_H(\phi(t))=-J\cdot \nabla_g H(\phi(t)).
\end{equation}

However, since $M$ is K\"ahler, we have the more refined isomorphisms as below
\begin{equation}
\tilde{\om}:T^{1,0}M\to \Lambda^{0,1}(M),\quad T^{0,1}M\to \Lambda^{1,0}(M).
\end{equation}
Let $W$ be a holomorphic function on $M$, by these two isomorphisms, we have
\begin{align}
\tilde{\om}^{-1}(\overline{\pat_M W}/2)&=\sum_l(-i\sum_{\ib} h^{\ib l}\pat_{\ib}\overline{W})\frac{\pat}{\pat z^l}=:X_W\nonumber\\
\tilde{\om}^{-1}({\pat_M W}/2)&=\sum_{\lb}(i\sum_{i} h^{\lb i}\pat_{i}{W})\frac{\pat}{\pat z^\lb}=\overline{X_W}.
\end{align}
Hence the vector field $X_{\re(W)}$ has the decomposition:
\begin{equation}
X_{\re(W)}=X_W+\overline{X_W},\;X_{\im(W)}=i(X_W-\overline{X_W}).
\end{equation}
Let $\phi:\R\mapsto M$ be given by $t\mapsto (\z(t),\bar{\z}(t))$, then the Hamiltonian equation for $\re(W)$ is given by
\begin{equation}\label{sec2:equa-0}
J\cdot \frac{\pat \phi(t)}{\pat t}=\nabla_g (\re(W))(\phi(t)),
\end{equation}
where $\phi(t)=(x_1(t),\cdots,x_n(t),y_1(t),\cdots,y_n(t))^T$.
This equation can be decoupled to the equation
\begin{equation}\label{sec2:equa-1}
\frac{\pat \z(t)}{\pat t}=X_W
\end{equation}
and its complex conjugate equation. Here we have used the facts that $z^i(t)=\phi^i(t)+i\phi^{i+n}(t)$ and
$$
\frac{1}{2}(\frac{\pat\phi}{\pat t}-iJ\cdot \frac{\pat \phi}{\pat t})=\sum_i\frac{\pat z^i}{\pat t}\frac{\pat}{\pat z^i},\;\frac{1}{2}(\frac{\pat\phi}{\pat t}+iJ\cdot \frac{\pat \phi}{\pat t})=\sum_{\ib}\frac{\pat z^\ib}{\pat t}\frac{\pat}{\pat z^\ib}.
$$

In local coordinates,  the Equation (\ref{sec2:equa-1}) has the form
\begin{equation}\label{hal-complex}
i\cdot\frac{\pat z^l(t)}{\pat t}=\sum_{\ib} h^{\ib l}\pat_{\ib}\overline{W}.
\end{equation}

We call the (half) Hamiltonian system (\ref{hal-complex}) as the  Witten equation of the (holomorphic) potential function $W$. Hence $u(t)$ is the solution of the Witten equation if and only if $\phi(t)=(u(t),\overline{u(t)})$ is the solution of the Equation (\ref{sec2:equa-0}).

\begin{lm}\label{sec2:lm-1}
We have
\begin{equation}
X_{\re(W)}=-J\cdot \nabla_g (\re(W))=-\nabla_g(\im(W))=-J\cdot X_{\im(W)}.
\end{equation}
\end{lm}

\begin{proof} By definition, we have
$$
X_{\re(W)}=X_W+\overline{X_W},\quad X_{\im(W)}=i(X_W-\overline{X_W}),
$$
then
$$
J\cdot X_{\im(W)}=-i(J\cdot\overline{X_W}-J\cdot X_W)=-X_{\re(W)}.
$$
Using the identity
$$
X_{\im(W)}=-J\cdot \nabla_g \im(W),
$$
we get the conclusion.
\end{proof}
By Lemma \ref{sec2:lm-1}, we have
$$
X_{\re(W)}=-\nabla_g (\im(W)).
$$
Hence $\phi(t)$ satisfies the Equation (\ref{sec2:equa-0}) if and only if it satisfies the equation
\begin{equation}\label{hal}
\frac{\pat \phi(t)}{\pat t}=-\nabla_g (\im(W)).
\end{equation}

\subsubsection*{Lefschetz thimble and its basic properties}\

Now we study the gradient flow of $\re(W)$:
\begin{equation}\label{sec2:equa-4}
\frac{\pat \psi(s)}{\pat s}=\nabla_g (\re(W))(\psi(s)).
\end{equation}

The flow $\psi(s)$ defines a diffeomorphism of $M$: $\psi_s: p\mapsto \psi_s(p)=:p\cdot s$. We define the stable manifold $L^+_p$ and the unstable manifold $L^-_p$ of this flow at the critical point $p\in C_W$:
\begin{align*}
&L^-_p:=\{q\in M|q\cdot s\to p,\;\text{as}\;s\to -\infty\}\\
&L^+_p:=\{q\in M|q\cdot s\to p,\;\text{as}\;s\to +\infty\}
\end{align*}

\begin{df}
$L^-_p,L^+_p$ are called the \textit{Lefschetz thimbles} attached to the critical point $p$ with respect to $W$. Similarly for $\theta\in [0,2\pi)$, we can define the Lefschetz thimbles $L^\pm_{p}(\theta)$ attached to the critical point $p$ with respect to $e^{i\theta}W$.
\end{df}

\begin{lm}We have
\begin{itemize}
\item[(i)] $L^\pm_p(0)=L^{\mp}_p(\pi),\forall p\in C_W=C_{-W}$.
\item[(ii)] $L^\pm_p(\frac{-\pi}{2})$ are the Lefschetz thimbles with respect to the following gradient flow equation:
\begin{equation}
 \frac{\pat \psi(s)}{\pat s}=\nabla_g (\im(W)).
\end{equation}
\end{itemize}
\end{lm}

\begin{lm}
Let $p\in C_W$ be a critical point. If $q\in L^-_p(q\in L^+_p)$,  then $\im(W(q))\equiv \im(W(p))$ and $\re(W(q))\ge \re(W(p))$ (or $\re(W(q))\le \re(W(p))$).
\end{lm}

\begin{proof}By Lemma \ref{sec2:lm-1}, $\psi(s)$ satisfies (\ref{sec2:equa-4}) if and only if $\psi(s)$ satisfies
\begin{equation}
\frac{\pat \psi(s)}{\pat s}=-X_{\re(iW)},
\end{equation}
which is equivalent to the Witten equation (if we write $\psi(s)=(u(s),\bar{u}(s))$) :
\begin{equation}
\frac{\pat u^l(s)}{\pat s}=\sum_{\ib} h^{\ib l}\pat_{\ib}\overline{W}.
\end{equation}
Take the hermitian innder product, then we obtain
$$
\frac{\pat \overline{W}}{\pat s}=\left|\frac{\pat u}{\pat s}\right|^2.
$$
Assume that $q\in L^-_p$, then we can integrate $t$ over $(-\infty,0]$ to get
\begin{equation}
\overline{W}(q)-\overline{W}(p)=\int^0_{-\infty}\left|\frac{\pat u}{\pat s}\right|^2(\psi_q(s)) ds
\end{equation}
which is equivalent to
\begin{align}
&\re(W)(q)-\re(W)(p)=\int^0_{-\infty} \left|\frac{\pat u}{\pat s}\right|^2 ds,\\
& \im(W(q))\equiv \im(W(p)).
\end{align}
Similarly, one can get the conclusion for $q\in L^+_p$.
\end{proof}

\begin{thm}\cite[Chapter 4]{BH} For generic metric $g$ or generic potential function $W$ in $C^2$-topology, the Lefschetz thimbles $L^\pm_p(\theta)$ are $n$-dimensional smooth submanifolds.
\end{thm}

\begin{lm}The Lefschetz thimbles $L^\pm(\theta)$ with respect to $e^{i\theta}W$ are Lagrangian submanifolds in the symplectic manifold $(M,\omega)$ underlying the K\"ahler manifold $(M,h)$.
\end{lm}

\begin{proof}It suffices to prove $\theta=0$ case for $L^-$. The Lefschetz thimbles $L^-$ are the integral submanifold of the flow $\psi(s)$ generated by:
$$
\frac{\pat \psi(s)}{\pat s}=\nabla_g \re(W)(\psi(s)).
$$
which is equivalent to the Hamiltonian system
\begin{equation}
\frac{\pat \psi(s)}{\pat s}=X_{\im(W)}
\end{equation}
Hence $\psi$ is a Hamiltonian diffeomorphism satisfies $\psi^*\om=\om$.

Let $q\in L^-_p $, $X,Y\in T_q L_p^-$. Since $p$ is the zero point of $X_{\im(W)}$, we have $((\psi_s)_*X)(p)=0$ as $s\to -\infty$. Furthermore, we have
$$
\om(X,Y)(q)=\psi_s^*\om(X,Y)(q)=\om((\psi_{s})_*(X(q)),(\psi_{s})_*(Y(q)))=0,
$$
as $s\to -\infty$.
\end{proof}

\begin{lm}\label{w:simply connected}
$L^{\pm}_p$ are orientable simply connected smooth manifolds.
\end{lm}

\begin{proof}
From the theory in dynamical systems (ref. \cite[Chapter 4]{BH}),   It is know that the smooth manifolds $L^{\pm}_p$ are diffeomorphic to the stable  or unstable subspaces of the tangent space $T_p M$ which can be identified with $\R^n$. Hence $L^{\pm}_p$ are orientable and simply connected.
\end{proof}

The following fact of $L^{\pm}_p$ is frequently used later:

\begin{lm}\label{w: lef-control-1}
Assume that $W$  is a regular Morse function. Let $q\in L^{\mp}_p$, then $\re(W(q))\to \pm\infty$ as $d(p,q) \to +\infty$, and  $d(p,q) \to 0$ as $\re(W(q))\to \re(W(p))$.
\end{lm}
\begin{proof} It suffices to treat the $L^-_p$ case. According to the assumption, $p$ is bounded away from other critical points. For each small neiborhood  $U$ of  $p$ in $L^-_p$, we have a positive number $a$, such that
$$|\nabla \re(W)|^2(q) >a$$
for any $q\in L^-_p\setminus U$.
As
$$
\frac{d \re (W(\psi_s(q)))}{ds}=g( \dot{\psi}_s(q),\nabla_g \re(W))=|\nabla_g \re(W)|^2,
$$
if  $\psi_s(q)\in L^-_p\setminus U $, then
$$
\frac{d\re(W(\psi_s(q)))}{dt}>a.
$$
Let $s_0=\sup\{s\in \R^+|\psi_{-t}(q) \text{ lies outside }U \text { for } 0\leq t\leq s\}$. We have $\psi_{-s_0}(q)\in \partial U$, and
$$
s_0\le \frac{1}{a}(\re(W(q))-\re(W(\psi_{-s_0}(q)))< \frac{1}{a} (\re(W(q))-\re(W(p)).
$$
Thus
\begin{align*}
&d(\psi_{-s_0}(q),q)\le \int^0_{-s_0} |\dot{\psi}_s(q)|ds=\int^0_{-s_0}|\nabla_g \re(W)|ds \le  \sqrt{s_0}(\int^0_{-s_0}|\nabla_g \re(W)|^2ds)^{1/2}\\
&=\sqrt{s_0}(\int^0_{-s_0}g(\nabla_g \re(W),\dot{\psi_s}(q))ds))^{1/2}\le \sqrt{\frac{1}{a}}(\re(W(q))-\re(W(p))^{\frac{3}{2}}.
\end{align*}
This gives a bound for $q$ in terms of $\re(W)$. Thus the first assertion is proved. As we can choose $U$ to have arbitrarily small radius, the second assertion is proved.
\end{proof}

If $W$ is a regular Morse function, we can assign an order to the set $C_W$ of critical points in the following way. If $\im(W)(p)<\im(W)(p')$, then we say $p>p'$. According to this order, we can number the critical points by $p_i$ such that $p_1>\cdots >p_m$, where $m:=\#C_W$. Hence we can number the critical values of $W$ by defining $\alpha_i=W(p_i)$, and the positive Lefschetz thimbles of $W$ by $L^+_{i}=L^+_{p_i}$.  Since the negative Lefschetz thimbles of $W$ are the positive Lefschetz thimbles of $e^{i\pi}W$, this induces the numbering of the negative Lefschetz thimbles of $W$ by identifying $L^-_i$ with the Lefschetz thimble $L^+_{m+1-i}$ of $-W$.

\subsubsection*{Transverse intersection}\

Now we consider the following equations

\begin{equation}\label{sec2:equa-nega-flow-1}
 \frac{\pat \phi(t)}{\pat t}=-\mathsf{k}\nabla_g \im(W)(\phi)=\mathsf{k} X_{\re(W)}(\phi),
\end{equation}
for $\mathsf{k}=1,2,3\cdots$.  We can regard $\mathsf{k}$ as some kind of ``speed'' here, and we call it the speed-$k$ equation to indicate such an index.

For each $\mathsf{k}$, we define the time-T diffeomorphism  $\phi_{T,k}: M\to M$ given by $q\mapsto \phi_{T,k}(q)$, the flow of equation \ref{sec2:equa-nega-flow-1}. For two critical points $p_0,p_1$ with $\im(W(p_0))>\im(W(p_1))$, denote by $L_0=L^+_{p_0}, L_1=L^+_{p_1}$ and $\Delta_{01}:=\im(W(p_0))-\im(W(p_1))$.

\begin{df}\label{trans-dfn}We say that $L_0$ \textit{intersects transversely with $L_1$ at time $T$ with speed $\mathsf{k}$} if and only if the Lagrangian submanifold $\phi_{T,k}(L_0)$ intersects $L_1$ transversely, which is written as  $L_0 \pitchfork^{T,\mathsf{k}} L_1$. If $L_0$ intersects transversely with $L_1$ at time 1 with speed $\mathsf{k}$, then we say that \textit{$L_0$ intersects transversely with $L_1$ with speed $\mathsf{k}$}.
\end{df}

By reparameterization, it is clear that $L_0 \pitchfork^{T,k} L_1$ is equivalent to $L_0 \pitchfork^{kT,1} L_1$.

To state and prove the transverse intersection theorem, we need some preparation.

Let $\varphi\in C^\infty(M)$, then the $C^0$-norm of $\varphi$ is defined as:
$$
||\varphi||_{C^0}:=\sup_{q\in M}|\varphi(q)|.
$$
Take a sequence of positive numbers $\epsilon^i$ which is rapidly decreasing to $0$ as $i\to \infty$ (for example, we can choose $\epsilon_k=(\frac{1}{2})^{k^2}$). We can define a norm of $\varphi\in C^\infty(M)$ as
\begin{equation}\label{metric-norm}
||\varphi||_\epsilon:=\sum^\infty_{k=0}\epsilon^k ||\nabla_g^k \varphi||_{C^0}.
\end{equation}
Define the Floer-$\epsilon$ space
$$
C^\infty_\epsilon(M)=\{\varphi\in C^\infty(M)|\; ||\varphi||_\epsilon<\infty\},
$$
which is a Banach space.

Let $C^\infty_0(M)$ be the space of smooth functions on $M$ with compact support, and define a closed subspace of $C^\infty_\epsilon(M)$ as below
$$
C^\infty_{c,\epsilon}(M)=\overline{C^\infty_0(M)\cap C^\infty_\epsilon(M)}.
$$
Define
\begin{equation}
\mathcal{H}=\{\varphi\in C^{\infty}_{c,\epsilon}(M):\omega_{\varphi}=\omega+\sqrt{-1}\pat_M\bpat_M\varphi>0\text{ on }M \}.
\end{equation}
This is an open convex subset in the Banach space $C^\infty_{c,\epsilon}(M)$, which can be considered as a Banach manifold, whose tangent space $T_\varphi\mathcal{H}$ at $\varphi$ is $C^\infty_{c,\epsilon}(M)$.

Let $L_0$ and $L_1$ be two Lefschetz thimbles. We define a closed subset $\G^L\subset \mathcal{H}$ with respect to $L_0, L_1$ as follows:
\begin{equation}\label{twothimbles}
\G^L:=\{\varphi \in \mathcal{H}\;|\;|\pat_M\bpat_M \varphi|_{L_0\cup L_1}=0 \}
\end{equation}

For $\mu>0$, define
$$
\G^L_{\mu}=\{\varphi\in\G^L\;|\; ||\pat_M\bpat_M \varphi||_{C^0(M)}<\mu\}.
$$
To simplify the notations, we identify the set $\G^L_\mu$ and the set of K\"ahler forms $\om+\sqrt{-1}\pat_M\bpat_M \G^L_\mu$.

\begin{lm}\label{fini-rad}
There exists a $\mu_0>0$ such that for any $\mu\in (0,\mu_0)$, any $\gt\in \G^L_{\mu}$ and any fixed point $p_0\in M$, there exists a $R>0$ such that for any $T\in (1/2,2),\mathsf{k}\in\mathbb{N}$, any solution of the following speed-$\mathsf{k}$ equation
\begin{equation}
\begin{cases}
 \frac{\pat \phi(t)}{\pat t}=-\mathsf{k}\nabla_{\gt} \im(W)(\phi(t))\\
 \phi(0)\in L_0,\phi(T)\in L_1,
 \end{cases}\label{w:equa-nega-flow-2}
 \end{equation}
lies entirely in $B_{R}(p_0)$.
\end{lm}

\begin{proof} Let $t_0\in [0,T]$. For any $t\in [t_0, T]$, there is
\begin{align*}
d(\phi(t),\phi(t_0))&\le\int_{t_0}^t |\frac{d\phi(s)}{ds}|ds=\int_{t_0}^t |\nabla_{\gt} \im(W)|ds\le \left(\int_{t_0}^t |\nabla_{\gt} \im(W)|^2ds \right)^{1/2}\sqrt{t-t_0}\\
=&\left(-\int_{t_0}^t \frac{1}{\mathsf{k}}\frac{d\im(W(\phi(t)))}{dt}\right)^{1/2}\sqrt{t-t_0}\\
=&\sqrt{\frac{t-t_0}{\mathsf{k}} \big(\im(W(\phi(t_0)))-\im(W(\phi(t)))\big)}\\
\le & \sqrt{\frac{1}{\mathsf{k}}T\Delta_{01}}.
\end{align*}
The second equality is by (\ref{w:equa-nega-flow-2}).
This implies that for any $t,t_0\in [0,T]$, there is
\begin{equation}
d(\phi(t),\phi(t_0))\le \sqrt{\frac{1}{k}T\Delta_{01}}\le \sqrt{T\Delta_{01}}\le 2\sqrt{\Delta_{01}}\label{w:equa-nega-flow-3}.
\end{equation}
By (ii) of Definition \ref{sec:df-tame-1}, we know that $|\pat_M W|_g\to \infty$ as $d(\phi,q_0)\to \infty$. Hence For fixed point $p_0\in M$, we can choose $R_0$ such that $|\nabla_{g} \im(W(\phi))|_g=|\pat_M W(\phi)|_g>4\sqrt{\Delta_{01}}$ for any point $\phi\not\in B_{R_0}(p_0)$. Furthermore, we can choose small $\mu_0>0$ such that for any $\mu\in (0,\mu_0)$, any metric $\gt\in \G^L_{\mu}$, there is
\begin{equation}
|\nabla_{\gt} \im(W)|_{\gt}>2\sqrt{\Delta_{01}}, \forall \phi\not\in B_{R_0}(p_0).
\end{equation}

Let $R=R_0+2\sqrt{\Delta_{01}}$. we claim that the solution of (\ref{w:equa-nega-flow-2}) lies entirely in $B_R(p_0)$.

If it is not true, then there would exist a $t_0\in [0,T]$ such that $\phi(t_0)\not\in B_R(p_0)$. By (\ref{w:equa-nega-flow-3}), we know that $\phi(t)\not\in B_{R_0}(p_0)$ for any $t\in [0,T]$, then by our choice of $R_0$ and the fact that $T>1/2$, we have
\begin{align*}
\Delta_{01}=&\im(W(\phi(0)))-\im(W(\phi(T)))=-\int_0^T \gt(\nabla_{\gt} \im(W),\frac{d\phi}{dt})dt\\
=&\mathsf{k}\int_0^T|\nabla_{\gt} \im(W)|_{\gt}^2dt\\
\ge&2\mathsf{k}\cdot T\Delta_{01}>k\Delta_{01}\geq \Delta_{01},
\end{align*}
which is a contradiction.
\end{proof}
\begin{df} 
Denote $\widetilde{L}_1:= L_1 \cap B_{R}(p_0)$ and $\widetilde{L}_0:= L_0 \cap B_{R}(p_0)$
\end{df}

We write $\phi_{T,\mathsf{k}}^{\gt}$ instead of $\phi_{T,\mathsf{\mathsf{k}}}$ to display the dependence on the metric $\gt$. Fix the $\mu=\mu_0$ now.

\begin{thm}\label{transversal} Let $(M, h, W)$ be a tame LG system. For generic metric $\gt\in \G^L_{\mu}$ and generic $T\in (1/2,2)$, we have $L_0$ and $L_1$ intersect transversally at time $T$ with speed $\mathsf{k}$ for each $\mathsf{k}$.
\end{thm}

\begin{proof} For each $\mathsf{k}$, define $F_{\mathsf{k}}:\widetilde{L}_0\times \G^L_{\mu}\times  (1/2,2)\mapsto M$ by
$$
F_{\mathsf{k}}(x,\gt,T)=\phi_{T,\mathsf{k}}^{\gt}(x).
$$
The space $\widetilde{L}_0\times \G^L_{\mu}\times  (1/2,2)$ is a $C^l$-Banach manifold for any $l$.  However, $F_{\mathsf{k}}$ is not defined for all points in this space as the flow $\phi_{T,\mathsf{k}}^{\gt}$ can go to infinity at some place. Take $A_{F_{\mathsf{k}}}$ to be the open subset in $\widetilde{L}_0\times \G^L_{\mu}\times  (1/2,2)$ so that $F_{\mathsf{k}}$ is well defined.
It is easy to see that for any  $(\gt,T)\in \G^L_{\mu}\times  (1/2,2)$, $\phi_T^{\gt}(x)=F(x,\gt,T)$ is a Fredholm map with finite index, near any point $x_0$ so that $(x,\gt,T)\in A_{F_{\mathsf{k}}}$. Intersection happens only at $F_{\mathsf{k}}^{-1}(\widetilde{L}_1)$.  If we can prove that $F_{\mathsf{k}}(A_{F_{\mathsf{k}}})$ intersects $\widetilde{L}_1$ transversally, then by Sard-Smale theorem (ref. \cite[Theorem 1.3.18]{Ck}), there are residual sets $R_{\mathsf{k}}\subset \G^L_{\mu}\times  (1/2,2)$ such that for any $(\gt,T)\in R_{\mathsf{k}}$, the map  $\phi_{T,\mathsf{k}}^{\gt}$ intersects $\widetilde{L}_1$ transversely. The intersection $\cap R_{\mathsf{k}}$ is also a residual set.

Therefore, we need only to prove that $F_{\mathsf{k}}(\widetilde{L}_0\times \G^L_{\mu}\times  (1/2,2)$ intersects $\widetilde{L}_1$ transversally.
With out loss of generality we let $k=1$ and write $F_1=F$ for simplicity.
Let $(x,\gt,T)\in F^{-1}(\widetilde{L}_1)$ and denote by $y=F(x,\gt,T)$. For any $2n-2$ vectors $\kappa_1,\cdots,\kappa_{2n-2}\in T_{\gt}\G^L_{\mu}$, let $ \eta_i=DF(x,\gt,T)\cdot \kappa_i, i=1,\cdots,2n-2$. Note that $\nabla_{\gt} \im(W(y))=-DF(x,\gt,T)\cdot \frac{\pat}{\pat t}$ and $\nabla_{\gt} \re (W(y))\in T_y\widetilde{L}_1$. If we can find $\kappa_i$ such that the tuple $\{\kappa_1,\cdots,\kappa_{2n-2}, \nabla_{\gt} \im(W(y)), \nabla_{\gt} \re (W(y))\}$ forms a basis of $T_yM$, then we are done.

For any $\kappa\in T_{\gt}\G^L_{\mu}$, write $\gt_{\lambda}$=$\gt+
\lambda \kappa$. Since there exists some $\delta_0>0$ such that $\phi_t^{\gt}(x)$ lies in a coordinate neighborhood of $y$ for any $t\in[T-\delta_0,T]$, we need only consider the flow equation and the choice of perturbation $\kappa$ in this coordinate neighborhood. The strategy of the proof is to perturb the K\"ahler potential in this coordinate neighborhood such that the arriving flow lines are transversal to $L_1$.

So without loss of generality, we consider the following flow equation in a small time interval $[0, T]$ and in the Euclidean space:
\begin{equation}\label{w:integration-1}
 \partial_t \Gamma(\lambda,t)=-\nabla_{\gt_{\lambda}}\im(W(\Gamma(\lambda,t))),
 \end{equation}
where $\Gamma(\lambda,t)=\phi_t^{h_{\lambda}}({x})\in \R^{2n}$.

For simplicity, we denote by $G=(\gt_{i\jb}), K=(\kappa_{i\jb})$ and $\nabla_0$ the gradient in Euclidean space. We have
\begin{equation}
\nabla_{\gt_{\lambda}}\im(W(\Gamma))=(G+\lambda K)^{-1}\nabla_0 \im(W(\Gamma))=(1-\lambda G^{-1}K)\nabla_{\gt}\im(W(\Gamma))+o(\lambda)
\end{equation}

Let $\nu(t)=\partial_{\lambda}\Gamma(0,t)$. Then by (\ref{w:integration-1}) it satisfies the following equation:
\begin{equation}\label{w:integration-2}
\frac{d\nu(t)}{dt}=-D\nabla_{\gt}\im(W(\phi_t^h({x})))  \cdot \nu(t)+G^{-1}(\phi_t^{\gt}({x}))K(\phi_t^{\gt}({x}))\nabla_{\gt}\im(W(\phi_t^{\gt}({x}))) .
\end{equation}

Since $\Gamma(\lambda,0)\equiv x$, $\nu(0)=0$, the equation (\ref{w:integration-2}) has the solution
\begin{equation}\label{w:integration-3}
\nu(t)=X(t)\int_0^t X(s)^{-1} G^{-1}(\phi_s^{\gt}({x}))K(\phi_s^{\gt}({x}))\nabla_{\gt}\im(W(\phi_s^{\gt}({x}))) ds.
\end{equation}
where $X(t)$ is the fundamental solution matrix of the corresponding homogeneous equation
$$
\begin{cases}
\frac{dX(t)}{dt}=-D\nabla_{\gt}\im(W(\phi_t^{\gt}(x)))  \cdot X(t)\\
X(T)=I
\end{cases}.
$$

Introduce an auxiliary smooth functions $\rho_1$ on $\R$ such that
$$
\begin{cases}
\rho_1(s)=0&  s<-1\\
0\leq \rho_1(s) \leq 1 & -1 \leq s \leq -\frac{1}{2}\\
\rho_1(s)=1 &-\frac{1}{2}\leq s \leq 0 \\
\rho_1(s)=\rho_1(-s) & s>0,
\end{cases}
$$

Define $f^{\epsilon_1}(x_1,y_1,\cdots,x_n,y_n)=y_1x_2\rho_1( x_1/\epsilon_1)\rho_1( y_1/\epsilon_1)\dots\rho_1(x_n/\epsilon_1)\rho_1(y_n/\epsilon_1)$ where $\epsilon_1$ is a positive number. Define the perturbation metrc $\kappa=\sum_{ij}f^{\epsilon_1}_{i\jb}dz_i\otimes d\bar{z_j}$, which vanishes outside $[-\epsilon_1,\epsilon_1]^{2n}$.

Let $\widetilde{y}=\phi_{T/2}^{\gt}({x})$. For any vector $\xi\in T_{\widetilde{y}}\R^{2n}$, which is perpendicular to $\nabla_{\gt}\im(W(\widetilde{y}))$ \and $J \nabla_{\gt}\im(W(\widetilde{y}))=-\nabla_{\gt}\re(W(\widetilde{y}))$, we can use a complex affine transform to make $\widetilde{y}=0,G(\widetilde{y})=I,\nabla_{\gt} \im(W(\widetilde{y}))=(2,0,\cdots,0)^T$ and $\xi=(0,0,1,0,\cdots,0)^T$.

Consider the component expression $\phi^{\gt}_t(\widetilde{y})=(x_1(t),y_1(t),\cdots,x_n(t),y_n(t))^T$, we have $\dot{x}_1(0)=-2$, and $\dot{x}_i(0)=0$, $\dot{y}_j(0)=0$ for $i=2,\cdots,n; j=1,\cdots,n$.

Choose an interval $[-a,a]$ such that $\dot{x}_1(t)<-3/2$ and $|\dot{x}_i(t)|<1/2$, $|\dot{y}_j(t)|<1/2$ for $i=2,\cdots,n; j=1,\cdots,n$.  As the  curve $\phi^{\gt}_t(\widetilde{y})$ has no-self intersection, we can choose $\epsilon_1$ small enough, such that $\phi^\gt_t(\widetilde{y})$ lies in the box $[-\epsilon_1,\epsilon_1]^{2n}$ only when $t\in [-a,a]$.  Note that $\phi^{\gt}_t(\widetilde{y})$ lies in the box only when $|t|<2\epsilon_1/3$ and at these points we have $f^{\epsilon_1}=\rho_1(x_1/\epsilon_1)y_1x_2$. So we have
$$
\sum_{ij}f^{\epsilon_1}_{i\jb}dz_i\otimes d\bar{z_j}=\frac{1}{4}(\dot{\rho_1}y_1/\epsilon_1-\sqrt{-1}\rho_1)dz_1\otimes d\bar{z_2}+\frac{1}{4}(\dot{\rho_1}y_1/\epsilon_1+\sqrt{-1}\rho_1)dz_2\otimes d\bar{z_1},
$$
and the corresponding Riemannian metric matrix
$$
K(\phi^\gt_t(\widetilde{y}))=\frac{1}{4}
\left[\begin{array}{cc}
\left[\begin{array}{cccc}
   0 &0  & \dot{\rho}_1y_1(t)/\epsilon_1& -\rho_1\\
   0 & 0 &\rho_1&\dot{\rho}_1y_1(t)/\epsilon_1 \\
 \dot{\rho_1}y_1(t)/\epsilon_1&-\rho_1&0 &0\\
  \rho_1&\dot{\rho_1}y_1(t)/\epsilon_1&0&0
  \end{array}\right]& 0\\
0&0
 \end{array}\right]
$$
By (\ref{w:integration-3}) and the fact $\phi^\gt_{t+T/2}(x)=\phi^\gt_t(\widetilde{y})$, $G(\widetilde{y})=I$, $\nabla_{\gt}\im(W(\widetilde{y}))=(2,0,\cdots,0)^T$, we have
\begin{equation}
\nu(T)=X(T)\int_{-\frac{2\epsilon_1}{3}}^{\frac{2\epsilon_1}{3}}X^{-1}(t+\frac{T}{2})G^{-1}(\phi^\gt_t(\widetilde{y}))K(\phi^\gt_t(\widetilde{y}))\nabla_{\gt}\im(W(\phi^\gt_t(\widetilde{y})))dt.\\
\end{equation}

Since $\dot{y}_1(0)=y_1(0)=0$, the intergration on $[-\frac{2\epsilon_1}{3},0]$ equals
\begin{align*}
&\frac{1}{2}\int_{-\frac{2\epsilon_1}{3}}^0(X^{-1}(T/2)+o(\epsilon_1))\left((0,0,\dot{\rho_1}/\epsilon_1\times O(\epsilon_1^2),\rho_1,0,\cdots,0)^T+o(\epsilon_1)\right)dt\\
=&\frac{1}{2}(X^{-1}(T/2)+o(\epsilon_1)) (0,0,a\epsilon_1^2,b\epsilon_1,0,\cdots,0)^T,\;\text{as}\;\epsilon_1\to 0,
\end{align*}
where $b>0$.

We have the similar estimate for the integral over $[0,\frac{2\epsilon_1}{3}]$. Hence we have the estimate as $\epsilon_1\to 0$:
\begin{equation}\label{sec5:inte-5}
 \nu(T)=C(\epsilon_1+o(\epsilon_1))X(T)X^{-1}(T/2)J\xi
 \end{equation}
 for some constant $C>0$. Note that we take $\xi=(0,0,1,0,\cdots,0)^T$.

Choose vectors  $\xi_1,\cdots,\xi_{2n-2}$ such that
\begin{equation} \label{w:basis}
\{ \nabla_{\gt}\im(W(\widetilde{y})), \nabla_{\gt}\re(W(\widetilde{y}),\xi_1,\cdots,\xi_{2n-2}) \}
\end{equation}
is an orthonomal basis of $T_{\widetilde{y}}\R^{2n}$. Replace $\xi$ in (\ref{sec5:inte-5}) by $\xi_i, i=1,\cdots, 2n-2$, we get $2n-2$ perturbed matrices $
\kappa_1,\cdots,\kappa_{2n-2}$ in $T_\gt\G^L_{\mu}$. The derivative of $\phi_T^{\gt}(x)$ along $\kappa_i$ are
$$
C(\epsilon_1+o(\epsilon_1))X(T)X^{-1}(T/2)J\xi_1,\cdots,C(\epsilon_1+o(\epsilon_1))X(T)X^{-1}(T/2)J\xi_{i}.
$$.
Notice that $J\xi_{2k-1}=\xi_{2k}$, $k=1,\cdots,n-1$, so these $\kappa_i$ will satisfy our requirement if
\begin{align*}
&\nabla_{\gt}\im(W(y)), \nabla_{\gt}\re(W(y)), C(\epsilon_1+o(\epsilon_1))X(T)X^{-1}(T/2)J\xi_1,\\
&\cdots,C(\epsilon+o(\epsilon))X(T)X^{-1}(T/2)J\xi_{2n-2}
\end{align*}
are linearly independent. This can be done if we take smaller $\epsilon_1$ and $T$.

Now we finished the whole proof of this theorem.
\end{proof}

From now on, we always make the assumption that, for each $\mathsf{k}\in\mathbb{N}$, each pair $(L_0,L_1)$ of Lefschetz thimbles intersects transversally at time 1. This means that the equation
\begin{equation}\label{soliton}
\begin{cases}
\frac{\pat \phi(t)}{\pat t}=-\mathsf{k}\nabla_g (\im(W)).\\
\phi(0)\in L_0,\;\phi(1)\in L_1.
\end{cases}
\end{equation}
or equivalently the Witten equation:
\begin{equation}\label{solition-c}
\begin{cases}
i\cdot\frac{\pat u^l(t)}{\pat t}=\mathsf{k}\sum_{\ib} h^{\ib l}\pat_{\ib}\overline{W}.\\
(u(0),\overline{u(0)})\in L_0,\;(u(1),\overline{u(1)})\in L_1
\end{cases}
\end{equation}
have finitely many solutions. Denote the solution set by $S_{\mathsf{k},W}(L_0,L_1)$.
\begin{df}\label{chord-space-df}
Define $S_W(L_0,L_1)$ to be $\bigcup_{\mathsf{k}} S_{\mathsf{k},W}(L_0,L_1)$. And define $S_W$ to be the union of all these $S_W(L_0,L_1)$, as  $(L_0,L_1)$ runs over all possible pairs of Lefschetz thimbles.
\end{df}

\section{Witten equation and action functional: \uppercase\expandafter{\romannumeral1}}\label{section-weq-act}

\subsubsection*{Perturbed Witten equation}\

In this section we fix two Lefschetz thimbles $L_0,L_1$ as well as the speed $\mathsf{k}$. We define the space
\begin{equation}\label{sec-4-pert-defi-1}
\mathcal{V}_{per}=\{\varpi\in C^{\infty}([0,1]\times M)\;|\;\nabla_g\varpi=0 \text{ on } \bigcup_{l\in S_W(L_0,L_1)}[0,1]\times \text{Im}(l)\},
\end{equation}
with the norms
$$
\|\varpi\|_{2}:=\sup_{x,t}|\nabla_g \varpi(t,x)|+\sup_{x,t}|\nabla_g^2 \varpi(t,x)|,\;||\varpi||_{\epsilon}=\sum_{k=0}^{\infty} \epsilon_k \sup_{(x,t)\in M\times [0,1]} |\nabla^k \varpi(t,x)|,
$$
where $\{\epsilon_k\}$ is a sequence of positive real numbers converging rapidly to zero so that $\|\;\|_{\epsilon}$ defines a norm on the complete Floer-$\epsilon$ space. For example, we can choose $\epsilon_k=(\frac{1}{2})^{k^2}$.

For any $\varpi\in \mathcal{V}_{per}$, we can consider the perturbed Hamiltonian equation

\begin{equation}\label{perturb-hal}
J\pat_t \phi-\nabla_g(\mathsf{k}\re W(\phi)+\varpi(t,\phi))=0,
\end{equation}

In local coordinates, if we set $u^j=\phi^j+i\phi^{j+n}$ for $j=1,\cdots,n$, then it is easy to prove that
$\bpat_z u=(I+iJ)\bpat_J\phi$. Let
$$
\nabla_g \varpi=\frac{1}{2}(I-iJ)\nabla_g \varpi+\frac{1}{2}(I+iJ)\nabla_g \varpi=\sum_j (\mu^j \frac{\pat}{\pat z^j}+\mu^\jb\frac{\pat}{\pat z^\jb}).
$$

Denote by $S_{\mathsf{k},W,\varpi}(L_0,L_1)$ the set of solutions $\phi$ satisfying (\ref{perturb-hal})  and the boundary condition $\phi(0)\in L_0, \phi(1)\in L_1$.

For simplicity, we write $\widetilde{H}_{\varpi}=\mathsf{k}\re{W}+\varpi$.

Similar to the proof of Lemma \ref{fini-rad}, we have

\begin{lm}\label{perturb-fini-rad}
There exists a constant $\delta_0>0$ such that if $\|\varpi\|_{2}<\delta_0$, then any solution in $S_{\mathsf{k},W,\varpi}$ lies entirely in $B_{R}(p_0)$ for any fixed $p_0$ and some $R>0$ (depending on $p_0$).
\end{lm}

\begin{thm}\label{perturb-hal-equi}
There exists a constant $\delta_0>0$ such that for any $\varpi\in \mathcal{V}_{per}$ with $\|\varpi\|_{2}<\delta_0$, $S_{\mathsf{k},W,\varpi}(L_0,L_1)=S_{\mathsf{k},W}(L_0,L_1)$.
\end{thm}

\begin{proof} By the definition of $\mathcal{V}_{per}$, we have $S_{\mathsf{k},W,\varpi}(L_0,L_1)\supset S_{\mathsf{k},W}(L_0,L_1)$.

The  equation (\ref{perturb-hal}) gives a time-1 diffeomorphism $\phi_{1,\varpi}$ and the diffeomorphism $\phi_{1,0}$ is just $\phi_1$, which is given by equation (\ref{hal}).  For a given path $l\in S_{\mathsf{k},W}(L_0,L_1)$, we can choose coordinate neighborhoods $U_0$ of $l(0)$ in $L_0$ and $U_1$ of $l(1)$ in $M$, such that $\phi_1(U_0)\subset U_1$. Furthermore, by coordinate transformation, we can require that $l(0)$ is the origin of $\R^n$, $l(1)$ being the origin of $\R^{2n}$, $L_1\cap U_1\subset \R^n\times \{0\}$ and $d\phi_1(0)=\begin{pmatrix}0 \\I_n\end{pmatrix}$.

Note that the variation field equation of (\ref{perturb-hal}) is given as follows
\begin{equation}
\begin{cases}
\dot{v}(t,x)=DX_{\widetilde{H}_{\varpi}}(\phi(t,x))\cdot v(t,x)\\
v(0,x)\in T_x L_0,
\end{cases}
\end{equation}
where $v(t,x)\in T_{\phi(t,x)}M$.

Hence $d\phi_{1,\varpi}(x) v(0)=v(1,x)$. This shows that the solution $v(1,x)$ depends continuously on the norm $||\varpi||_2$ and the variable $x\in L_0$. We can take a small $\delta_0>0$ and a smaller neighborhood $U'_0$ such that for any $\varpi$ satisfying $||\varphi||_2<\delta_0$, and $x\in U'_0$, the following holds
$$
\|d\phi_{1,\varpi}(x)-d\phi_1(0)\|<1/3.
$$
Here we take the matrix norm.

This implies that $\phi_{1,\varpi}: U'_0\to U_1$ is an embedding and  intersects $L_1$ only at $0$.

Now for any $l\in S_{\mathsf{k},W}(L_0,L_1)$, we can obtain a neighborhood $U'_0(l)$ such that $\phi_{1,\varpi}(U'_0(l))$ intersects $L_1$ transversely if $||\varpi||_2$ is small enough. Since
$$
\phi_1(L_0\cap (B_{R}(p_0)\setminus \bigcup_{l\in S_{W}}\overline{U'_0(l)}))\cap L_1=\emptyset,
$$
for $\varphi\in \mathcal{V}_{per}$ with sufficently small $\|\nabla_g\varpi\|_{C^0}$, we still have
$$
\phi_{1,\varpi}(L_0\cap (B_{R}(p_0)\setminus \bigcup_{l\in S_{W}}\overline{U'_0(l)}))\cap L_1=\emptyset.
$$
Hence, we proved $S_{\mathsf{k},W,\varpi}(L_0,L_1)=S_{\mathsf{k},W}(L_0,L_1)$ if $||\varpi||_2$ is sufficiently small.
\end{proof}

\begin{df}\label{perturb-space-stripe}
We define
$$
\|\varpi\|_{ratio}=\sup_{t,\;x}\frac{|\nabla_g \varpi(t,x)|}{|\nabla_g \re (W)|},
$$

and define

$$
\mathcal{V}^{\epsilon,\delta_0}_{per}=\{\varpi\in \mathcal{V}^{\epsilon}\;|\;||\varpi||_{2}<\delta_0, ||\varpi||_{ratio}<\delta_0\}.
$$

From now on, we will assume $\delta_0$ to be small enough, so that the results in Lemma \ref{perturb-fini-rad} and Theorem \ref{perturb-hal-equi} hold for $\varpi\in \mathcal{V}^{\epsilon,\delta_0}_{per}$.
\end{df}
\subsubsection*{Path space and actional functional}\

Let $p_{0},p_{1}\in C_W$. We have the order $p_{0}>p_1$ which means that $\im W(p_0)<\im W(p_1)$. We have the corresponding Lefschetz thimbles $L_i:=L^+_{p_i},i=0,1$.

Define the path space
$$
\Omega(L_0,L_1):=\{\phi\in C^\infty([0,1], M)\;|\;\phi(0)\in L_0, \phi(1)\in L_1\}.
$$
We can endow $\Omega(L_0,L_1)$ a (real) $L^2$ metric
$$
\langle \xi,\eta\rangle_g=\int_0^1 g(\xi(t),\eta(t))dt,
$$
for any $l\in \Omega (L_0,L_1)$ and $\xi,\eta \in T_l\Omega(L_0,L_1)$. Similarly, we can define the Hermitian metric $\langle\;,\;\rangle_h$ by using the K\"ahler metric on $M$.

There is a natural one form (called action one form)  $\Psi_{\widetilde{H}_{\varpi}}: T\Omega(L_0,L_1)\to \R$ defined by
$$
\Psi_{\widetilde{H}_{\varpi}}(l,\xi)=\int^1_0 \om(\dot{l}(t)-X_{\widetilde{H}_{\varpi}}(l(t)),\xi(t))dt.
$$

The zero points of $\Psi_{\widetilde{H}_{\varpi}}$ are the solutions of the equation (\ref{perturb-hal}), simply written as

\begin{equation}\label{perturb-hal-r}
J\pat_t \phi-\nabla_g\widetilde{H}_{\varpi}=0.
\end{equation}

The space $\Omega(L_0,L_1)$ is not connected but may have infinitely many connected components. We can fix one connected component by choosing a base path $l_0\in \Omega(L_0,L_1)$, i.e., we choose $\Omega(L_0,L_1;l_0)$ which consists of all paths $l$ that can be homotopically connected to $l_0$. We consider the universal covering of $\Omega(L_0,L_1;l_0)$. Consider the set of pairs $(l,\Phi)$ where the map $\Phi:[0,1]\times [0,1]\to M$ satisfies the boundary conditions:
\begin{align}
&\Phi(0,t)=l_0(t),\quad \Phi(1,t)=l(t),\forall t\in [0,1]\\
&\Phi(s,0)\in L_0,\quad \Phi(s,1)\in L_1,\forall s\in [0,1].
\end{align}
So $\Phi:s\mapsto \Phi(s,\cdot)\in \Omega(L_0,L_1;l_0)$ provides a path from $l_0$ to $l$. We denote the homotopy class of $\Phi$ by $[l,\Phi]\in \pi_1(\Omega(L_0,L_1;l_0))$. Those $[l,\Phi]$'s form the universal covering of $\Omega(L_0,L_1;l_0)$. However, we want to consider a smaller covering. Let $(l,\Phi),(l,\Phi')$ be two paths connecting $l$ and $l_0$, then the contenation:
$$
\overline{\Phi}\#\Phi':S^1\to \Omega(L_0,L_1;l_0)
$$
defines a based loop $\gamma:S^1\to \Omega(L_0,L_1;l_0)$ which can be regarded as the following map satisfying the boundary condition:
\begin{align}
&C:S^1\times [0,1]\to M\nonumber\\
&C(s,0)\in L_0,\;C(s,1)\in L_1\label{sec2:cond-1}.
\end{align}
It is easy to see that the symplectic area of $C$, denoted by
\begin{equation}
I_\om(C)=\int_C\om
\end{equation}
depends only on the homotopy class of $C$ satisfying the condition (\ref{sec2:cond-1}). Hence we can define the homomorphism:
\begin{equation}
I_\om:\pi_1(\Omega(L_0,L_1;l_0))\to \R.
\end{equation}

Note that for the map $C:S^1\times [0,1]\to M$ satisfying (\ref{sec2:cond-1}), we can associate to it a symplectic bundle pair
$$
E_C=C^*TM,\;F_C=C(\cdot,0)^*TL_0\sqcup C(\cdot,1)^*TL_1.
$$
Hence we can define another homomorphism:
\begin{equation}
I_\mu:\pi_1(\Omega(L_0,L_1;l_0))\to \Z,\;I_\mu(C)=\mu(E_C,F_C),
\end{equation}
where $\mu(E_C,F_C)$ is the Maslov index for the bundle pair $(E_C,F_C)$.

\begin{df}We say that $(l,\Phi)$ is \textit{$\Gamma$-equivalent} to $(l,\Phi')$ if they satisfy
\begin{equation}
I_\om(\overline{\Phi}\#\Phi')=0=I_\mu(\overline{\Phi}\#\Phi').
\end{equation}
We denote the equivalence relation by the $(l,\Phi)\sim (l, \Phi')$ and the set of all equivalence classes by $\tilde{\Omega}(L_0,L_1;l_0)$. We call this the \textit{$\Gamma$-covering space} of $\Omega(L_0,L_1;l_0)$.
\end{df}

We denote by $G:=G(L_0,L_1;l_0)$ the deck transformation group, then it can be shown (see \cite{Oh2}, Proposition 13.4.2) that $G(L_0,L_1;l_0)$ is an abelian group and the homomorphisms $I_\om$ and $I_\mu$ can be factorized through $G(L_0,L_1;l_0)$. Take any commutative ring $R$, we have the group ring $R[G]$.

Consider the formal sum
$$
\sigma=\sum_{g\in G}a_g \cdot g
$$
and set
$$
supp(\sigma)=\{g\in G\;|\;a_g\neq 0\}.
$$
\begin{df}We define the \textit{Novikov ring} associated to the group $G$ by
$$
\Lambda(L_0,L_1;l_0)=\{\sigma\;|\;\#supp(\sigma)\cap I_\om^{-1}((-\infty,a))<\infty\;\text{for all }a\}
$$
with the obvious ring structure.
\end{df}
It is easy to see that $\Lambda(L_0,L_1;l_0)$ has the graded $R$-module structure
$$
\Lambda(L_0,L_1;l_0)=\oplus_{k\in \Z}\Lambda^k(L_0,L_1;l_0),
$$
where $\Lambda^k(L_0,L_1;l_0)$ is the $R$-submodule generated by $g\in G, \mu(g)=k$.

For any pair $(l,\Phi)$, we define the action functional
\begin{equation}
A_{\widetilde{H}_{\varpi}}(l,\Phi)=-\int_{[0,1]\times [0,1]}\Phi^*\om-\int^1_0 \widetilde{H}_{\varpi}(l(t))dt.
\end{equation}
This functional only depends on the homotopy class $[l,\Phi]$ and is defined on $\tilde{\Omega}(L_0,L_1;l_0)$.

\begin{lm}Let $\pi:\tilde{\Omega}(L_0,L_1;l_0)\to \Omega(L_0,L_1;l_0)$ be the $\Gamma$-covering space and $\Psi_{\widetilde{H}_{\varpi}}$ be the action one form on $\Omega(L_0,L_1;l_0)$. The following conclusions hold.
\begin{itemize}
\item[(i)]
$$
dA_{\widetilde{H}_{\varpi}}=\pi^*\Psi_{\widetilde{H}_{\varpi}}.
$$
\item[(ii)] The gradient of $A_{\widetilde{H}_{\varpi}}$ with respect to the metric $\langle\;,\;\rangle_g$ is
$$
grad A_{\widetilde{H}_{\varpi}}(l)(t)=J(l(t))\cdot \frac{\pat l(t)}{\pat t}-\nabla_g (\widetilde{H}_{\varpi})(l(t)).
$$
\end{itemize}
\end{lm}

\begin{proof}In order to prove (i), we can regard $A_{\widetilde{H}_{\varpi}}$ as a circle valued functional defined on $\Omega(L_0,L_1;l_0)$ and then prove that $dA_{\widetilde{H}_{\varpi}}=\Psi_{\widetilde{H}_{\varpi}}$. Let $\eta\in T_l(\Omega(L_0,L_1;l_0))$ be the variational vector field on $l(t)$, i.e., there is a variation $l_\lambda(t)$ of $l(t)$ for $\lambda \in (-\epsilon,\epsilon)$, where $\epsilon > 0$ is sufficiently small such that $l_{\lambda=0}=l$ and $\pat l_\lambda(t)/\pat \lambda|_{\lambda=0}=\eta(t).$ Given the pair $(l,\Phi)$, we can construct the pair $(l_\lambda, \Phi_\lambda)$ where $\Phi_\lambda (\lambda\in [0,\epsilon))$ is defined as follows:
\begin{equation}
\Phi_\lambda(s,t)=
\begin{cases}
\Phi(\frac{s}{\lambda},t),&0\le s\le \lambda\\
l_{s-\lambda}(t),&\lambda\le s\le 2\lambda\\
l_\lambda(t),&2\lambda\le s\le 1.
\end{cases}
\end{equation}
Now we have
$$
A_{\widetilde{H}_{\varpi}}([l_\lambda,\Phi_\lambda])=-\int_{[0,1]\times [0,1]}\Phi_\lambda^*\om-\int^1_0 (\widetilde{H}_{\varpi})(l_\lambda(t))dt.
$$
Take the derivative of $\lambda$ at $\lambda=0$, we have
\begin{align*}
\delta A_{\widetilde{H}_{\varpi}}(l)\cdot \eta&=-\int^1_0 \frac{d}{d\lambda}|_{\lambda=0}\int^{2\lambda}_{\lambda} \om(\frac{\pat \Phi_\lambda}{\pat s},\frac{\pat \Phi_\lambda}{\pat t})dsdt-\int^1_0 d(\widetilde{H}_{\varpi})(l(t))(\eta)dt\\
&=\int^1_0 \om(\frac{\pat l}{\pat t},\eta(t))dt-\int^1_0 \om(X_{\widetilde{H}_{\varpi}},\eta)dt\\
&=\Psi_{\widetilde{H}_{\varpi}}(l,\eta).
\end{align*}

The above identity shows that
$$
\langle grad A_{\widetilde{H}_{\varpi}},\eta\rangle_g=\langle J\cdot \frac{\pat l}{\pat t}-\nabla(\widetilde{H}_{\varpi}),\eta\rangle_g.
$$
This gives (ii).
\end{proof}

The $L^2$ negative gradient flow equation of $A_{\widetilde{H}_{\varpi}}$ on $\Omega(L_0,L_1;l_0)$ has the following form:
\begin{equation}\label{equa-flow-real-1}
\frac{\pat \phi}{\pat s}+J(\phi)\frac{\pat \phi}{\pat t}-\nabla_g \widetilde{H}_{\varpi}=0,
\end{equation}
where $\phi(s,t):\Sigma:=\R\times [0,1]\to M$ satisfies the boundary conditions:
\begin{equation}\label{sec2:equa-flow-bdry-1}
\phi(s,0)\in L_0,\;\phi(s,1)\in L_1.
\end{equation}

The equation (\ref{equa-flow-real-1}) is Floer's equation, which is just the perturbed Witten equation
\begin{equation}\label{perturb-lg}
\pat_s \phi+J\pat_t \phi-\nabla_g(\mathsf{k}\re W(\phi)+\varpi(t,\phi)) =0,
\end{equation}
or simply written as
\begin{equation}\label{perturb-lg-r}
\pat_s \phi+J\pat_t \phi-\nabla_g \widetilde{H}_{\varpi}=0.
\end{equation}
In local complex coordinates, it is
\begin{equation}\label{perturb-lg-complex}
\pat_\zb u^j= \mathsf{k}\sum_{\ib}h^{\ib j}\pat_\ib \overline{W}+\mu^j,
\end{equation}
where $\mu^j$ comes from the perturbation.
The energy functional of (\ref{equa-flow-real-1}) is given by
\begin{equation}\label{sec4:ener-func}
E_{\widetilde{H}_{\varpi}}(\phi)=\frac{1}{2}\int^\infty_{-\infty}\int^1_0 \left(\left|\frac{\pat \phi}{\pat s}\right|^2+\left|\frac{\pat \phi}{\pat t}-X_{\widetilde{H}_{\varpi}}(\phi) \right|^2 \right).
\end{equation}

\begin{prop} Let $\phi$ be a smooth solution of (\ref{equa-flow-real-1}) flowing from $l^-\in S_W(L_0,L_1)$ to $l^+\in S_W(L_0,L_1)$. Then we have
\begin{equation}
E_{\widetilde{H}_{\varpi}}(\phi)=A_{\widetilde{H}_{\varpi}}(l^-,\Phi^-)-A_{\widetilde{H}_{\varpi}}(l^+,\Phi^+).
\end{equation}
where $\Phi^+=\Phi^-\#\phi$.
\end{prop}

\begin{proof}We have the following computation:
\begin{align*}
&\int_{\R\times [0,1]} g(\frac{\pat \phi}{\pat s},J\cdot\frac{\pat\phi}{\pat t}-\nabla_g(\widetilde{H}_{\varpi}))\\
&=-\int_{\R\times [0,1]}\om(\frac{\pat \phi}{\pat s},\frac{\pat \phi}{\pat t})-\int_{\R\times [0,1]}\frac{\pat (\widetilde{H}_{\varpi})}{\pat s}\\
&=-\int \phi^*\om-\int^1_0(\widetilde{H}_{\varpi}(l^+(t))-\widetilde{H}_{\varpi}(l^-(t)))dt\\
&=-\int (\Phi^+)^*\om+\int (\Phi^-)^*\om-\widetilde{H}_{\varpi}(l^+(0))+\widetilde{H}_{\varpi}(l^-(0))\\
&=A_{\widetilde{H}_{\varpi}}(l^+,\Phi^+)-A_{\widetilde{H}_{\varpi}}(l^-,\Phi^-).
\end{align*}
Hence we have
\begin{align*}
E_{\widetilde{H}_{\varpi}}(\phi)&=\frac{1}{2}\int^\infty_{-\infty}\int^1_0 \left(\left|\frac{\pat \phi}{\pat s}+J\cdot\frac{\pat \phi}{\pat t}-J\cdot X_{\widetilde{H}_{\varpi}}(\phi) \right|^2 \right)+A_{\widetilde{H}_{\varpi}}(l^-,\Phi^-)-A_{\widetilde{H}_{\varpi}}(l^+,\Phi^+)\\
&=A_{\widetilde{H}_{\varpi}}(l^-,\Phi^-)-A_{\widetilde{H}_{\varpi}}(l^+,\Phi^+)
\end{align*}
\end{proof}

\begin{thm}\label{sec4-thm-ener-conv-expo-1}
Let $\phi:\Sigma \to M$ be a smooth solution of (\ref{equa-flow-real-1}). Then the following three conclusions are equivalent:
\begin{itemize}
\item[(i)] $E_{\widetilde{H}_{\varpi}}<\infty$
\item[(ii)] There exist solutions $l^\pm(t)\in S_{\mathsf{k},W}(L_0,L_1)$ such that
$$
\lim_{s\to \pm\infty}\phi(s,t)=l^\pm(t).
$$
\item[(iii)] There exist constants $\delta>0$ and $c>0$ such that
$$
|\pat_s\phi(s,t)|\le ce^{-\delta|s|},
$$
for all $(s,t)\in \Sigma$.
\end{itemize}
\end{thm}

\begin{proof}(iii)$\Rightarrow$ (i) is easy. The proof of (i)$\Rightarrow$ (ii) is given in Theorem \ref{limit} and the proof of (ii)$\Rightarrow$ (iii) is given in Theorem \ref{sec5:exponential}.
\end{proof}

If $M$ is an exact K\"ahler manifold,  the action functional $A_{\widetilde{H}_{\varpi}}(l,\Phi)$ can be further simplified. As $M$ is exact, $\omega$ can be expressed as $d\lambda$ for some $\lambda\in \Omega^1(M)$. Let $L_i$ be Lagrangian submanifold corresponding to $p_i\in C_W$. Since $d\lambda|_{L_i}=\om|_{L_i}=0$ and $L_i$ is simply connected, there exists a smooth function $\gamma_i$ defined on $L_i$ such that $\lambda=d\gamma_i$. $\gamma_i$ is uniquely determined if we normalize it by setting $\gamma_i(p_i)=0$.

\begin{thm}\label{exact-functional}
Assume that $M$ is exact, then the action functional $A_{\widetilde{H}_{\varpi}}(l,\Phi)$ is independent of $\Phi$, where $\Phi$ is any homotopy connecting $l$ and $l_0$. If we define
\begin{equation}
A_{\widetilde{H}_{\varpi}}(l)=-\left[\gamma_0(l(0))-\gamma_1(l(1))+\int^1_0 l^*\lambda\right]-\int^1_0 \widetilde{H}_{\varpi}dt,
\end{equation}
then
$$
A_{\widetilde{H}_{\varpi}}(l,\Phi)=A_{\widetilde{H}_{\varpi}}(l)-A_{\widetilde{H}_{\varpi}}(l_0).
$$
\end{thm}

\begin{proof} It suffices to compute the first term of $A_{\widetilde{H}_{\varpi}}(l,\Phi)$,
\begin{align*}
-\int_{[0,1]\times[0,1]} \Phi^* \omega&=-\int_{\pat ([0,1]\times[0,1])} \Phi^* \lambda\\
&=-\int_0^1 l^*\lambda+\int^1_0 l_0^*\lambda-\int^1_0 \Phi^*(s,0)d\gamma_0+\int^1_0 \Phi^*(s,1) d\gamma_1\\
&=-[\gamma_0(l(0)-\gamma_1(l(1))+\int_0^1 l^*\lambda]+[\gamma_0(l_0(0)-\gamma_1(l_0(1))+\int_0^1 l_0^*\lambda]
\end{align*}
\end{proof}

\begin{crl}Assume that $M$ is exact. Let $\phi$ be a smooth solution of (\ref{equa-flow-real-1}) flowing from $l^-\in S_{\mathsf{k},W}(L_0,L_1)$ to $l^+\in S_{\mathsf{k},W}(L_0,L_1)$. Then we have
\begin{equation}
E_{\widetilde{H}_{\varpi}}(\phi)=A_{\widetilde{H}_{\varpi}}(l^-)-A_{\widetilde{H}_{\varpi}}(l^+).
\end{equation}
\end{crl}

\section{Deligne-Mumford-Stasheff compactifications}\label{sec-5}

In subsections \ref{Recall1} and \ref{Recall2}, we follow closely Seidel's treatment in \cite{Si3}. In subsection \ref{sec5-cano-sect}, we discuss the auxiliary section of the log-canonical bundle which will be important for Witten equation.

\subsection{Riemann surface with strip-like ends}\label{Recall1}

Let $\cCh$ be a compact Riemann surface with boundary, and let $\cZ$ be a finite set of boundary points of $\cCh$. Then $\cC=\cCh\setminus \cZ$ is a pointed-boundary Riemann surface.  We can classify the points in $\cZ$ as two types: the incoming points and the outcoming points and denote by $\cZ^-$ the subset consisting of incoming points and $\cZ^+$ the subset consisting of outcoming points. We can number the components $C\subset \pat \cC$ that are coming from the same boundary component of $\cCh$ as follows. Let $B$ denote a connected component of $\pat\cCh$ and $\cZ_{B}=\cZ\cap B$. Then $B$ has an induced orientation from $\cCh$. Along this orientation, we can number the punctures by $\zeta_0,\zeta_1,\cdots,\zeta_{d}$, where $d+1=\#\cZ_{B}$, and number the component $C$ ahead of $\zeta_i$ by $C_i$ for $i=0,\cdots,d$.

There are some special pointed-boundary surfaces. The closed upper half plane $\Hb$ can be considered as the closed unit disc $D$  removing one boundary point and defining this point at infinity as the incoming point. The surface $\bar{\Hb}$ is the closed half plane $\Hb$ with the point at infinity defining as the outgoing one. We can remove two boundary points from $D$ and label one point as the incoming point and the other one as the outgoing point. This surface is biholomorphically equivalent to the infinitely long strip $\Sigma=\R\times [0,1]$ with coordinates $(s,t)$, where the incoming point appears at $s=-\infty$ and the outgoing point appears at $s=+\infty$.

Denote $\Sigma^\pm=\R^\pm\times [0,1]$, where $\R^+:=[0,\infty)$ and $\R^-:=(-\infty, 0]$.

\begin{df} A set of \textit{strip-like ends} for a pointed-boundary Riemann surface $\cC$ consists of proper holomorphic embeddings $\ve_\zeta: \Sigma^\pm\to \cC$, one for each $\zeta\in \cZ^\pm$, satisfying
$$
\ve_\zeta^{-1}(\pat \cC)=\R^\pm\times\{0,1\}, \;\text{and}\;\lim_{s\to \pm \infty}\ve_\zeta(s,\cdot)=\zeta,
$$
and with the additional requirement that the images of $\ve_\zeta$ are pairwise disjoint.
\end{df}

The Riemann surface structure on $\Sigma^\pm$ extends to the one-point compactification $\widehat{\Sigma}^\pm=\Sigma\cup \{\pm \infty\}$ and each strip-like end extends to a holomorphic embedding $\widehat{\ve}_\zeta:\widehat{\Sigma}^\pm\to \cCh$ which maps $\pm\infty$ to $\zeta$.

Given two punctured Riemann surfaces $\cC_1$ and $\cC_2$ with two punctures $\zeta^+_1$ and $\zeta^-_2$ respectively, there is a gluing operation $\#_l$ with respect to the gluing length $l>0$. This is defined as follows. Set
\begin{align*}
\cC'_1&=\cC_1\setminus \ve_{\zeta^+_1}((l,\infty)\times [0,1])\\
\cC'_2&=\cC_2\setminus \ve_{\zeta^-_2}((-\infty,-l)\times [0,1]),
\end{align*}
and define the new glued surface $\cC_l:=\cC_1\#_l\cC_2:=(\cC'_1\sqcup \cC'_2)/\sim$, where the equivalence relation $\sim$ is the following identification
$$
\ve_{\zeta^-_2}(s-l,t)\sim \ve_{\zeta^+_1}(s,t),\forall (s,t)\in [0,l]\times [0,1].
$$

A family of pointed-boundary oriented surfaces can be given as a fibration $\pi:\cC_\cR\to \cR$, where $\cR$ is the base manifold. Each fiber is given by a pointed-boundary surface with the set of punctures $\cZ$ and the structure group is the group of oriented diffeomorphisms. A set of strip-like ends for $\cC_\cR\to \cR$ consists of proper embeddings $\ve_\zeta: \cR\times \Sigma^\pm\to \cC_\cR$ fibered over $\cR$, one for each $\zeta\in \cZ^\pm$, which restrict to strip-like ends on each fiber. There is a natural compactification $\hat{\pi}:\cCh_\cR \to \cR$ such that each point $\zeta\in \cZ$ gives a section $\zeta: \cR\to \pat \cCh_\cR$ such that $\cC_{\cR}=\cCh_{\cR}\setminus \zeta (\cR)$. A family of pointed-boundary Riemann surfaces is a family of pointed-boundary oriented surfaces equipped with a family of complex structures $I_\cC$. The proper embeddings $\ve_\zeta:\cR\times \Sigma^\pm \mapsto \cC_\cR$ extend automatically to a smooth fiberwise holomorphically embeddings $\cR\times \widehat{\Sigma}^\pm\to \cCh_\cR$, whose restrictions to $\pm \infty\in \widehat{\Sigma}^\pm$ are the corresponding section $\zeta$.

Any family of $\cC_\cR\to \cR$ admits strip-like ends (\cite[PP. 113]{Si3}). Using the local trivialization of $\cC\to \cR$, we can make the strip-like ends constant in the following sense. For $r_0\in \cR$, we have $\cC:=\cC_\cR|_{r_0}$ and the strip-like ends $\ve_{r_0,\zeta}=\ve_\zeta(r_0,\cdot): \Sigma^\pm \to \cC$. For a sufficiently small neighborhood $U\ni r_0$, there is a fiber diffeomorphism $\Psi: U\times \cC\to \cC_\cR|_U, \Psi(0)=Id$ and $\ve_{\zeta}(r, s,t):=\Psi(r, \ve_{\cR,\zeta}(s,t))$. Then $\Psi^* I_{\cC_\cR}$ is a family of complex structures on $\cC$, parametrized by $r\in U$, which varies only on the compact subset $\cC\setminus \sqcup_{\zeta}\ve_\zeta(\Sigma^\pm)$.

We can study the deformation of the pointed-boundary Riemann surfaces. Let $T\cC$ be the holomorphic tangent bundle of a boundary pointed Riemann surface $\cC$ with set of punctures $\cZ$. Let $T_\R (\pat\cC)$ be the real tangent bundle of the boundary $\pat \cC$. Define the sheaf $T_\pat \cC$ to be the sheaf of sections of $T \cC$ valued in $T_\R(\pat\cC)$ when restricted on boundary $\pat \cC$. The sheaf cohomology group $H^0(\cC, T_\pat \cC)$ gives the global continuous automorphism group, while $H^1(\cC, T_\pat \cC)$ is the obstruction of the infinitesimal deformation and is used to classify the complex structure. If $\cC_\cR\to \cR$ is a family of pointed-boundary Riemann surfaces and $\cC:=\cC_\cR|_{r_0}$, then there is the Kodaira-Spencer map:
$$
\kappa_{r_0}: (T\cR)_{r_0}\to H^1(\cC, T_\pat \cC).
$$
If $\kappa_{r}$ is isomorphism for any $r\in \cR$, then the deformation $\cC_\cR\to \cR$ is called a miniversal deformation. The dimension of the miniversal deformation space is $|\cZ|-3\chi(\cCh)$ if $\cC$ is stable, i.e,. $H^0(\cC, T_\pat \cC)=0$. This number can be computed by using the boundary version of Rieman-Roch theorem of the Cauchy-Rieman operator $\bpat$ (ref. \cite[PP. 114]{Si3}).

\subsection{Pointed discs and Deligne-Mumford-Stasheff compactifications}\label{Recall2}

To construct the genus $0$ Fukaya category, one needs to study the pointed discs and its Deligne-Mumford-Stasheff compactification space.

\begin{df} A \textit{$(d+1)$-pointed disc}, $d\ge 0$, is a pointed-boundary Riemann surface $\cC$ with compactification $\cCh\cong D$, and which has one incoming point and $d$ outgoing points.
\end{df}

We label the incoming point by $\zeta_0$, and then along the anti-clockwise direction label the other $d$ points by $\zeta_1,\cdots,\zeta_d$. So we have $\cZ^-=\{\zeta_0\}$ and $\cZ^+=\{\zeta_1,\cdots,\zeta_d\}$. We denote by $C_i$ the component of the boundary $\pat \cC$ between $\zeta_i$ and $\zeta_{i+1}$ for $i=0,1,\cdots, d$ (here we set $d+1\equiv 0$ for the subscript).

If $d\ge 2$, the $(d+1)$-pointed discs has no nontrivial automorphisms. There is a universal deformation $\cC^{d+1}\to \cR^{d+1}$, where $\dim \cR^{d+1}=d-2$. Given the order of punctures $\zeta_0\to \zeta_1\to \cdots \zeta_d$, the universal family can be constructed as follows. Let $\Conf_{d+1}(\pat D)\subset (\pat D)^{d+1}$ be the configuration space of $(d+1)$-tuples of different points on the circle respect the given order. We have the automorphism group $\aut(D)=\mathbb{P} S\!L_2(\R)$, the deformation space $\cR^{d+1}=\Conf_{d+1}(\pat D)/\aut(D)$ and the compactified universal curve $\cCh^{d+1}=\Conf_{d+1}(\pat D)\times_{\aut(D)}D$. The section $\zeta_k: \cR^{d+1}\to \cCh^{d+1}$ is defined by $[z_0,\cdots,z_d]\mapsto [z_0,\cdots,z_d; z_k]$. Then we have the universal family of the $(d+1)$-pointed boundary Riemann surface $\cC^{d+1}=\cCh^{d+1}\setminus \bigsqcup_k \zeta_k(\cR^{d+1})$.

\begin{df}
Define a \textit{universal choice of strip-like ends} to be a choice of, $\forall d\ge 2$, a set of strip-like ends $\{\ve^{d+1}_k, 0\le k\le d\}$ for $\cC^{d+1}\to \cR^{d+1}$ such that any family of $(d+1)$-pointed discs inherits strip-like ends through its classifying map.
\end{df}

The deformation bundle $\cC_\cR\to \cR$ is not compact. Its compactification is related to the gluing operation of two pointed discs stratified by tree graph.

\begin{df} A \textit{$d$-leafed tree}, for $d>0$, is a proper embedding planar tree $T\subset \R^2$, which has $d+1$ semi-infinite edges: one of them is called \textit{root}, and others are called the \textit{leaves}. The semi-infinite edges are also called the \textit{exterior} and the others \textit{interior} ones.
\end{df}

Let $\Ve(T)$ be the set of vertices of $T$, and let $\Ed^{\Int}(T)$ be the set of the interior edges. A flag is a combination $f=(v,e)$, where $v\in \text{Ve}(T)$ and $e$ is an adjacent edge to $v$. Let $\Fl^{\Int}(T)$ be the set of interior flags. We will orient $T$ by flowing "upwards" from the root to the leaves. A flag $f=(v,e)$ is called positive, if the orientation points away from $v$, and negative otherwise. A flag $f=(v,e)$ related to an exterior edge $e$ always has a fixed orientation, while an interior edge $e$ corresponds to two flags $f^{\pm}(e)$, one is positive and the other negative. Sometimes it is useful to number the flags adjacent to a given vertex $v$ by $f_0(v), \cdots f_{|v|-1}(v)$, where $|v|$ is the valency. This numbering always starts with the unique negative flag $f_0(v)$ and continues anti-clockwise with respect to the given embedding.

\begin{df} A $d$-leafed tree $T$ is called \textit{stable} if $|v|\ge 3$ for any vertices $v$ and \textit{semi-stable} if $|v|\ge 2$.
\end{df}

If assigning  a length $l_e>0$ to each interior edge $e\in \Ed^{\Int}(T)$, then a $d$-leafed tree $T$ with these lengths $\{l_e\}$ will generate a family of $(d+1)$-disc $\cC$'s. $\cC$ is generated by the following gluing operation. Let each vertex $v$ correspond to the deformation space $\cR^{|v|}$ consisting of $|v|$-pointed disc $\cC_v$ with strip-like ends $\cZ(v)$. Let each flag $f^\pm=(v,e)$ correspond to a puncture $\zeta_f\in \cZ^\pm(v)$. Take $\cC_v\in \cC_{\cR(v)}$. The first part of the gluing process is to cut off a piece of each end belonging to an interior flag:
$$
\cC'_v =\cC_v\setminus \bigsqcup_{f=(v,e)\in \Fl^{\Int}(T)} \ve_f(\Sigma^\pm_{l_e}),
$$
where $\Sigma^-_{l_e}=(-\infty,-l_e)\times [0,1]$ and $\Sigma^+_{l_e}=(l_e,\infty)\times [0,1]$.  Then identifying the stubs of these cut-off ends with each other: set
$$
\cC_l=\bigsqcup_v \cC'_v/\sim,
$$
where $\sim$ means the identification $\ve_{f^-(e)}(s-l_e,t)\cong \ve_{f^+(e)}(s,t)$ for any $(s,t)\in [0,l_e]\times [0,1]$. The strip-like ends of $\cC_l$ are inherited from the exterior flags of $T$. The $k$-th strip-like end of $\cC_l$ is given by the map
$$
\ve_{l,k}=\ve_{f_{\mathsf{k}}}:\Sigma^\pm \mapsto \bigsqcup_v \cC'_v\mapsto \cC_l.
$$
The glued surface $\cC_l$ has a thick-thin decomposition:
\begin{align*}
 \cC_l^{thin}:&=\bigsqcup_{k=0}^d \ve_{l,k}(\Sigma^\pm)\;\sqcup \bigsqcup_{e\in \Ed^{\Int}(T)} \ve_{f^+(e)}([0,l_e]\times [0,1])\\
\cC_l^{thick}:&=\bigsqcup_v\left(\cC_v\setminus \bigcup_{k=0}^{|v|-1} \ve_{F_{\mathsf{k}}(v)}(\Sigma^\pm) \right).
\end{align*}
The thin part $\cC^{thin}_l$ represent the gluing domain.

For convenience, we can change the gluing parameter $l_e\in [0,\infty)$ by $\rho_e\in (-1,0]$, where
$l_e=-\ln (-\rho_e)/\pi$, and write $\cC_l=\cC_\rho$. The $\rho_e=0$ case corresponds to $l_e=+\infty$. We can allow the gluing parameter $\rho_e$ to be $0$, which means that we do not do any gluing operations between those $\cC_v$'s. In this case, the tree $T$ degenerates to $\bar{T}$ by collapsing all the interior edges.

Now the gluing operation defines a family of pointed-boundary Riemann surface:
\begin{equation}\label{sec8-DMSC-equa-1}
\cC_\cR\to \cR=(-1,0)^{\Ed^{\Int}(T)}\times \prod_{v} \cR_v,
\end{equation}
and we can also get $\overline{\cC}_\cR\to \overline{\cR}$ if the gluing parameters are allowed to be zero.

As a set, the Deligne-Mumford-Stasheff (DMS) compactification of $\cR^{d+1}$ is
\begin{equation}\label{sec8-DMSC-equa-2}
\overline{\cR}^{d+1}=\coprod_T \cR^T,
\end{equation}
where the union is over all stable $d$-leafed tree $T$, and one sets $\cR^T=\prod_{v\in \Ve(T)} \cR^{|v|}$. The natural topology or smooth structure can be defined on $\overline{\cR}^{d+1}$ as follows. The family of surfaces defined in (\ref{sec8-DMSC-equa-1}) is classified by a smooth gluing map
$$
\gamma^T: \cR^T\to \cR^{d+1}.
$$
One can allow (some or all of ) the gluing parameters to become zero. The result is a degenerate glued surface, which is disjoint union of pointed discs governed by a collapsed tree $\bar{T}$. This determines a point in the $\bar{T}$-stratum in (\ref{sec8-DMSC-equa-2}). Hence we have the canonical extension of the gluing map:
\begin{equation}\label{sec8-DMSC-equa-3}
\bar{\gamma}^T:\overline{\cR}^T:=(-1,0]^{\Ed^{\Int}(T)}\times \cR^T \mapsto \overline{\cR}^{d+1}.
\end{equation}

Now the DMS compactification is the one induced from the maps (\ref{sec8-DMSC-equa-3}) for all $T$. The smooth structure of $\overline{\cR}^{d+1}$ can also be obtained by the embeddings:
$$
\overline{\cR}^{d+1}\subset \MM_{0,d+1}(\R)\subset \MM_{0,d+1},
$$
where $\MM_{0,d+1}(\R)$ and $\MM_{0,d+1}$ are the real and complex Deligne-Mumford spaces of closed oriented Riemann surfaces.

\begin{thm}
$\overline{\cR}^{d+1}$ is a real $(d-2)$-dimensional smooth manifold with corners.
\end{thm}

Now we have two ways of constructing the strip-like end structure for a surface lying near the boundary of $\overline{\cR}^{d+1}$: one is inherited from the pull-back from the classifying map to the universal family $\cC_\cR^{d+1}\to \cR^{d+1}$ and the other is obtained by the gluing map $\gamma^T$ from those components $\cC^{|v|}$. If the two constructions coincide, we say that the DMS compactification \textit{has a consistent universal choice of strip-like ends}.

\begin{lm}\cite[Lemma 9.3]{Si3} Consistent universal choices of strip-like ends exist.
\end{lm}

\subsection{The auxiliary section $\sigma$ of the $\log$-canonical bundle}\label{sec5-cano-sect}\

To define the Witten equation on a general Riemann surface $\cC$ with or without boundary, one needs to choose a auxiliary section $\sigma$ of the $\log$-canonical bundle $K_{\cC,\log}=K_\cC\otimes \cO(D), D=\sum_{i=0}^d p_i$. Assume that $\cC$ is a closed Riemann surface. By Kodaira vanishing theorem, we have $H^1(\cC, K_{\cC,\log})=\{0\}$. By Riemann-Roch theorem, we have
$$
\dim H^0(\cC, K_{\cC,\log})=g+|D|-1=g+d.
$$
The stability condition of $\cC$ requires that $\deg(K_{\cC,\log})>0$, which is equivalent to $2g-2+|D|>0$, i. e. $2g+d-1>0$.

If genus $g\ge 1$ or $g=0, d\ge 2$, then the dimension of the global holomorphic section of $K_{\cC,\log}$ is positive and we can always get a global holomorphic section $\sigma$. From now on, when talking about a \textit{global section} of $K_{\cC, \log}$, we always mean a global holomorphic section. In this case, the automorphism group is finite. If genus $g=0, d=1$, we can also get such a global section $\sigma$, but the automorphism group is of complex $1$-dimension. After modulo the equivalence, the global section is unique in this case. Hence if $g=0$, then we require $d\geq 1$ to grantee the existence of the global section $\sigma$.

Now we consider the case of Riemann surfaces with boundary, and for notational simplicity and for our purpose in this paper, we specialize to the case that the closure $\widehat{\mathscr{C}}$ has one boundary component and the genus of $\mathscr{C}$ is $0$ (the general case can be worked out similarly), where recall $\widehat{\mathscr{C}}$ is introduced in subsection \ref{Recall1}. Namely, let $\cC$ be a punctured-boundary disc with $d+1$ distinct punctures on $\pat\cC$. Using the complex structure $j_\cC$, the log-canonical bundle $K_{\cC,\log}$ can be splitted into the direct sum of the real bundle $K^R_{\cC,\log}$ and the imaginary bundle $K^I_{\cC,\log}$.  Define $H^0(\cC,\pat\cC,K_{\cC,\log})$ to be the space of global section $\sigma$ of $K_{\cC,\log}$ with simple poles at punctures and satisfying the boundary condition $\sigma|_{\pat\cC}\in K^R_{\cC,\log}$ (equivalently, $\im(\sigma)|_{\pat \cC}=0$).  Using the index gluing technique from elementary pieces, one can show that
$$
\dim_\R H^0(\cC,\pat\cC,K_{\cC,\log})=d.
$$

For $d\geq 2$, the open moduli space of such discs is $\mathscr{R}^{d+1}$ of $d-2$ real dimension, as discussion in subsection \ref{Recall2}, essentially determined by the relative positions of punctures on $\partial\mathscr{C}$. Fix such a disc, the space of holomorphic sections $\sigma$ of $K_{\mathscr{C}, \text{log}}$ with prescribed simple poles at punctures and $\text{Im}(\sigma)|_{\partial\mathscr{C}}=0$ is $d$ real dimensional. To understand this more concretely, first remark that if holomorphic $\sigma$ is nonzero, the sum of residues must be zero (using genus $0$), by integrating along the loop tracing out the punctured boundary (limiting from the interior). In the strip-like end description, the size of the (real) residue (real as determined by the boundary restriction $\text{Im}(\sigma|)_{\partial\mathscr{C}}=0$) is nothing but the strip width, and the residue is negative for outgoing ends/punctures and positive for incoming ends. We will consider discs with one incoming puncture and $d$ outgoing puncture for the $A_\infty$ structure underlying the associativity of the product operation up to homotopy; and in this case such a $\sigma$ is completely determined by fixing the strip widths at $d$ positive punctures (as the residue at the negative puncture will be determined as the sum), which explains the dimensionality. We discretize the widths (sizes of the negative residues) at $d$ positive punctures from $\mathbb{(R^+)}^d$ to $\mathbb{N}^d$ by choosing natural numbers, partly because $CF(L^{\#}_i, L^{\#}_j)$ (a notion we will introduce later) with generators of different fixed widths are quasi-isomorphic anyway (but we do not need this fact in our construction), partly because $\mathbb{N}$ is subset of $\mathbb{R}^+$ which is closed under multi-additions of lengths (induced from higher multiplications). In other words, if we consider the moduli space of $d+1$ punctured-boundary discs with $d$ positive punctures each with prescribed width valued in natural numbers is precisely $\mathscr{R}^{d+1}\times \mathbb{N}^d$ and each element then admits a unique such a $\sigma$ such that the sizes of its negative residues at those positive punctures are precisely the prescribed widths. They glue compatibly in the expected way in describing the neighborhood of the compactification. If $d=1$, then such a global section $\sigma$ exits and is unique after modding out a real $1$-dimensional automorphism group. The above understanding of the auxiliary section $\sigma$ is crucial because it allows the Witten equation expression to make sense in the same setting, analogous to the role of W-structures in the FJRW closed quantum singularity case.

\section{Witten equation and action functional: \uppercase\expandafter{\romannumeral2}}\label{Witten equation section}\label{sec-6}

\subsection{Witten equation and its perturbation}\label{sec-6.1}\

Let $(M,h,W)$ be a regular tame exact LG system. In section \ref{sec-3}, we have defined the following space
\begin{equation*}
\mathcal{H}=\{\varphi\in C^{\infty}_{c,\epsilon}(M)\;|\;\omega_{\varphi}=\omega+\sqrt{-1}\pat_M\bpat_M\varphi>0\text{ on }M \},
\end{equation*}
where $\om$ is the K\"ahler (symplectic) form associated to the metric $h$.

We extend the definition of the subspace $\G^L_{\mu}\subset \mathcal{H}$ in section \ref{sec-3} (\ref{twothimbles}) to the following version:
\begin{equation*}
\G^L:=\{\varphi \in \mathcal{H}\;|\;\pat_M\bpat_M \varphi|_{L_p^\pm}=0,\forall p\in C_W\}
\end{equation*}

Similarly, for $\mu>0$, define
$$
\G^L_{\mu}=\{\varphi\in\G^L\;|\; ||\pat_M\bpat_M \varphi||_{C^0(M)}<\mu\}.
$$
As before, we identify the set $\G^L_\mu$ and the set of K\"ahler forms $\om+\sqrt{-1}\pat_M\bpat_M \tilde g, \tilde g\in \G^L_\mu$.

Let $p_{i_0},p_{i_1}\in C_W$ are two critical points satisfying $p_{i_0}>p_{i_1}$, then the corresponding positive Lefschetz thimbles satisfy the order relation $L_{i_0}^+>L_{i_1}^+$.

\textbf{Convention}: In the following in this article, we always set $L_{i_j}=L^+_{i_j}, j=1,2$.

For any $\gt\in \G^L_{\mu},\mathsf{k}\in \mathbb{N}$, consider the equation of Hamiltonian chords:
\begin{equation}
\begin{cases}
 \frac{\pat \phi(t)}{\pat t}=-\mathsf{k}\nabla_{\gt} \im(W)(\phi(t))\\
 \phi(0)\in L_{i_0},\phi(T)\in L_{i_1},
 \end{cases}\label{sec9:equa-nega-flow-1}
 \end{equation}

By Lemma \ref{fini-rad} and Theorem \ref{transversal}, we can choose a $\mu$ sufficiently small, then for generic metric $\gt\in \G^L_\mu$ and generic $T\in (1/2,2)$, the solutions of (\ref{sec9:equa-nega-flow-1}) are finite in number, and $L_{i_0}$ and $L_{i_1}$ intersect transversally at time $T$ at any speed $\mathsf{k}\in \mathbb{N}$ in the sense of Section \ref{sec-3}.

Thus, for generic $\gt\in \G^L_\mu$, the set $S_{\mathsf{k},W}$ of all Hamiltonian chords is a finite set. Let $\cC$ be a pointed-boundary Riemann surface with the set of punctures $\cZ=\cZ^+\sqcup \cZ^-$. Each element $\zeta\in \cZ^\pm$ corresponds to a strip-like end $\ve_\zeta(\Sigma^\pm)$. $(\cC,L)$ is called a \textit{pointed-boundary Riemann surface with directed Lagrangian labels}, if any component $C\subset \pat \cC$ is labeled by a Lefschetz thimble $L_C$ and satisfies the following order criterion: for any two components $C_{i_0}$ and $C_{i_1}$ belongs to the same component of $\pat \cCh$, if $i_0<i_1$, then $L_{i_0}>L_{i_1}$ (which means $\text{Im}W(x_{i_0})>\text{Im} W(x_{i_1})$).

By the discussion of Section \ref{sec5-cano-sect}, there exists a nonzero global section $\sigma$ of the canonical bundle $K_{\cC,\log}$ such that the volume form $\nu_\cC=\frac{i}{2} h_\cC \sigma\wedge \bar{\sigma}$, where $h_\cC$ is a smooth positive function on $\cC$. In strip-like ends, we can take $\sigma=dz$ and $h_\cC=1$. Note that $\im(\sigma)|_{\pat \cC}=0$.

Let $\phi\in C^\infty(\cC,M)$, then we have the pull-back section $X_{R,\sigma}(\phi)\in \Omega^{0,1}(\cC, \phi^* TM)$ defined by
$$
X_{R,\sigma}(\phi)=-\frac{1}{2}(X_{\re(W)}(\phi)\im (\sigma)+X_{\im(W)}(\phi)\re(\sigma)).
$$
In local coordinates, if we take $\sigma=dz, z=s+it$, then we have
\begin{align*}
X_{R,\sigma}(\phi)&=-\frac{1}{2}(X_{\re(W)}(\phi)dt+X_{\im(W)}(\phi)ds)\\
&=-\frac{1}{2}(I+J\circ(\;)\circ j_\cC)X_{\im(W)}(\phi)ds\\
&=-\frac{1}{2}(I+J\circ(\;)\circ j_\cC)X_{\re(W)}(\phi)dt.
\end{align*}
With this pull-back section $X_{R,\sigma}(\phi)$, we can define the Witten equation on $\cC$:
$$
\bpat_J \phi+X_{R,\sigma}(\phi)=0,
$$
which has the local expression
$$
\frac{\pat \phi}{\pat s}+J(\phi)\frac{\pat \phi}{\pat t}-\nabla_g \re(W)(\phi)=0.
$$
If $\cC$ is equipped with directed Lagrangian labels $\{L_C\}$, then we can study the boundary value problem:
\begin{equation}
\begin{cases}
\bpat_J \phi+X_{R,\sigma}(\phi)=0\\
\phi|_C\in L_C, \forall \text{ component } C\subset \pat \cC.
\end{cases}
\end{equation}
Now for each puncture $\zeta\in \cZ$, there are two adjacent components $C_{\zeta,0}$ and $C_{\zeta,1}$ with the corresponding Lagrangian labels $L_{\zeta,0}>L_{\zeta,1}$. For each $\zeta\in \cZ$, we can take a Hamiltonian chord $l_\zeta\in S_W(L_{\zeta,0}, L_{\zeta,1})$. Now letting $p>2$ and fixing the set $\{l_\zeta,\zeta\in \cZ\}$, we can define the configuration space
\begin{align*}
&\B^{d+1}(\cC, L, l)\\
=&\{\phi\in W^{1,p}_{\loc}(\cC,M)\;|\;\phi|_C\in L_C, \forall \text{ component }C\subset \pat\cC;\\
&\lim_{s\to\pm\infty}\phi(\ve_\zeta(s,\cdot))=l_\zeta,\forall \zeta\in \cZ,\#\cZ=d+1\}.
\end{align*}
For simplicity, we write it as $\B_{\cC}$. Its tangent space at $\phi$ is given by $T_\phi \B_{\cC}=W^{1,p}(\cC, \phi^*TM,F)$, where $F|_C=\phi^*TC$ for any component $C\subset \pat \cC$. There is a Banach bundle $\E_{\cC}\to \B_{\cC}$ whose fiber at $\phi$ is
$L^p(\cC, \Omega^{0,1}_\cC\otimes \phi^* TM)$.

\begin{df}
The moduli space $\W^{d+1}(\cC, L, l)\subset \B^{d+1}(\cC, L, l)$ consists of those elements $\phi\in W^{1,p}_{\loc}$ satisfying the following problem:
\begin{equation}
\begin{cases}
\bpat_J \phi+X_{R,\sigma}(\phi)=0\\
\phi|_C\in L_C, \forall \text{ component }C\subset\pat \cC\\
\lim_{s\to\pm\infty}\phi(\ve_\zeta(s,\cdot))=l_\zeta,\forall \zeta\in \cZ.
\end{cases}
\end{equation}
For simplicity, we write it as $\W_{\cC}$.
\end{df}

The map $\WI_\cC(\phi)=\bpat_J \phi+X_{R,\sigma}(\phi)$ is a section of the bundle $\E_{\cC}\to \B_{\cC}$ and we call it the Witten map. Then the moduli space $\W_\cC$ is the zero locus of the Witten map. Let $\overline{\cR}$ be the DMS space, then we have the Banach bundle $\B\to \overline{\cR}$ and the Banach bundle $\E\to \B$. Then we have a section of $\E\to \B$ which takes the value $\WI_\cC(\phi)$ at $\phi\in \B_\cC$.  We call this section the \textit{extended Witten map}.

\subsection{Perturbed Witten map}\label{perturb-witten-section}\

For any $\zeta\in \cZ$, there are two adjacent Lefschetz thimbles $L_{\zeta,0}$ and $L_{\zeta,1}$. For any $\varpi_{\zeta,\delta_0}\in V^\epsilon_{\per}$, consider the equation of Hamiltonian chords
\begin{equation}
\begin{cases}
J(\phi)\frac{\pat \phi}{\pat t}-\nabla_g \widetilde{H}_{\zeta}(\phi)=0\\
\phi(0)\in L_{\zeta,0},\phi(1)\in L_{\zeta,1},
\end{cases}
\end{equation}
where $\widetilde{H}_\zeta=\re(W)+\varpi_\zeta$.

Let $S_{W,\varpi_\zeta}(L_{\zeta,0},L_{\zeta,1})$ be the solution set of the above equation.  Note that, by the setting of $V^{\epsilon,\delta_0}_{\per}$ in Definition \ref{perturb-space-stripe}, $L_{\zeta,0}$ will intersect $L_{\zeta,1}$ transversally at time $1$ and we have $S_{W,\varpi_\zeta}(L_{\zeta,0},L_{\zeta,1})=S_{W}(L_{\zeta,0},L_{\zeta,1})$. We call such $\varpi_\zeta$ the \textit{admissible perturbation} on the strip-like end $\ve_\zeta(\Sigma^\pm)$. We take an admissible perturbation $\varpi_\zeta(t,x)$ for any $\zeta\in \cZ$.

Let $\cC$ be a pointed-boundary Riemann surface with directed Lagrangian labels. Denote by $\Omega^{i}(\cC, \mathcal{H})$  the space of  function-valued i-forms on $\cC$. For a $\varpi\in V^{\epsilon,\delta_0}_{\per}$, an element $K\in \Omega^{0}(\cC, \mathcal{H})$ is called an admissible perturbation with respect to $\varpi$, if
\begin{enumerate}
\item[(i)] $\ve_\zeta^*K\equiv \varpi_\zeta(t,x)$ for  $\varpi_\zeta$ on the strip-like end $\Sigma^\pm$ (the domain of $\ve_\zeta$) for any $\zeta\in \cZ$,
\item[(ii)]$
\forall \text{ component }C\subset \pat \cC, \forall \xi\in T C, K(\xi)|_C\equiv 0$.
\end{enumerate}
For an admissible perturbation on $\cC$, we have the perturbation vector valued 1-form on $\cC$:
$$
Y(\phi)=X_K(\phi) \im(\sigma)\in\Omega^{1}(\cC, \phi^* TM),
$$
and
$$
Y^{0,1}(\phi)=\frac{1}{2}(I+I_\cC J)X_K \im(\sigma)\in \Omega^{0,1}(\cC, \phi^* TM),
$$
Here $\sigma$ is the holomorphic 1-form on $\cC$ such that $\nu_\cC=i h_\cC\sigma\wedge \bar{\sigma}$.

We have the perturbed Witten equation:
\begin{equation}\label{perturb-lg-poly}
\bpat_J \phi+X_{R,\sigma}(\phi)=Y^{0,1}(\phi).
\end{equation}
When restricted to the strip-like ends, this equation can be written as
$$
\frac{\pat \phi}{\pat s}+J(\phi)\frac{\pat \phi}{\pat t}-\nabla_g \widetilde{H}_\zeta(\phi)=0,
$$
where $\widetilde{H}_\zeta=\re(W)+\varpi_\zeta$.
\begin{df}
Let $0<\delta_0<\frac{1}{2}$. Define $V^{\epsilon,\delta_0}_{\per}(\cC,\varpi)$ to be the space of admissible perturbations $K$ on $\cC$ with respect to $\varpi$ and such that $|\nabla_g K(z,\phi)|<\delta_0|\nabla_g\re(W)(\phi)|$ for any $z\in \cC, \phi\in M$.
\end{df}

The perturbed Witten map is the section of the bundle $\E_\cC\to \B_\cC$, which is the perturbation of the Witten map by the admissible perturbation $K$.

Let $\overline{\cR}$ be the DMS space. If we want to define the extended Witten map as the perturbation of the Witten map $\WI: \B\to \E$, we must define a family of $K$ in $ \Omega^{0,1}_{\cC/{\bar{\cR}}}(\cC, \mathcal{H})$ with respect to the smooth structure of $\overline{\cR}$, so that $K\in V^{\epsilon,\delta_0}_{\per}(\cC,\varpi)$ for each fixed disc $\cC$.  For any $d$-leafed stable tree $T$ with Lagrangian labels, we have the gluing operation
$$
\gamma^{\bar{T}}: \cR^{\bar{T}}\times (-1,0)^{\Ed^{\Int(T)}}\to \overline{\cR}^{d+1}
$$
Since $K$ is admissible, the perturbation data $\varpi_{\zeta_{f^+(e)}}$ coincides with $\varpi_{\zeta_{f^-(e)}}$ for the gluing edge $e\in \Ed^{\Int(T)}$. Hence $K\in \Omega^{0}_{\cC/{\cR^{\bar{T}}}}(\cC, \mathcal{H})$ can be naturally defined in the image of $\gamma^{T}$. If this extension of $K$ coincides with the original $K\in\Omega^{0}_{\cC/{\overline{\cR}^{d+1}}}(\cC, \mathcal{H})$ for any $T$, then we say that $K$ is a consistent admissible perturbation defined on DMS space $\overline{\cR}$.

\begin{lm} A consistent admissible perturbation $K$ exists on DMS space $\overline{\cR}$.
\end{lm}

\begin{proof}This result was essentially given in \cite[Lemma 9.5]{Si3}. Another construction method is using partition of unity and the induction method with respect to the order of the stable trees.
\end{proof}

Fixing a consistent admissible perturbation $K$, we can define the extended perturbed Witten map $\WI_K$ which is a section of the bundle $\E\to \B$, which takes the value $\WI_{\cC,K}(\phi)$ at the point $(\cC,\phi)\in \B$.

\begin{df}
The perturbed moduli space $\W^{d+1}(\cC, L, l;K)\subset \B^{d+1}(\cC,L,l)$ is a set of maps $\phi\in W^{1,p}_{\loc}$ solving the following problem:
\begin{equation}\label{perturb-lg-poly-system}
\begin{cases}
\bpat_J \phi=\widetilde{X}^{0,1}(\phi)\\
\phi|_C\in L_C, \forall\text{ component } C\subset \pat \cC\\
\lim_{s\to\pm\infty}\phi(\ve_\zeta(s,\cdot))=l_\zeta,\forall \zeta\in \cZ,
\end{cases}
\end{equation}
where $\widetilde{X}^{0,1}=-X_{R,\sigma}(\phi)+Y^{0,1}(\phi)$. For simplicity, we write it as $\W_\cC(K)$.
\end{df}

We have the perturbed Witten map $\WI_{K,\cC}=\bpat_J \phi-\widetilde{X}^{0,1}(\phi)$, which is the perturbed section $\WI_K$ of the Witten map $\WI: \B\to \E$ at the point $\cC$.

In particular, for  two fixed Lefschetz thimbles $L_0$ and $L_1$, and  $l^+,l^-\in S_W(L_0,L_1)$,  we define $\tM(L_0, L_1,l^-,l^+;\varpi)$ to be the set of those $\phi:
\Sigma=\R\times [0,1]\to M$, satisfying the perturbed Witten equation (\ref{perturb-lg-r}) on the strip, with
$$
\phi(s,0)\in L_0, \phi(s,1)\in L_1;\;\lim_{s\to \pm\infty}\phi(s,t)=l^{\pm}(t),\forall t\in [0,1].
$$

Clearly, $\tM(L_0,L_1,l^-,l^+;\varpi)$ is just $\W^2(\cC, L_0, L_1, l^-,l^+;\varpi)$, if we take $\cC$ to be a twice pointed disc.

One can generalize the energy $E_{\widetilde{H}_\varphi}(\phi)$ defined on the strip $\Sigma$ to the general $\cC$ as follows. Notice that $\widetilde{X}^{0,1}(\phi)$ is exactly the $(0,1)$-part of
\begin{equation}
\widetilde{X}(z,\phi)=[X_{\re(W)}(\phi)+X_K(z,\phi)]\im (\sqrt{2}\sigma).
\end{equation}

One can define the energy of $\phi\in \B_{\cC}$ as
\begin{equation}
E_K(\phi)=\frac{1}{2}\int_\cC |d\phi-\widetilde{X}|\nu_\cC.
\end{equation}

\begin{df} Let $\W^{d+1}(K)$ be the moduli space of all $W^{1,p}_{\loc}$ maps satisfying
\begin{equation}\label{sec5-pert-equa-1}
\bpat_J \phi=\widetilde{X}^{0,1}(\phi)
\end{equation}
and $E_K(\phi)<\infty$. Let $\tM(\varpi)$ be the set of maps $\phi :\Sigma=\R\times [0,1]\to M$ satisfying (\ref{perturb-lg-r}) and $E_{\widetilde{H}_{\varpi}}(\phi)<\infty $.
\end{df}

\begin{lm}\label{lg-poly-identity} For any $\phi\in W^{1,p}_{\loc}$, the following identities hold:
\begin{align*}
&(i)\quad (|\bpat_J \phi|^2+|X_{R,\sigma}(\phi)|^2)\nu_\cC=|\bpat_J \phi+X_{R,\sigma}(\phi)|^2 \nu_\cC+d\left( \im(W(\phi)\cdot\sqrt{2}\sigma) \right)\\
&(ii)\quad \frac{1}{2}\left(|d\phi|^2+|X_{\re(W)}(\phi)|^2|\sigma|^{2}\right)\nu_\cC=|\bpat_J \phi+X_{R,\sigma}(\phi)|^2\nu_\cC+\phi^*\om +d\left( \im(W(\phi)\cdot\sqrt{2} \sigma) \right)
\end{align*}
\end{lm}

\begin{proof} In local coordinate chart, we take $\sigma=dz/\sqrt{2}, z=s+it$, and the volume $\nu_\cC=i h_\cC \sigma\wedge \bar{\sigma}=h_\cC ds\wedge dt$ with $h_\cC>0$. We have $|\sigma|^2=h_\cC^{-1}$.

Since $X_{R,\sigma}(\phi)=-\frac{1}{2}\left( X_{\re(W)}(\phi) dt+X_{\im(W)}ds\right)$, we have
\begin{equation}\label{sec9-iden-1}
|X_{R,\sigma}(\phi)|^2\nu_\cC=\frac{1}{2} |X_{\re(W)}(\phi)|^2 |\sigma|^2 \nu_\cC.
\end{equation}

Using the identity
$$
\bpat_J \phi=\frac{1}{2}\left((\frac{\pat \phi}{\pat s}+J\cdot \frac{\pat \phi}{\pat t})ds+ (\frac{\pat \phi}{\pat t}-J\cdot \frac{\pat \phi}{\pat s})dt\right),
$$
we have
\begin{align*}
&\quad 2(\bpat_J \phi, X_{R,\sigma}(\phi)) \nu_\cC\\
&=-\frac{1}{2}\left(h(\frac{\pat \phi}{\pat s}, X_{\im(W)})-h(\frac{\pat \phi}{\pat t}, J\cdot X_{\im(W)})+h(\frac{\pat \phi}{\pat t}, X_{\re(W)})+h(\frac{\pat \phi}{\pat s}, J\cdot X_{\re(W)}) \right)ds\wedge dt\\
&=-\left( h(\frac{\pat \phi}{\pat s}, J\cdot X_{\re(W)}(\phi))+h(\frac{\pat \phi}{\pat t}, X_{\re(W)}) \right)ds\wedge dt\\
&=-[\frac{\pat (\re(W)(\phi))}{\pat s}-\frac{\pat (\im W(\phi))}{\pat t}]ds\wedge dt\\
&=-d(\re(W)(\phi)dt+\im W(\phi) ds)\\
&=-d (\im (W\cdot \sqrt{2}\sigma)).
\end{align*}
This gives (i).

Now by (\ref{sec9-iden-1}) and the identities
\begin{align*}
\begin{cases}
|\bpat_J\phi|^2+|\pat_J \phi|^2=|d\phi|^2\\
(|\bpat_J \phi|^2-|\pat_J \phi|^2)\nu_\cC=-2 \phi^*\omega,
\end{cases}
\end{align*}
it is easy to prove (ii).
\end{proof}

\begin{thm}\label{polytope-lm-2}(Energy identity) Let $\phi\in \W^{d+1}(\cC, L, l)$, then we have
\begin{equation}\label{polytope-lm-2-1}
E_K(\phi)=\int_\cC |\phi_s|^2 |\sigma|^2 \nu_\cC=\A_{\widetilde{H}_\varpi}(l_d)-\sum^{d-1}_{i=0}\A_{\widetilde{H}_\varpi}(l_i),
\end{equation}
where
$$
\A_{\widetilde{H}_\varpi}(l)=\gamma_1(l(1))-\gamma_0(l(0))-\int^1_0 [l^*\lambda+\widetilde{H}_{\varpi}(l(t))\;dt].
$$
\end{thm}

\begin{proof} In local coordinates, $\sqrt{2}\sigma=dz=s+it$, we have
\begin{align*}
\widetilde{X}(\phi)&=X_{\re(W)+K} dt,\\
\widetilde{X}^{0,1}&=\frac{1}{2}(I+J\cdot j_\cC)\widetilde{X}(\phi)=\frac{1}{2}[X_{\re(W)+K}dt+\nabla_g(\re(W)+K)ds].
\end{align*}
Hence the equation (\ref{sec5-pert-equa-1}) is equivalent to
\begin{equation}
\phi_s +J(\phi_t)-\nabla_g(\re(W)+K)=0.
\end{equation}
By the above equation, we have
\begin{align*}
|d\phi-\widetilde{X}|^2&=|\phi_s ds+\phi_t dt-X_{\re(W)+K} dt|^2\\
&=\left(|\phi_s|^2+|\phi_t-X_{\re(W)+K} |^2\right)|\sigma|^2\\
&=2|\phi_s|^2|\sigma|^2=2|\phi_t-X_{\re(W)+K}|^2 |\sigma|^2.
\end{align*}
Hence, we have
\begin{align*}
E_K(\phi)&=\frac{1}{2}\int_\cC |d\phi-\widetilde{X}|^2 \nu_\cC\\
&=\int h(\phi_s, -J(\phi_t-X_{\re(W)+K}))ds\wedge dt\\
&=\int_\cC h(\phi_s, -J\phi_t)ds\wedge dt+\int_\cC h(\phi_s, \nabla_g (\re(W)+K) ds\wedge dt\\
&=\int_\cC \phi^*\om+\int_\cC d((\re(W)+K)\im(\sqrt{2}\sigma))\\
&=\int_\cC \phi^*\om+\int_{\pat\cC}  (\re(W)+K)\im(\sqrt{2}\sigma)\\
&=I+II.
\end{align*}
For the I-term, noticing that $\om=d\lambda$ and $\lambda|_{L_i}=d\gamma_i$, we have
\begin{align*}
I&=\sum^d_{j=0}\int_{C_j}\phi^*\lambda+\sum^d_{j=0}\phi^*\lambda\\
&=\gamma_0(l_0(0))-\gamma_0(l_d(0))+\sum^{d-1}_{j=1}\left(\gamma_j(l_j(0))-\gamma_{j}(l_{j-1}(1)) \right)+\gamma_d(l_d(1))-\gamma_d(l_{d-1}(1))\\
&+\sum^{d-1}_{j=0}\int_{l_j}\phi^*\lambda-\int_{l_d}\phi^*\lambda\\
&=\left( \gamma_d(l_d(1))-\gamma_0(l_d(0))-\int_{l_d}\phi^*\lambda\right)\\
&-\sum^{d-1}_{j=1}\left(\gamma_{j+1}(l_j(1))-\gamma_{j}(l_{j}(0))- \int_{l_j}\phi^*\lambda\right).
\end{align*}
For the II-term, using the fact that $\im(\sigma)|_{C_j}=0, \varepsilon_{\zeta_j}^*K(z,x)=\varpi(t, x), K(z, x)|_{C_j}\equiv 0$, then we have
\begin{align*}
II&=\sum^d_{j=0} \int_{l_j} (\re(W)+K)\im(\sqrt{2}\sigma)\\
&=\sum^{d-1}_{j=0} \int^1_0 \widetilde{H}_\varpi(l_j(t)) dt-\int^1_0 \widetilde{H}_\varpi(l_d(t))dt.
\end{align*}
Combining I and II, we get the conclusion.
\end{proof}

As the generalization of Theorem \ref{sec4-thm-ener-conv-expo-1}, we can obtain the following result:

\begin{thm}Let $\phi$ be a smooth solution of the equation (\ref{sec5-pert-equa-1}). Then the following three statements are equivalent:
\begin{itemize}
\item[(i)] $E_K(\phi)<\infty$
\item[(ii)] There exist solutions $l^\pm(t)\in S_{W}$ such that
$$
\lim_{s\to \pm\infty}\phi(s,t)=l^\pm(t).
$$
\item[(iii)] There exist constants $\delta>0$ and $c>0$ such that
$$
|\pat_s\phi(s,t)|\le ce^{-\delta|s|},
$$
for all $(s,t)\in \Sigma$.
\end{itemize}
\end{thm}



\section{Compactness and regularity}\label{sec-7}
\subsection{Basic settings and energy control}\

Let $\widetilde{X}(\phi)=X_{\re W}(\phi)\im \sigma+Y(\phi)$, the perturbed Witten equation (\ref{perturb-lg-poly}) becomes:
\begin{equation}\label{perturb-lg-poly-revised}
\bpat_J \phi=\widetilde{X}^{0,1}(\phi).
\end{equation}
Note that, if we take $(0,1)$ part of $X_{\re W}(\phi)\im \sigma$, it should be $-X_{R,\sigma}(\phi)$.

Choose a consistent admissible perturbation $K\in V^{\epsilon,\delta_0}_{\per}(\cC,\varpi)$, so we get a $Y$ such that $|Y|<\delta_0 |X_{R,\sigma}(\phi)|$ at any points of $\cC$ if $\delta_0<1/2$.

For any $a>0$, denote by $\Sigma^+_a=[a,\infty)\times [0,1]$ and $\Sigma^-_a=(-\infty, -a]\times [0,1]$. Let $\cC$ be a pointed-boundary Riemann surface. Define $\cC_{Inn,a}:=\cC\setminus \bigcup_{\zeta\in \cZ^\pm}\varepsilon_\zeta(\Sigma^\pm_a)$ and $\cC_{Inn}:=\cC_{Inn,0}$ is the interior part of $\cC$. On the strip-like ends, we consider the length $1$ band $B_a:=\varepsilon_\zeta([a, 1+a]\times [0,1])$ for $a>0$ and some $\zeta\in \cZ^\pm$.

Recall that in (\ref{sec-4-pert-defi-1}) we have defined the space of local perturbation term $\varpi$ of Hamiltonian equations.
Let
$$
\|\varpi\|_{ratio}=\sup_{t, x}\frac{|\nabla_g \varpi(t, x)|}{|\nabla_g \re (W)|}.
$$

\begin{lm}\label{polytope-lm-3} There exists a constant $C$ such that for any $\phi\in \W^{d+1}(\cC, L, l;K)$ and any $a>0$, the following holds
$$
 \int_{B_a}\big  |X_{R,\sigma}(\phi)|^2 \nu_{\mathscr{C}}\leq C.
$$
 \end{lm}

\begin{proof}
Without loss of generality, we can only consider $B=[0,1]\times[0,1]$.  Since $\phi$ satisfies (\ref{perturb-lg-r}) and the boundary condition $\phi([0,1]\times\{i\})\subset L_i$ for $i=0,1$, we have
\begin{align}
&\int_0^1|\nabla_g \re(W)(\phi)|^2dt=\int_0^1g(\nabla_g \re(W),\pat_s\phi+J\pat_t\phi-\nabla_g \varpi(\phi,t))dt\notag\\
\leq& \int_0^1|\nabla_g \re(W)(\phi)| |\partial_s \phi|dt-\int_0^1g(\nabla_g \im(W)(\phi),\partial_t \phi)dt+\int_0^1\|\varpi\|_{ratio}|\nabla_g\re(W(\phi))|^2 dt\notag\\
\leq&(\frac{1}{2} +\|\varpi\|_{ratio})\int_0^1 |\nabla_g \re(W)(\phi)|^2+\frac{1}{2}\int^1_0 |\pat_s \phi|^2 dt+\Delta_{01},
\notag
\end{align}
where $\Delta_{01}=\im W(L_0)-\im W(L_1)$. By the choice of $Y$, $\|\varpi\|_{ratio}\leq 1/4$. Hence by Theorem \ref{polytope-lm-2}, we obtain
\begin{equation}
\int_{B}|\pat_M W(\phi)|^2\nu_\cC<C\left(\int_{B} |\pat_s \phi|^2 \nu_\cC+\Delta_{01}\right)<C,
\end{equation}
where $C$ is independent of $\phi$, since the positions of those chords $l_i$'s appearing in (\ref{polytope-lm-2-1}) are located inside a compact set of $M$ which is independent of $\phi$.
\end{proof}

 \begin{lm}\label{polytope-lm-1} Let $\cC$ be a pointed-boundary Riemann surface corresponding to any given point $r\in\overline{\cR}$. There exists an universal constant $C$ independent of $r$ and any map $\phi\in \W^{d+1}(\cC, L, l;K)$ such that
$$
\int_{\overline{\cC_{Inn}}} |X_{R,\sigma}|^2 \nu_{\mathscr{C}}\leq C,
$$
\end{lm}

\begin{proof} Integrating the identity (i) of Lemma \ref{lg-poly-identity} over ${\cC_{Inn,1+a}}$ where $a\in [0,1]$, we have
\begin{align*}
\int_{{\cC_{Inn,1+a}}} (|\bpat_J \phi|^2+|X_{R,\sigma}(\phi)|^2)\nu_\cC &\le  \int_{{\cC_{Inn,1+a}}}|Y^{0,1}|^2 \nu_\cC+\int_{\pat {\cC_{Inn,1+a}}}\im(W(\phi)\sqrt{2}\sigma)\\
\end{align*}

Due to the choice of $Y$ and using the equation of $\phi$, we obtain
\begin{align*}
&\int_{{\cC_{Inn,1+a}}} |X_{R,\sigma}(\phi)|^2 \nu_\cC \le \sum^d_{j=0} \int_{l^a_j} \re(W(\phi))\im
(\sqrt{2}\sigma)\\
&+\sum^d_{j=0} \int_{l^a_j} \im (W(\phi))\re
(\sqrt{2}\sigma)+\sum^d_{j=0} \int_{C_j} \im (W(\phi))\re
(\sqrt{2}\sigma)\\
&\le \sum^d_{j=0} \int_{l^a_j} |\re(W(\phi))|\im
(\sqrt{2}\sigma)|+C,
\end{align*}
where $l^a_j$ are the boundary components and $C$ is an universal constant. Here we used the fact that $\im(\sigma)|_{C_j}=0$.

Now integrating the above inequality for $a\in [0,1]$, we obtain

\begin{equation*}\label{int-control-polytope-pre-1}
\int_{{\cC_{Inn,1}}} |X_{R,\sigma}(\phi)|^2 \nu_\cC\le \int^1_0 da \int_{{\cC_{Inn,1+a}}} |X_{R,\sigma}|^2 \nu\leq C+C\sum_j \int_{B_j}|\re W(\phi)|\nu_{\cC}.
\end{equation*}
By Lemma \ref{polytope-lm-3}, we get the conclusion.
\end{proof}

\begin{crl}\label{rem-1-energybound} Let $\cC^*$ be either $\cC_{Inn}$ or any length $1$ band $B$ and let $p>1$. There exists an universal constant $C$ such that for any $\cC$ in the DMS space $\overline{\cR}$ and any $\phi\in \W^{d+1}(\cC, L, l;K)$, the following estimates hold:
\begin{align}\label{p-control}
\int_{\cC^*} d^2(\phi, q_0)+|d\phi|^2 \le C,\; \int_{\cC^*} |\pat_M W(\phi)|^p \nu_{\cC}\le C,
\end{align}
\end{crl}

\begin{proof} By Lemma \ref{polytope-lm-3} and Lemma \ref{polytope-lm-1}, we know that
$$
\int_{\cC^*} |d\phi|^2+|X_{R,\sigma}|^2 \nu_{\cC}\le C
$$
Since $|\pat_M W(\phi)||\sigma|=|X_{R,\sigma}|$, by the tame condition of $W$ we have
$$
\int_{\cC^*} d^2(\phi, q_0)\nu_{\cC}\le C.
$$
Thus the $W^{1,2}$-norm of the function $d(\phi,q_0)$ is bounded. By John-Nirenberg inequality, for any $\alpha>0$ we have
$$
\int_{\cC^*} e^{2\alpha d(\phi,q_0)}\le C\exp\left(\alpha ||d(\phi,q_0)||^2_{W^{1,2}(\cC^*)} \right) \le C,
$$
By the tame condition of $W$, this implies that for any $p>1$,
$$
\int_{\cC^*} |\pat_M W(\phi)|^p \nu_{\cC}\le C.
$$
\end{proof}

\subsection{$C^0$ estimate}\

We first need some general elliptic estimates:
\begin{thm}\label{sect-5-regu-main-1} Let $U\subset \C$ be a bounded domain containing a connected segment $\pat_1 U$ of the boundary $\pat U$ and let $K\subset U$ be a compact set satisfying $\pat K\cap \pat U\subset\subset \pat_1 U$. Fix positive integers $k, n$, and a real number $p>2$. Let $\J_n\subset \R^{2n\times 2n}$ denote the set of complex structure on the vector space $\R^{2n}$. Then for every constant $c_0>0$ there exists a constant $c=c(c_0, K, U,n,k,p)>0$ such that the following holds. If $J\in W^{k,p}(U, \J_n)$ satisfies
$$
||J||_{W^{k,p}(U)}\le c_0,
$$
then every function $\phi\in W^{k+1,p}(U,\R^{2n})$ with the constraint $\phi(\pat_1 U)\subset \R^{n}\subset \C^n\cong \R^{2n}$ satisfies the inequality
\begin{equation}\label{sesct-5-regu-equa-1}
||\phi||_{W^{k+1,p}(K)}\le c\left(||\pat_s \phi+J\pat_t \phi||_{W^{k,p}(U)}+||\phi||_{W^{k,p}(U)}+||\phi||_{W^{1,\infty}(U)}\right).
\end{equation}
\end{thm}

\begin{proof} If $U$ is an open set and $\pat_1 U=\emptyset$, Theorem \ref{sect-5-regu-main-1} was proved in \cite[Lemma 3.3]{CGMS}. If $\pat_1 U\neq \emptyset$, a similar estimate for $k=1$ was proved in \cite[Theorem 8.3.5]{Oh1}. The conclusion is obtained by modifying slightly the proof in \cite[Lemma 3.3]{CGMS} by using the elliptic boundary estimate as shown in Lemma \ref{sect-5-regu-bdry-1} below, together with interpolation inequality.
\end{proof}

\begin{lm}\label{sect-5-regu-bdry-1}\cite[Lemma 2.2]{f1}
 (We denote by $D$ either the open unit disc or the half disc with boundary $\pat D = (- 1,1)$.)  Denote the standard Cauchy-Riemann operator $\bpat_0 =
(\pat /\pat x) + i(\pat /\pat y)$ for maps from $\C$ to $\C^n$.
For every $l>k,l-\frac{2}{q}>k-\frac{2}{p}>0$, there exists a constant $C$ such that if a compactly supported $\xi\in W^{k,p}(D,\C^n)$ with $\xi|_{\pat D}\subset \R^n$ and $\bpat_0 \xi \in W^{l-1,q}(D,\C^n)$, then $\xi\in  W^{l,q}(D,\C^n)$, and
$$
||\xi||_{W^{l,q}(D)}\le C||\bpat_0 \xi||_{W^{l-1,q}(D)}.
$$

\end{lm}

For $d=1$, let $\mathscr{C}$ denote the strip $\mathbb{R}\times [0,1]$ and $\mathscr{W}^{2}$ below should be treated as $\widetilde{\mathscr{W}}$ by convention, see Theorem \ref{thm:split} for notation. For $d\geq 2$, let $\mathscr{C}=\mathscr{C}^{d+1}$ be the universal curve over (uncompactified) Deligne-Mumford-Stasheff $\mathscr{R}=\mathscr{R}^{d+1}$, equipped with a choice of strip-like ends consistent with the choices for the universal curves for the smaller $d'$. The fibers of points in $\mathscr{R}$ near $\overline{\mathscr{R}}\backslash\mathscr{R}$ are discs modeled by gluing the strip-like ends from (several) fibers of $\mathscr{R}^{d'+1}$'s of smaller $d'$. For each such fiber, we remove the images of strip-like ends associated to the punctures and the glued regions in the interior of the disc coming from gluing the strip-like ends (or strips) due to interior stretching/degeneration (or breaking of strips); and the remaining subspace of $\mathscr{C}$ (consisting of ``thick'' part of each disc fiber) is a trivial fibration along the gluing parameter direction in $\mathscr{R}$ near $\overline{\mathscr{R}}\backslash \mathscr{R}$, and thus can be covered by a finite cover consisting of product open sets in $\mathscr{C}$ each of which is an open disc or half disc (the image of a standard disc or half disc under a holomorphic parametrization) in the fiber direction and an open set in $\mathscr{R}$ direction. On each fiber $\mathscr{C}_b, b\in \mathscr{R}$, via projecting onto the fiber direction, this induces a finite cover for the thick part of $\mathscr{C}_b$. The complement of thick part in $\mathscr{C}_b$, which consists of strips (via data of strip-like end or from inductive gluing from strip-like ends of $\mathscr{C}^{d'+1}_b$'s of smaller $d'$), can be covered by standard discs (half discs) with radii $1/2$ via the parametrizations induced by inclusions. Here, the total number in the covering for the glued strip regions becomes unbounded over the points in $\mathscr{R}$ near $\overline{\mathscr{R}}\backslash\mathscr{R}$).

\textbf{Notation 1} ($\{D'_{b,a}\}_{a\in A_b}$, $\{D_{b,a}\}_{a\in A_b}$): To summarize, we have an open cover $\{D^{+\epsilon}_{b, a}\}_{a\in A_b}$ for each fiber $\mathscr{C}_b, b\in \mathscr{R}$ that over the fiberwise complement of (glued) strip-like regions in $\mathscr{C}$ are induced from a finite collection of product open nerighborhoods in $\mathscr{C}$. Here $A_b$ is the index set for the open cover for $\mathscr{C}_b$. Denote the finite index subset $A^{\text{thick}}_b$ for open neighborhoods projected from the product neighborhoods. For each $a\in A^{\text{thick}}_b$, denote the holomorphic parametrization $\theta_{b, a}: B_{1/2+\epsilon}(0)\to D^{+\epsilon}_{b,a}$ from the standard disc (respectively mapping from standard half disc $ B_{1/2+\epsilon}(0)\cap \mathbb{H}$). For $a\in A_b\backslash A^{\text{thick}}_b$, we have $\theta_{b,a}: B_{1/2}(0)\to D_{b,a}$ induced from the inclusion (respectively, the half disc version). We can and will assume that this open cover is chosen fine enough such that $D'_{b,a}:=\theta_{b, a}(B_{1/8}(0)), a\in A_b$ (where for half discs this is understood as $\theta_{b,a}(B_{1/8}(0)\cap\mathbb{H})$) still covers $\mathscr{C}_b$. Denote (for $a\in A^{\text{thick}}_b$) $D_{b,a}:=\theta_{b,a}(B_{1/2}(0))$ (respectively $\theta_{b,a}(B_{1/2}(0)\cap\mathbb{H})$ for the half disc version). The main reason for $1/8$ radius shrinking will be clear in the proof for Lemma \ref{control-m-cases} below.

\textbf{Notation 2} ($\{g_b\}_{b\in\mathscr{R}}$, $C_c$, $d_b$): We can fix a consistent choice of fiberwise Riemannian metrics $\{g_b\}_{b\in\mathscr{R}=\mathscr{R}^{d+1}}$ on $\mathscr{C}=\mathscr{C}^{d+1}_b$ where $g_b$ restricted to strip-like ends and inductively glued regions is the natural standard metric; and for the remaining thick parts due to fiberwise compactness and trivial fibering along gluing parameter direction near $\overline{\mathscr{C}}\backslash\mathscr{C}$, any two such choices are comparable (by a fixed constant depending on the two choices $\{g_b\}$ and $d$). 
By the same reasoning (and precompactness of $B_{1/2}(0)$ in $B_{1/2+\epsilon}(0)$) (respectively the half disc version), we have a constant $C_c\geq 1$ such that $\frac{1}{C_c^2}g_0\leq (\theta_{b,a})^\ast g_b\leq C_c^2 g_0$ on $B_{1/2}(0)$ (respectively the same form for the half disc version), where $g_0$ is the standard metric. Here, $C_c$ depends on $d$ and the choice $\{g_b\}_{b\in \mathscr{R}^{d+1}}$ and the above cover whose dependence can be removed by going to a refinement. Denote the induced distance by $d_b$ on $\mathscr{C}_b$ and the (induced) standard distance $d_{\text{std}}$ on $B_{1/2}(0)$ (respectively for the half disc version). Here $d$ is the number from $\mathscr{C}^{d+1}$ and $d_b$ denotes the induced distance on the fiber $\mathscr{C}_b$ which will always bear a subscript $b$.


We will proceed to establish various estimates and bounds on Witten equation defined over this universal curve (one for each $d$), which will eventually lead to $C^0$ and $C^1$ bounds for the moduli spaces; and if transverse, $0$-dimensional moduli spaces will give rise to (structure constants of) the Fukaya category of LG model.

We first consider the control on the boundary:
\begin{lm}\label{poly-boundary-c_0}
For $(\phi:\mathscr{C}_b\to M)\in \W_\cC(K)=\W^{d+1}(\cC, L, l;K)$,
$$
\sup_{z\in \pat\mathscr{C}_b}d(\phi (z),q_0)<C,
$$
where the constant $C$  is independent on $\phi$.
\end{lm}
\begin{proof} Consider any $\phi$ with domain $\mathscr{C}_b$.

On each (half) disc, the equation for $\phi$ can be written as
$$
\pat_s \phi+J\pat_t \phi=\frac{1}{2}(X_{\re(W)}(\phi)\im(\sigma) + X_{\im(W)}(\phi) \re(\sigma))+Z(\phi),
$$
where $Z=Y^{0,1}$.
In complex coordinates $u^j=\phi^j+i\phi^{j+n}$, writing $\mu^j=Z^j+i Z^{j+n}$, we have
$$
\pat_\zb u^j= \sum_{\ib}h^{\ib j}\pat_\ib \overline{W}\sigma+\mu^j.
$$
On a (half) disc $D_{b,a}$, it is centered at $0$ with radius $1/2$ in the local coordinate. Choose a smooth function $\rho:\R^2\mapsto [0,1]$ with support in the open ball $B_{1/2}(0)$ and satisfying $0\leq \rho \leq 1$ and $\rho\equiv 1$ on $B_{1/8}(0)$. Let $W_{\rho}(s,t)=\rho W(\phi(s,t))$. It is a map from a (half) disc with radius $1/2$ to $\C\cong \R^2$. In complex coordinates, we have
\begin{align}
\bpat W_{\rho}=&\bpat \rho W+\rho \bpat (W\circ \phi)\notag\\
=&\bpat \rho W+\rho  \sum_j\pat_j W (\sum_{\ib}h^{\ib j}\pat_\ib \overline{W}\sigma+\mu^j).\notag
\end{align}
 As $|Y^{0,1}(\phi)|< \delta_0 |X_{R, \sigma}(\phi)|$, we have
 $$
 \|\bpat W_{\rho}\|_{L^{p}(D_{b,a})}\leq C(\|W\|_{L^{p}(D_{b,a})}+\|\pat_M W\|^2_{L^{2p}(D_{b,a})})
 $$
Since the image of $W_\rho$ lies on the line when restricted to the boundary, by Lemma \ref{sect-5-regu-bdry-1}, for any $p>2$, we have
\begin{align*}
\|W_{\rho}\|_{W^{1,p}(D'_{b,a}) }\leq&C(\|W\|_{L^{p}((D_{b,a})}+\|\pat_M W\|^2_{L^{2p}(D_{b,a})})\\
\leq&C(1+\|\pat_M W\|^2_{L^{2p}(D_{b,a})}).
\end{align*}
In the last inequality, we have used the tame condition of $W$.

The last expression is bounded by (\ref{p-control}). By Sobolev embedding theorem,  we have $|W\circ \phi| \leq C$ in $D'_{b,a}=\theta_{b,a}(B_{1/8}(0))$. By the fact that $\{D'_{b,a}\}_{a\in A_b}$ covers $\mathscr{C}_b$, we know that $|W\circ \phi|$ is bounded on $\mathscr{C}_b$, and in particular is bounded on $\pat \mathscr{C}_b$. However, $\phi$ maps pieces of $\pat \mathscr{C}_b$ to Lefschetz thimbles, so by Lemma \ref{w: lef-control-1}, we get the bound of $\phi$ on the boundary pieces.
\end{proof}

\textbf{Notation 3} ($m$): For $(\phi: \mathscr{C}_b\to M)\in\mathscr{W}^{d+1}(\mathscr{C}, L, l; K)$. With the above choice of $g_b$\footnote{This gives a norm for $T_z\mathscr{C}_b$ for all $z$; and using this and the norm from $h$ for tangent vectors on $M$, we can define the following quantity $|d\phi(z)|$.}, denote $$m:=m_\phi:=\max_{z\in\mathscr{C}_b}|d\phi(z)|,$$ which depends on $\phi$ but with the dependence suppressed for notational clarity. For a given $\phi$, it is clear that $m:=m_\phi<\infty$.

The next lemma gives a control of $|\pat_M W(\phi(z))|$ (thus $d(\phi,q_0)$ via the tame condition) in terms of $m$.

\begin{lm}\label{control-first-direction}
 For any $\nu>0$, there is a constant $C$ depends only on $q$ and the geometry of $M$ such that for any $(\phi: \mathscr{C}_b\to M)\in \mathscr{W}^{d+1}(\mathscr{C}, L, l; K)$ we have
 $$
\sup_{z\in \mathscr{C}_b}|\pat_M W(\phi(z))|\leq C(m+1)^{1/\nu}.
 $$
\end{lm}
\begin{proof}
Consider function $F^\nu(\phi)=|\pat_M W(\phi)|^\nu$,
$$|dF^\nu(\phi)|\leq \nu|\nabla_M^2 W(\phi)|\;|\pat_M W(\phi)|^{\nu-1}m\leq C\nu m(1+|\pat_M W(\phi)|^{\nu+1}),$$
where the last one is by tame condition. By (\ref{p-control})
$$
\int_{\Omega}\big (|dF^\nu(\phi)|^p+|F^\nu(\phi)|^p \big)\nu_{\mathscr{C}} \leq C(m^p+1),
$$
where the constant $C$ depends only on $p$ and $\nu$, but not $\Omega$. Here the exponent $\nu$ should not be confused with the area form $\nu_{\mathscr{C}}$. Fix a $p>2$, Sobolev embedding gives
$$
\max_{z\in \Omega}|F^\nu(\phi(z))|\leq C(m+1).
$$
 Thus the lemma is proved, as $\Omega$ runs over all $\{D'_{b,a}\}_{a\in A_b}$.
\end{proof}
The next step is to control $m$ in terms of $|\pat_M W(\phi(z))|$. Before doing so, we need some observation. By Lemma \ref{poly-boundary-c_0}, there exists a closed geodesic ball $M_R\subset M$ of radius $R$ centered at $q_0$ such that $\phi(\pat \Sigma)$ lies  in the ball of the same center with a smaller radius $R'$. By Weinstein's Lagrangian neighborhood theorem, there exists an open neighborhood $U_i$ of $V_i:=L_i\cap M_R$ in $M$ such that $$\{p\in M\;|\;d(p,L_i\cap M_{R'})<r_1 \}\subset U_i$$ for some positive $r_1$ as well as a symplectomorphism
 \begin{equation}\label{sec5: c-0-3}
\psi_{L_i}:U_i \mapsto T^*V_i, \psi_{L_i}|_{V_i }\equiv id.
 \end{equation}

Since the Lefschetz thimble $L_i$ is diffeomorphic to $\R^n$, we can regard $V_i$ as an open subset of $\R^n$ and regard $\psi_{L_i}$ as the map $(U_i,V_i)\mapsto (\C^n,\R^n)$ onto the image such that $\om=\psi_{L_i}^*\om_0$, where $\omega_0$ is the standard symplectic form in $\C^n$.


\textbf{Notation 4} ($r_0, r_1, r_2$): Let $r_2=\min\{r_0,r_1,1\}$, where $r_0$ is the injectivity radius of $M$ and $r_1$ is the number that works for all $L_i$ above.

\textbf{Notation 5} ($N$): For $(\phi:\mathscr{C}_b\to M)\in\mathscr{W}^{d+1}(\mathscr{C}, L, l; K)$, let $$N:=N_\phi:=\sup_{z\in\mathscr{C}_b}|\pat_M W(\phi(z))|,$$ which depends on $\phi$ but with the dependence suppressed for notational clarity. Note that for a fixed $\phi$, $N:=N_\phi$ is finite.

We have:

\begin{lm}\label{control-m-cases}
Let $\{D'_{b,a}\subset D_{b,a}\}_{a\in A_b, b\in\mathscr{R}}$ be a choice of cover in Notation 1 above. At least one of the following conclusions holds for some 
$\gamma>0$ and a constant $C>0$ which is independent of $\phi\in\mathscr{W}^{d+1}(\mathscr{C}, L, l; K)$ (mapping from a fiber $\mathscr{C}_b$ of the universal curve $\mathscr{C}$):

\begin{enumerate}

\item[(i)] $m\le C$;

\item[(ii)] $r_2^2\le C \frac{1}{ln(m)}$ and $m>C$ \footnote{Note that item (ii) can be put into the form of item (i), but we prefer to leave it in the current form to be more revealing to the logic structure of the argument in the proof.};

\item[(iii)] $m$ is achieved at $z_0$, there exists a disc chart $\theta_{b,a}: B_{1/2}(0)\to D_{b,a}$ mapping $x_0$ to $z_0$ and containing $B_{m^{-1}}(x_0)$ where the radius is with respect to the local coordinate, and still denoting $\phi\circ\theta_{b,a}$ by $\phi$, we have
\begin{equation}
\int_{B_{m^{-1}}(x_0)}|d\phi|^2 \le C (\frac{1}{ln(m)}+\frac{N^2+1}{m^{\gamma}});
\end{equation}

\item[(iv)] $m$ is achieved at $z_0$, there exists a half disc chart $\theta_{b,a}: B_{1/2}(0)\cap\mathbb{H}\to D_{b,a}$ mapping $x_0$ to $z_0$ and containing half disc $B_{m^{-1}}(z_1)\cap\mathbb{H}$ around $x_0$, and still denoting $\phi\circ\theta_{b,a}$ by $\phi$, we have
$$
\int_{B_{m^{-1}}(x_1)\cap\mathbb{H}} |d \phi|^2\le C(\frac{1}{ln(m)}+\frac{N^2+1}{m^{\gamma}}).
$$
\end{enumerate}
\end{lm}
To prove Lemma \ref{control-m-cases}, we need two types of isoperimetric inequalities. The first one is the linear isoperimetric inequality which holds for a symplectic vector space.

 \begin{lm}\label{sec5:c-0-sublemma-0}\cite[Lemma 4.4.3]{MS} Let $(V,\omega)$ be a symplectic vector space with a $\omega$-tame complex structure $J$ and a compatible inner product $g_J$. Let $\Lambda$ and $J\Lambda$ be Lagrangian subspaces of $V$. Then
  $$
  |A(\gamma)|\le \frac{1}{4\pi}L(\gamma)^2,
  $$
  for every smooth loop $\gamma:S^1 \to V$
  where
  $$
  A(\gamma):=\frac{1}{2}\int_{[0,2\pi]}\omega(\dot{\gamma}(\theta), \gamma(\theta))d\theta, \;\; L(\gamma):=\int_{[0,2\pi]}|\dot{\gamma}(\theta)|_{g_J}d\theta,
  $$
  and
  $$
  |A(\gamma)|\le \frac{1}{2\pi}L(\gamma)^2
  $$
  for every smooth path $\gamma:[0,\pi]\to V$, such that
  $\gamma(0),\gamma(\pi)\in \Lambda$, where
  $$
  A(\gamma):=\frac{1}{2}\int_{[0,\pi]}\omega(\dot{\gamma}(\theta), \gamma(\theta))d\theta, \;\; L(\gamma):=\int_{[0,\pi]}|\dot{\gamma}(\theta)|_{g_J} d\theta.
  $$
 \end{lm}

The isoperimetric inequality of the second type is the usual one without Lagrangian boundary constraints.

\begin{lm}\label{sec5:c-0-sublemma-1}\cite[Lemma 8.7.8]{Oh1}. Let $(M, J, g)$ be an almost complex manifold with a Riemannian metric $g$. If the metric $g$ is complete, has bounded curvature and its injectivity radius is bounded away from zero (with lower bound $r_0>0$), and if $J$ is uniformly continuous with respect to $g$, then the following holds: for all  $x\in M$, the exponential map $\text{exp}_x : B_{ r_0}(0)\subset T_x M \mapsto B_{ r_0}(x)\subset M $ is a diffeomorphism; every loop $\gamma$ in $M$ contained in some ball $B = B_r(x)$ with $r \leq r_0$ bounds a disc in $B$ of area less than $C\cdot length(\gamma)^2$ for some constant $C$.
 \end{lm}

\begin{proof} [\textbf{Proof of Lemma \ref{control-m-cases}}]


Consider any $\phi$ with the domain $\mathscr{C}_b$, a fiber of the universal curve $\mathscr{C}$. This determines $m=m_\phi$ and $N=N_\phi$.

Fix a number $\alpha\in(0,1/8)$.

We can assume $m>1$ and $m^{-1/2}+m^{-1}< m^{-1/2+\alpha}$; or (i) already holds by a constant independent of $\phi$. Indeed, if the first inequality holds and the second inequality does not hold, then $1\geq (m^{1/2})(m^\alpha-1)\geq m^\alpha-1$, thus $m\leq 2^{1/\alpha}$.

Since at the strip-like end associated $\zeta$ with the local coordinate $\phi_\zeta$ for $\phi$, we have $|\partial_s\phi_\zeta(s,t)|\to 0$ and $|\partial_t \phi_\zeta(s,t)|\to |\nabla_g \text{Re}(W)(l_\zeta(t))|=|\partial_M W(l_\zeta(t))|$. Denote $C_{\text{chords}}:=\max_{\zeta, t}|\partial_M W(l_\zeta(t))|$, which is also a constant independent of $\phi$.

Thus we define the constant\footnote{The constant $C$ in Lemma \ref{control-m-cases} is the maximum of $C_{\text{(i)}}$ and the last bounds $C$ in inqualities (\ref{integrationbound}) for cases (1) to (4) below, which are independent of $\phi$. Note that if items (i) and (ii) holds for $C_{\text{(i)}}$, then they hold for this bigger $C$.} in item (i) as $$C_{\text{(i)}}:=\max(1, 2^{1/\alpha}, C_{\text{chords}}, 4^4, 4^2, e^{\tilde C r_0^{-2}})=\text{max}(C_{\text{chords}}, 4^4, e^{\tilde C r_0^{-2}})$$ which is independent of $\phi$ (and the fiber $\mathscr{C}_b$), where the need for the these quantities will be apparent below and $\tilde C$ is the bound in inequality \ref{lastbound} below.

The conclusion in Lemma \ref{control-m-cases} has four alternatives, we call them items (i) to (iv); and below we have four cases to consider, we call them cases (1) to (4), to avoid confusion.

\textbf{Notation 6} ($z_0, d_0, x_1 , L_i$): Assume that item (i) does not hold, then in particular $m>C_{\text{chords}}$, thus the maximum in the definition of $m$ is realized on $\mathscr{C}_b$ (not its closure including chords at infinity) and denote the maximum point by $z_0$. In the case that $d_b(z_0,\partial\mathscr{C}_b)\leq\frac{1}{8C_c}$, here the distance between $z_0$ and the boundary of $ \mathscr{C}_b$ is defined using the distance $d_b$ defined via the metric on the fiber $\mathscr{C}_b$ we have chosen, we can define the following quantity $d_0$: We can choose $z_2\in\partial\mathscr{C}_b$ that realizes this distance, and choose a $D'_{b,a_0}=\theta_{b,a_0}(B_{1/8}(0)\cap\mathbb{H})$ that covers $z_2$. Then $D_{b,a_0}=\theta_{b,a}(B_{1/2}(0)\cap\mathbb{H})$ covers $z_0=\theta_{b,a_0}(x_0)$ for some $x_0$ with $d_{\text{std}}(x_0, 0)=|x_0|\leq C_cd_b(z_0, z_2)+d_{\text{std}}(\theta_{b,a_0}^{-1}(z_2), 0)< 2\cdot\frac{1}{8}=\frac{1}{4}$. Denote $x_0=:(s_0, d_0)$ in the local coordinate for $B_{1/2}(0)$ and denote $x_1:=(s_0, 0)$, then the closed half disc centered at $x_1$ of radius $1/4$ still lies in $B_{1/2}(0)\cap\mathbb{H}$ with the restricted holomorphic parametrization $\theta_{b,a_0}|_{{B_{1/4}(x_1)}\cap\mathbb{H}}$. Denote the thimble associated to the boundary component that realizes this distance by $L_i$.

We separate our proof into the following 4 cases:
\begin{enumerate}
\item $d_b(z_0, \partial\mathscr{C}_b)\leq\frac{1}{8C_c}$ and $m^{-1}\le d_0< m^{-1/2}$. Since we assume item (i) does not hold, by the definition of $C_{\text{(i)}}$ above, we have $m^{-1/4}<1/4$. Thus for $\alpha\in (0,1/8)$, we have $m^{-1/2+2\alpha}<1/4$, and $B_r(x_1)\cap \mathbb{H}\subset B_{1/2}(0)\cap\mathbb{H}$ for all $r\in [m^{-1/2+\alpha}, m^{-1/2+2\alpha}]$. Still denote $\phi\circ\theta_{b,a_0}$ by $\phi$ below.


Define
$$
f(r):=\int_{B_r(x_1) \cap \mathbb{H}}|d\phi|^2\nu_{\mathscr{C}} \quad\text{and}\quad l(r):=\int_{\pat B_r(x_1) \cap \mathbb{H}}|\frac{\pat \phi}{\pat \tau}|d \tau.
$$

 A direct calculation gives
\begin{equation}\label{c-0-derivative-polytope}
f'(r)\geq\frac{l(r)^2}{\pi r}.
\end{equation}
Letting $l_{\text{min}}:=\min_{r\in [m^{-1/2+\alpha},m^{-1/2+2\alpha}]} l(r)$ and integrating the above inequality, we have
\begin{equation}\label{delta-energy}
f(m^{-1/2 +2\alpha})-f(m^{-1/2 +\alpha})\geq\frac{\alpha}{\pi}l_{\text{min}}^2 ln(m).
\end{equation}
On the other hand, by Lemma \ref{polytope-lm-1} and Lemma \ref{polytope-lm-2} we know that the integration of $|d\phi|^2\nu_{\mathscr{C}}$ on such a disc with bounded radius is bounded (by a constant only depending on the already assigned boundary conditions and chords), so $f(m^{-1/2 +2\alpha})\leq C$. Combining this with (\ref{delta-energy}), we know
\begin{equation}\label{arc-bound}
l_{\text{min}}^2\le C \frac{1}{ln(m)}
\end{equation}
for a constant $C$ independent of $\phi$.

If $C \frac{1}{ln(m)}\geq r_2^2$, then the conclusion (ii) holds. Hence we assume that
\begin{equation}
C \frac{1}{ln (m)}< r_2^2
\end{equation}
holds. In this case, by (\ref{arc-bound}) we have $l_{\text{min}}< r_2\leq r_1$.

By definition, that the minimum $l_{\text{min}}$ is attained on the arc $\hat{a}_{\text{min}}(\tau)=x_1+m^{-1/2+\beta}e^{i\tau}$ with $\tau\in [0,\pi]$, for some $\beta\in [\alpha,2\alpha]$.
By construction, we have $\phi(\hat{a}_{\text{min}}(0))\in \phi(\partial\mathscr{C}_b)\subset L_i\cap M_{R'}$, and the image of $\phi\circ\hat{a}_{\text{min}}$ lies in the Weinstein neighborhood $U_i$ of $L_i\cap M_R$, the domain of the symplectomorphism $\psi_{L_i}: (U_i, L_i\cap M_R)\to (T^*\R^n\cong\C^n, \R^n)$ onto the image. The two ends of $\phi\circ\hat{a}_{\text{min}}$ lie on $L_i$, thus the curve $\psi_{L_i}\circ\phi\circ\hat{a}_{\text{min}}$ in $\C^n$ has its two ends landing on $\R^n$. We define a new map
$\varphi: B_{m^{-1/2+\beta}}(x_1)\cap \mathbb{H} \mapsto \C^n$ by
\begin{align*}
x_1+rm^{-1/2+\beta}e^{i\tau}\mapsto r(\psi_{L_i}(\phi(\hat{a}_{\text{min}}(\tau)))),
\end{align*}
which is scaling towards $0$ rather than $x_1$ (recall that $\phi$ here is $\phi\circ\theta_{b,a_0}$) and it is a half disc with boundary in $\R^n$.


Since $\om$ is exact on $M$, $\lambda|_{L_i}=dh_0$ is exact for some (and any) primitive $\lambda$ of $\om=d\lambda$ (as the thimble $L_i$ is contractible), and the ends of $\psi_{L_i}\circ \phi\circ\hat{a}_{\text{min}}$ lie on $\R^n$, by applying Lemma \ref{sec5:c-0-sublemma-0}, we have
\begin{align*}
&\int_{B_{m^{-1/2+\beta}}(x_1)\cap\mathbb{H}} \phi^* \om\\
=&\int_{B_{m^{-1/2+\beta}}(x_1)\cap \mathbb{H}} (\psi_{L_i}^{-1}\circ\varphi)^* \om+\int_{\partial (B_{m^{-1/2+\beta}}(x_1)\cap \mathbb{H})} (\phi^*\lambda|_{L_i}-(\psi_{L_i}^{-1}\circ\varphi)^* \lambda|_{L_i})\\
=&\int_{B_{m^{-1/2+\beta}}(x_1)\cap \mathbb{H}} (\psi_{L_i}^{-1}\circ\varphi)^* \om+\int_{B_{m^{-1/2+\beta}}(x_1)\cap \partial\mathbb{H}} (\phi^*\lambda|_{L_i}-(\psi_{L_i}^{-1}\circ\varphi)^* \lambda|_{L_i})\\
=&\int_{B_{m^{-1/2+\beta}}(x_1)\cap \mathbb{H}} \varphi^* \om_0 +\int_{-\hat a_{\text{min}}(0)\cup\hat a_{\text{min}}(\pi)} (\phi^*h_{L_i}-\phi^*h_{L_i}) \\
=&\int^1_0\int^\pi_0 \om_0(r\frac{d (\psi_{L_i}(\phi(\hat{a}_{\text{min}}(\tau))))}{d\tau}, \psi_{L_i}(\phi(\hat{a}_{\text{min}}(\tau))))d\tau dr\\
\le& \int^1_0 rdr \left(\frac{1}{2\pi}\left(\int^\pi_0 |\frac{d (\psi_{L_i}(\phi(\hat{a}_{\text{min}}(\tau))))}{d\tau}|_{g_0} d\tau\right)^2\right)\\
\le & \frac{1}{4\pi}|d\psi_{L_i}|^2_{C^0} \left( \int^\pi_0 |\frac{d (\phi(\hat{a}_{\text{min}}(\tau)))}{d\tau}|_{g} d\tau\right)^2\\
\le & Cl_{\text{min}}^2.
\end{align*}

So by (ii) of Lemma \ref{lg-poly-identity},
\begin{align*}
&\int_{B_{m^{-1/2+\beta}}(x_1)\cap \mathbb{H}}|d\phi|^2 \nu_{\mathscr{C}}\\
\leq& \int_{B_{m^{-1/2+\beta}}(x_1)\cap \mathbb{H}}(\phi^*\omega+|Y|^2 \nu_{\mathscr{C}})+\int_{\pat(B_{m^{-1/2+\beta}}(x_1)\cap \mathbb{H})}\im(W(\phi)\cdot\sqrt{2}\sigma)\\
\leq&C\big(l_{\text{min}}^2+\int_{B_{m^{-1/2+\beta}}(x_1)\cap \mathbb{H}}N^2+\int_{\pat(B_{m^{-1/2+\beta}}(x_1)\cap \mathbb{H})}(N^2+1)\big).
\end{align*}
Here we use the tame condition to get $|W|\leq C(|\pat_M W|^2+1)$. So we have
\begin{align}
&\int_{B_{m^{-1/2+\beta}}(x_1)\cap \mathbb{H}}|d\phi|^2 \nu_{\mathscr{C}} \nonumber\\
\leq &C(l_{\text{min}}^2+(N^2+1)(m^{4\alpha-1}+m^{2\alpha-\frac{1}{2}})) \nonumber\\
\leq &C(l_{\text{min}}^2+(N^2+1)m^{2\alpha-\frac{1}{2}}) \nonumber \\
\leq &C(\frac{1}{ln(m)}+(N^2+1)m^{2\alpha-\frac{1}{2}}) \label{integrationbound},
\end{align}
where $C$ on each line is a new $C$, the second inequality is by $m>1$ and $2\alpha-\frac{1}{2}<0$, and the last inequality is by (\ref{arc-bound}).

By the setting in this case, the distance between $x_0$ and $x_1$ in the local chart is $d_0$ and which is at most $m^{-1/2}$ in case (1), and $x_0$ and $x_1$ have the same horizontal coordinate. Since we assume that item (i) does not hold, by definition of $C_{\text{(i)}}$, we have $m^{-1/2}+m^{-1}<m^{-1/2+\alpha}\leq m^{-1/2+\beta}$. Therefore, a ball of radius $m^{-1}$ centered at $x_0$ is contained in the half disc $B_{m^{-1/2+\beta}}(x_1)\cap\mathbb{H}$. So item (iii) is satisfied as we choose $\gamma=\frac{1}{2}-2\alpha$.


\item $d_b(z_0,\partial\mathscr{C}_b)\leq\frac{1}{8C_c}$ and $d_0<m^{-1}$.\
Similar to the previous case, we choose $x_1$ as above, and then the half disc centered at $x_1$ with radius $m^{-1}$ contains $x_0$ (that corresponds to $z_0$) and is contained in the half disc $B_{m^{-1/2+\beta}}(x_1)$. The integration bound (\ref{integrationbound}) restricted to $B_{m^{-1}}(x_1)\cap \mathbb{H}$ gives rise to item (4).


\item $d_b(z_0,\partial\mathscr{C}_b)\leq\frac{1}{8C_c}$ and $d_0\geq m^{-1/2}$. Since $d_0\leq |x_0|<1/4$ by construction, the full disc $B_{m^{-1/2}}(x_0)$ lies in $B_{1/2}(0)\cap\mathbb{H}$ and is away from the boundary. The remaining argument will proceed as in the next case.

\item $d_b(z_0,\partial\mathscr{C}_b)>\frac{1}{8C_c}$.\
Then $z_0$ lies in a whole disc chart $D'_{b,a_0}=\theta_{b,a}(B_{1/8}(0))$ in the cover, where $z_0=\theta_{b,a_0}(x_0)$. As we assume item (i) does not hold, by the definition of $C_{\text{(i)}}$, we have $m^{-1/2}<\frac{1}{4}$, so $B_{m^{-1/2}}(x_0)$ lies entirely in the local chart $D_{b,a_0}=\theta_{b,a}(B_{1/2}(0))$ and away from the boundary.

Now both cases (3) and (4) proceed as follows: By the same argument as in (\ref{c-0-derivative-polytope}) and (\ref{delta-energy}) but integrating over an annulus centered at $z_0$ with the radius $m^{-1}\leq r\leq m^{-1/2}$, we can find a $\beta\in [0,1/2]$ such that
\begin{align}
l_{\text{min}}^2=(\int_{\pat(B_{m^{-1+\beta}}(x_0))}|\frac{d\phi}{d\tau}|d\tau)^2<\tilde C\frac{1}{ln (m)}.\label{lastbound}
\end{align}

By the choice of $C_{\text{(i)}}$ and as item (i) does not hold, $l_{\text{min}}<(\frac{\tilde C}{ln(m)})^{1/2}<r_0$; so the image of $\partial B_{m^{-1}+\beta}(x_0)$ under $\phi$ lies in a geodesic ball of radius $r_0$ in $M$, and it spans a disc $\varphi':D\to M$ inside this geodesic ball (e.g. using geodesics to connect from a point on the loop to other points).
We have the control of the symplectic area:
$$
\int_{B_{m^{-1+\beta}}(x_0)} \phi^*\om=\int_{D}(\varphi')^\ast \omega\le C\frac{1}{ln (m)},
$$
where the first equality uses exactness\footnote{There is another argument which does not use the exactness of $\omega$ but gives a weaker $\gamma$. Both work for our purpose.} of $\omega$, and the inequality uses the isoperimetric inequality in Lemma \ref{sec5:c-0-sublemma-1}, together with the above estimate of $l_{\text{min}}^2$ in inequality (\ref{lastbound}).

The rest argument is exactly the same as case (1), and item (iii) follows with $\gamma=\frac{1}{2}$.
\end{enumerate}
\end{proof}

\begin{lm}\label{anti-control}
There exists some $k>0$ such that
\begin{equation}
m\le C(N^k+1)
\end{equation}
holds. The constant  $C$ depends only on $l:=\{l_{\xi}\}$ involved in $\W^{d+1}(\cC, L, l;K)$, and bounds in the bounded geometry of $M$.
\end{lm}

\begin{proof}
By Lemma \ref{control-m-cases}, at least one of items (i)-(iv) holds. If items (i) or (ii) holds, then this lemma holds automatically.

Now we assume that item (iii) holds. Then we have $\theta_{b,a}(B_{m^{-1}}(x_0))\subset D_{b,a}$ from a local chart $D_{b,a}$ in the chosen cover, such that $m$ is attained at the point $z_0=\theta_{b,a}(x_0)$.

Denote $r_c:=\frac{r_2}{C_c}\leq 1$. We have $|d\theta_{b,a}|\leq C_c$ using $g_b$ on $\mathscr{C}_b$ and the standard metric in $B_{1/2}(0)$. Define the rescaled map $$\hat{\phi}: B_{r_c}(0)\subset \C\to M,\; \; \hat{\phi}(z):=\phi(\theta_{b,a}(x_0+z/m)).$$ $$d(\hat\phi(z),\hat\phi(0))\leq \sup_{z\in B_{r_c}(0)}|(d\hat\phi)(z)|\;|z|\leq m\cdot C_c\cdot \frac{1}{m}\cdot \frac{r_2}{C_c}\leq r_2.$$

Denote $p_0:=\hat\phi(0)=\phi(z_0)$. Since $r_2$ is no greater than the injectivity radius $r_0$ of $M$, the image of $\hat{\phi}$ lies in the normal coordinate ball $B_{r_2}(p_0)$. $\hat{\phi}$ satisfies the following equation:
\begin{equation}\label{comp-fina-1}
\hat{\phi}_s+J(\hat{\phi})\hat{\phi}_t=\frac{G}{m},
\end{equation}
where $G:=\widetilde{X}(\phi\circ\theta_{b,a})(\frac{\pat}{\pat s})$. Denote by $\xi(s,t)=\hat{\phi}_s(s,t)$ and $\eta(s,t)=\hat{\phi}_t(s,t)$.

We can choose an orthonormal basis $\{e_i(p_0),i=1,\cdots,2n\}$ of the tangent space $T_{p_0}M$ such that
$$
J(p_0)(e_1(p_0),\cdots, e_{2n}(p_0))=(e_1(p_0),\cdots, e_{2n}(p_0))J_0,
$$
where $J_0=\begin{pmatrix} 0&-I_n\\I_n &0\end{pmatrix}$ is the standard complex structure in $\R^{2n}$. Using the parallel transport along the geodesic radial direction of $B_{r_2}(p_0)$, we obtain the orthonormal basis $\{e_i(p),i=1,\cdots, 2n\}$ of $T_pM$, for all $p\in B_{r_2}(p_0)$. Since $J$ is parallel with respect to the connection $\nabla$, we have
 $$
J(p)(e_1(p),\cdots, e_{2n}(p))=(e_1(p),\cdots, e_{2n}(p))J_0,\forall p\in B_{r_2}(p_0).
$$
Let $\xi=\sum_i \xi^i e_i, \eta=\sum_i \eta^i e_i$ and $G=\sum_i G^i e_i$, the equation (\ref{comp-fina-1}) becomes the component equation:
\begin{equation}\label{comp-fina-2}
\xi^i+(J_0\cdot \eta)^i=\frac{G^i}{m}, i=1,\cdots,2n,
\end{equation}
here $J_0\cdot \eta:=J_0\cdot(\eta^i)=(\sum_j(J_0)^i_j\eta^j)$. Taking the derivative $\pat_s$ of the equation (\ref{comp-fina-2}), we have
\begin{equation}\label{comp-fina-3}
\pat_s \xi^i +(J_0\cdot \pat_t \xi)^i=\frac{(\nabla_\xi G)^i-\sum_{k,j}\Gamma^i_{kj}(\hat{\phi})G^k\xi^j}{m}, i=1,\cdots, 2n.
\end{equation}

By the definition of $G$ and the choice of $\varpi$, we have
$$
\sqrt{\sum_i \left(\frac{(\nabla_\xi G)^i-\sum_{k,j}\Gamma^i_{kj}(\hat{\phi})G^k\xi^j}{m}\right)^2}\le \frac{|\nabla_\xi G|+C|G||\xi|}{m}\le C\frac{N^2+N+1}{m}|\xi|,
$$
here we have used $|\Gamma^i_{kj}|\leq C$ from the bounded geometry of $(M,h)$.

Now by applying the standard elliptic estimate for a disc domain to the Cauchy-Riemann equation (\ref{comp-fina-3}), for a fixed $p>2$, we have
 \begin{align*}
 \|\xi\|_{W^{1,p}(B_{r_{c}/2}(0))}\leq C(\frac{N^2+N+1}{m}\| \xi\|^{2/p} _{L^2(B_{r_{c}}(0))}+ \|\xi\|^{2/p}_{L^{2}(B_{r_{c}}(0))}).
 \end{align*}
Here we used the fact that $\frac{1}{C_c}|\xi|\leq 1$ (thus $\frac{|\xi|^p}{C_c^p}\leq \frac{|\xi|^2}{C_c^2}$, so $|\xi|^p\leq C_c^{p-2} |\xi|^2$) in the above inequality.
So
 $$
 \|\xi\|_{W^{1,p}(B_{r_{c}/2}(0) )}\le C(\frac{N^2+1}{m}+1)\| \xi\|^{2/p} _{L^2(B_{r_{c}}(0))}.
 $$
 By Sobolev embedding theorem and (iii) of Lemma \ref{control-m-cases}, we obtain
 $$
|\xi(0)|\le C(\frac{N^2+1}{m}+1)(\frac{1}{ln(m)}+\frac{N^2+1}{m^{\gamma}})^{1/p}.
 $$
 Note that by the equation for $\phi$ we know $|\eta|\le|\xi|+CN/m$ and $$|d\hat{\phi}(0)|=1\le |\xi(0)|+|\eta(0)|.$$ We have
 \begin{equation}\label{longestimate}
 1\le C\left(\frac{N}{m}+(\frac{N^2+1}{m}+1)(\frac{1}{ln(m)}+\frac{N^2+1}{m^{\gamma}})^{1/p}\right).
 \end{equation}
 This implies (by arguing by contradiction) that there exists some $k>0$ and $C$ which are independent of $N$ and $m$ (in particular, independent of LG solutions $\phi$ with which $N$ and $m$ are associated) such that
 \begin{equation}
  m \le C(N^{k} +1).
 \end{equation}
So we have proved this lemma in the case (iii).

Now assume that item (iv) in Lemma \ref{control-m-cases} holds. Denote $r_c:=\frac{r_2}{C_c}\leq 1$. Define the rescaled map $$\hat\phi: B_{r_c}(0)\cap\mathbb{H}\to M,\;\; \hat\phi:=\phi(\theta_{b,a}(x_1+z/m)).$$ Denote $p_0=\hat\phi(0)=\phi(\theta_{b,a}(x_1))$, where we recall Notation 6. We define the vector fields $\xi$ and $\eta$ as in the discussion for item (iii) above. 
Then the image $\hat{\phi}(B_{r_c}(0))$ is contained in the normal coordinate ball $B_{r_2}(p_0)\subset M$. We pick an orthonormal frame $\{e_1, \cdots, e_n\}$ for $TL|_{L\cap B_{r_2}(p_0)}$, this extends to a frame $\{e_1, \cdots, e_n, Je_1, \cdots, Je_n\}$ for $TM|_{L\cap B_{r_2}(p_0)}$, then parallel transport it along normal geodesics into a frame for $TM|_{U(L\cap B_{r_2}(p_0))}$ where $U(L\cap B_{r_2}(p_0))$ is a fiber bundle over $L\cap B_{r_2}(p_0)$ where each fiber is the image under the exponential map of normal vectors to $L$ of length less than $r_2$. This frame serves the same purpose as in the previous case for item (iii), except they are not orthonormal frame anymore (only orthonormal along each fiber of $U(L\cap B_{r_2}(p_0))$). By construction, the equation (\ref{comp-fina-1}) has the form (\ref{comp-fina-2}) with $\xi(\pat \mathbb{H}\cap B_{r_c}(0))\subset \text{span}(e_1,\cdots, e_n)\cong\R^n$. Taking the derivative $\pat_s$ to the equation (\ref{comp-fina-2}), we obtain the equation (\ref{comp-fina-3}). Now we can follow almost the same route as in item (iii) to get the conclusion. However, there are two different points in the proof. The first one is that we must use Lemma \ref{sect-5-regu-bdry-1} to treat the boundary estimate. The second one is that the basis we took is not orthonormal basis anymore, hence the constant appeared in the estimate may depend on the choice of the holomorphic coordinate chart. However, by Lemma \ref{poly-boundary-c_0} and the setting of the half disc we choose, the image $\hat{\phi}(B_{r_c}(0)\cap \mathbb{H})$ is always contained in a compact set $K$ which only depends on the geometry of the tame system and the energy of $\phi$. By the geometry of the tame system, we mean the bounded geometry of $(M,h)$ together with the tame condition coming from a tame system $(M,h,W)$ under consideration. Hence there are only finitely many holomorphic charts whose union covers $K$. Hence we can take a uniform constant $C$ in the inequality (\ref{anti-control}). Thus we proved the lemma in case (iv).

Finally, we can choose a common $k$ and $C$ for all cases, namely, define $k'$ to be the maximum of $k$ chosen in all cases and let $C':=2C$. By considering whether $N\leq 1$ or not, we see that $m\leq C'(N^{k'}+1)$ in all cases.
\end{proof}

Combining Lemma \ref{control-first-direction}, Lemma \ref{anti-control} and the tame condition, we have
\begin{thm}\label{c_0_polytope}
We have
$$
d(\phi,q_0)<C,\;\;|d\phi|<C,
$$
For all $\phi$  in $\W^{d+1}(\cC, L, l;K)$, where the  The constant  $C$ depends only on Hamiltonian chords $l:=\{l_{\xi}\}$ involved in $\W^{d+1}(\cC, L, l;K)$, and bounds in the bounded geometry of $M$.
\end{thm}

\subsection{Compactness}\label{cpt}
\subsubsection*{Elliptic regularity}\

\begin{lm} \label{regular}Let $p>2$.  If $\phi$ is a locally $W^{1,p}$-solution of the perturbed Witten equation (\ref{perturb-lg-poly}), with the boundary condition $\phi|_C\in L_C, \forall{\; component }\;C\subset \pat \cC$. Suppose that both $d(\phi,q_0)$ and $|d\phi|$ are bounded, then $\phi$ is in the space $C^\infty(\Sigma)\cap C^0(\overline{\Sigma})$.
\end{lm}

\begin{proof} Since $p>2$, any locally $W^{1,p}$-solution $\phi$ is continuous. We can choose a small neighborhood $U_z$ of $z\in \mathscr{C}$ and a coordinate neighborhood $V_{\phi(z)}$ of $\phi(z)\in M$ such that $\phi(U_z)\subset V_{\phi(z)}$. Then we can choose a holomorphic coordinate transformation $\Psi: V_{\phi(z)}\to \C^n$ such that $\Psi\circ J=i\Psi$ and $\Psi\circ \phi(\overline{U_z}\cap \pat \mathscr{C})$ is mapped to a segment of a bounded domain in $\C^n$.  Now we apply the Theorem \ref{sect-5-regu-main-1} to $\Psi\circ \phi$. Note that in this case the term $||\Psi\circ\phi||_{W^{1,\infty}}$ in the estimate (\ref{sesct-5-regu-equa-1}) disappears because $J=i$ is constant now. So we have the estimate
\begin{align*}
||\Psi\circ\phi||_{W^{2,p}(K_z)}&\le c(||\nabla_g \widetilde{H}_{\varpi}(\Psi\circ \phi)||_{W^{1,p}(U_z)}+||\Psi\circ \phi||_{W^{1,p}(U_z)})\\
&\le c(||\widetilde{H}_{\varpi}||_{C^2((\phi(U_z))}+1) (||\Psi\circ \phi||_{W^{1,p}(U_z)}),
\end{align*}
where $K_z$ is a smaller compact subset of $U_z$.

Then using a bootstrap argument, we can prove that $\phi$ is smooth at any point $z\in \mathscr{C})$.
\end{proof}

We need a lemma which is similar to Theorem B.4.2 in \cite{MS} to study the convergence of solutions. We rewrite it and include the proof for completeness.

\begin{lm}\label{convergence}
Suppose that $\Omega$ is a bounded open set in the half plane $\Hb$ and $T$ is a positive number. Let $\{v_k\}$ be a sequence of maps from $\Omega$  to $\C^n$ with $v_k(\Omega\cap \Hb)\subset L_0$ (or $L_1$) and $|v_k|_{C^1(\Omega)}<C$ for some $C>0$, $k=1,2\cdots$.
Let $\{h_k\}$ be a sequence of smooth maps from $\C^n\times [0,T]$ to $\C^n$ with $h_k\to h$ with all derivatives, for some map $h$ on $B_C(0)\times  [0,T]$.

If $v_k$ satisfies:
\begin{equation}
\partial_{s}v_k+J\partial_t v_k+h_k(v_k,t)=0,
\end{equation}
then on any compact set $K\subset\Omega$, $\{v_k\}$ have a subsequence converging with all derivatives to some $v$ on $K$, with
\begin{equation}\label{convergence-1}
\partial_s v+J \partial_t v+  h(v,t)=0.
\end{equation}
\end{lm}
\begin{proof}
Since $v_k$ have uniform $C^1$-bound, by Arzela-Ascoli theorem we can get a subsequence, still denoted by $\{v_k\}$, which converges to some $v$, with $\|v\|_{C^{\alpha}(K)}<\infty$ for some $\alpha>0$.

For each point $z\in \Omega$, there is neighborhood $V_z$ of $v(z)$ and a coordinate map $\varphi_z: V_z\to \C^n$ such that $\varphi_z(V_z\cap L_0)\subset \R^n$ (or $\varphi_z(V_z\cap L_1)\subset \R^n$) if it is not empty. By the convergence of $\{v_k\}$ and the uniform boundedness of $dv_k$,  there is a neighborhood $U_z$ of $z$ in $\Omega$ such that $v_k(U_z)\subset V_z$ if $k$ is large enough. Choose a finite covering  $\{U_1,\cdots,U_l\}$ of $K$ from these $\{U_z|z\in \Omega\}$. Then for sufficiently large $k$, we have $v_k(U_j)\subset V_j$. Let $\varphi_j: V_j\to \C^n$ be the corresponding coordinate map. Define $\widetilde{v}^j_k=\varphi_j\circ v_k$. We have
\begin{equation}
\partial_s \widetilde{v}^j_k+i\partial_t \widetilde{v}^j_k+ \tilde{h}_k(\widetilde{v}^j_k,t)=0
\end{equation}
on $U_j$, where $\tilde{h}_k=(\varphi_j)_* h_k\circ \varphi^{-1}_j$.

Since all the derivatives of $h_k$ have uniform upper bound and $v_k$ has uniform $C^1$-bound, we know that $ \tilde{h}_k(\widetilde{v}^j_k,t)$ has a uniform $C^1$-bound. By Schauder estimate of elliptic operators, $\widetilde{v}^j_k$ has a uniform $C^{1,\alpha}$-estimate for $0<\alpha<1$. By bootstrap argument, we can prove that $\widetilde{v}^j_k$ has a uniform $C^k$-estimate in $U_j$ for any $k\in \mathbb{N}$.  Hence, there exists a subsequence of $\widetilde{v}^j_k$ which converges to some  $\tilde{v}^j\in C^k(U_j,\C^n)$ for any given $k$. $\tilde{v}^j$ satisfies the equation
$$
\partial_s \widetilde{v}^j+i\partial_t \tilde{v}^j+ \tilde{h}(\tilde{v}^j,t)=0,
$$
where $\tilde{h}=(\varphi_j)_* h\circ \varphi^{-1}_j$.

Clearly, $v:=\varphi^{-1}_j \circ \tilde{v}^j$ satisfies (\ref{convergence-1}).  Thus the lemma is proved.
\end{proof}

\begin{thm}\label{d control} We have $\|\phi\|_{C^k(\Omega)}<C$ for any $\phi\in \W^{d+1}(\cC, L, l;K)$, where the constant  $C$ depends only on Hamiltonian chords $\{l_{\xi}\}$ involved in $\W^{d+1}(\cC, L, l;K)$, and the geometry of $M$.
\end{thm}

\begin{proof} By Theorem \ref{c_0_polytope}, we can choose a compact subset $K$ of $M$, such that the image of $\phi $ is contained in $K$.   Choose a finite covering  of $K$ consisting of the coordinate neighborhoods $\{U_1,\cdots,U_l\}$. Let $\varphi_i :U_i\to \C^n$ be the coordinate maps such that $\varphi_i(U_i\cap L_0)\subset \R^n$ ( or $\varphi_i(U_i\cap L_1)\subset \R^n$) in case that they are not empty. Since we have uniform $C^1$-bound for any solutions, then there exists a $r>0$ such that for any $z\in\Sigma$ and $\phi\in \W^{d+1}(\cC, L, l;K)$, $\phi(B_{r}(z)\cap \Sigma)$ is contained in some $U_i$. Then we can apply Lemma\ref{convergence} to get the $C^k$-bound for any $k\in \mathbb{N}$.
\end{proof}

\begin{thm}\label{limit}
If $\phi$ satisfies the same condition in Lemma \ref{regular}, then on any strip-like ends, $\phi\circ \ve_\zeta(s,)\to l$, for some Hamiltonia chord $l$, as $s\to +\infty \;(-\infty)$ (depending on the direction of the ends). In particular,
\begin{equation}
\tM(\varpi)=\bigcup_{l^-,l^+\in S_{W}}\tM(l^-,l^+;\varpi).
\end{equation}
\end{thm}

\begin{proof}
Without loss of generality, assume that $\phi$ is defined on some $(a,+\infty)\times [0,1]$, after some local chart map on the disc.
Define $\phi_{\tau}(s,t)=\phi(s+\tau,t)$ for $\tau>a$. $\phi_k$ and its first order derivatives are uniformly bounded on $[0,1]\times[0,1]$. So by Lemma \ref{convergence}, we can get a subsequence $\{\phi_{k_j}\}$, converging to some $\phi^+$ in $C^1$-topology. $\phi^+$ should satisfy the equation
\begin{equation}
\partial_s\phi^++J \partial_t\phi^+-\nabla_g\widetilde{H}_{\varpi}(\phi^+,t)=0.
\end{equation}
and the equation
\begin{equation}
\partial_s\phi^+=J \partial_t\phi^+-\nabla_g  \widetilde{H}_{\varpi}(\phi,t)=0.
\end{equation}
by the finite energy condition. So $\phi^+(s,t)=l^+(t)$ for some $l^+ \in  S_{W}$.

If $l^+$ is not the limit of the whole sequence $\phi_k$, we can get another subsequence, called $\{\phi_{k'_j}\}$, which converges to some $l'^+\in  S_{W}$. Select such $k'_j$ to make sure $k_j<k'_j$. As $ S_{W}$ has finite elements, there should exist a sequence $\{s_j\}$, $k_j<s_j<k'_j$, and a positive number $\epsilon$, so that
\begin{equation}\label{sec4:limit}
\sup_{l\in  S_{W}}\|\phi(s_j,\cdot)-l\|_{C^0([0,1])}>\epsilon
\end{equation}
This is impossible, because we can repeat the above argument to find a subsequence of $\{\phi_{s_j}\}$, which converges to some $l''^+\in S_{W}$ on $[-1,1]\times[0,1]$, which contradicts  (\ref{sec4:limit}).
\end{proof}

The proof of the above theorem implies the following conclusions.

\begin{lm}\label{higher convergence}
If $\phi \in \W^{d+1}(\cC, L, l;K)$, then $\forall \zeta\in \cZ$
\begin{equation}
\lim_{s\to\pm\infty}\phi(\ve_\zeta(s,\cdot))=l_\zeta,
\end{equation}
with all its derivatives.
\end{lm}

Suppose we have a sequence of pointed discs $\{\cC_i\}$, that each one is equipped with the same directed Lagrangian labels  $\{L_C\}$ for ordered $d+1$ punctures ${\zeta_0}_i,\cdots,{\zeta_d}_i$. Then we have a Lefschetz timble pair $(L_{k,0}, L_{k,1})$ for $k=0,\cdots,d$. Suppose we fixed the choice of Hamiltonian chords $l_{\zeta_k}=l_k \in S_W(L_{k,0}, L_{k,1}),k=0,\cdots,d$ and we have Solutions  $\phi_i\in \W^{d+1}(\cC_i, L, l;K)$.  With the help of  Theorem \ref{d control} and Theorem \ref{limit}, we can use a standard argument as in the proof of  Theorem 1 of \cite{f1} to prove the following two conclusion.
\begin{thm}[Gromov compactness, unstable case]\label{gromov-unst}
Suppose $\cC_i\to \cC$ for some pointed disc $\cC$, with the limit of punctures ${\zeta_k}_i\to \zeta_k$. Then we must have a collection of Hamiltonian chords $l'=\{l'_{\zeta_k}\in S_W(L_{k,0},L_{k,1}) \;|\;k=0,\cdots,d\}$, and a $
\phi\in \W^{d+1}(\cC, L, l';K)$, that
$$
\phi_i \to\phi,
$$
subsequently in $C^{\infty}_{\loc}$.

 Moreover, at each outgoing  strip-like end $\epsilon_k(\Sigma^+)$ for puncture $\zeta_k$,we have a set $\{l^0:=l'_{\zeta_k},l^1,...,l^{N-1}, l^N:=l_{\zeta_k}\}\subset S_W(L_{k,0},L_{k,1})$, and $\psi^m\in \tM(l^{m-1},l^m;\varpi)$, together with $N$ subsequence $\{s_i^m\}$ of real numbers, such that
\begin{equation}
s_i^m*\phi_i\circ\epsilon_k \to \psi^m\;\text{subsequently in}\;C^{\infty}_{\loc}.
\end{equation}
Here $s_i^m*\phi_i\circ\epsilon_k(s,t)=\phi_i\circ\epsilon_k(s+s_i^m,t)$.
On incoming point, we have the similar result, with the only difference that $\{l^0:=l_{\zeta_k},l^1,...,l^{N-1}, l^N:=l'_{\zeta_k}\}$.
\end{thm}

In particular we have
\begin{thm}[Gromov compactness results for Floer homology]\label{thm:split}
Suppose $L_0$ and $L_1$ are two fixed Lefschetz thimbles, and $l^-,l^+\in S_W(L_0,L_1)$. For any sequence  $\{\phi_i\}\subset \tM(l^-,l^+;\varpi)$, there exist $N$ real sequences $\{s^1_i\},\cdots,\{s^N_i\}$ and $\{l^0:=l^-,l^1,...,l^{N-1}, l^N:=l^+\}\subset S_{W}$,  as well as $\phi^m\in \tM(l^{m-1},l^m;\varpi), m=1,2,...,N$, such that
\begin{equation}
s^m_i*\phi_i\to \phi^m\;\text{in}\;C^{\infty}_{\loc}.
\end{equation}
\end{thm}

\begin{thm}[Gromov compactness, stable case]\label{gromov:st}
Suppose $\cC_i$ can be written as $\cC_i=\cC_i^1\#_{l_i}\cC_i^2$, with $\cC_i^2$ taking $m$-th to $m+r$-th punctures($m>0$,$1<r<d-2$), and the gluing parameter $l_i\to +\infty$ as $i\to \infty$. Suppose $\cC_i^1\to \cC^1$ and $\cC_i^2\to \cC^2$ and denote by $L^1,L^2$ the corresponding Lagrangian labels for them, inheriting from $L$. Denote by $\zeta'$ the gluing point. Then we must have $l'_k\in S_W(L_{k,0},L_{k,1})$ for $k=0,\cdots,d$,and one  $l^*\in S_W(L_{\zeta',0},L_{\zeta',1})$, so that
$$
\phi_i\to \phi^1\text{ on } \cC_i^1, \phi_i\to \phi^2\text{ on } \cC_i^2 \text{ in } C^{\infty}_{loc},
$$
for some $\phi^1\in \W^{d+1}(\cC, L^1, l^1;K),\phi^2\in \W^{d+1}(\cC, L^2, l^2;K)$, where
$$ l^1= \{l'_0,\cdots,l'_{m-1},l^*,l'_{m+r+1},\cdots,l'_d\}  , l^2=\{l^*,l'_m,\cdots,l'_{m+r}\}.
$$
\end{thm}

\section{Fredholm theory}\label{Fredholm theory section}\label{sec-8}

\subsection{Linearization}\

Recall that in section \ref{Witten equation section}  we defined the perturbed Witten map $\WI_{\cC, K}(\phi)$, which is a section of the bundle $\E_\cC\to \B_\cC$. Consider the linearization $D\WI_{\cC, K}(\phi)$ of $\WI_{\cC, K}(\phi)$, write it as $D_\phi$ for simplicity.

If we varies the disc, but keep the branes and Hamiltonian chords in order, we have the Banach manifold (which is already defined in Section \ref{sec-6.1})
$$
\B^{d+1}(L,l)=\{(\cC,\phi)\;|\;\phi\in \B^{d+1}(\cC,L, l), \cC\in \overline{\cR}\}.
$$

 Thus we can extend the Witten map $\WI_{\cC, K}$ to a map $\WI_{ K}$  on $\B^{d+1}(L,l)$ by $ \WI_{ K}(\cC,\phi)=\WI_{\cC, K}(\phi)$. Thus
 we can write $D\WI_{ K}=D_{\phi}+D_{\cC}$, where $D_{\cC}$ is a linear map on tangent space of the pointed disc configuration space at $\cC$.

  Similarly we define the extend moduli space  $\W^{d+1}(L, l;K)=\{(\phi,\cC)|\phi\in \W^{d+1}(\cC, L, l;K)\}$. Clearly, $\W^{d+1}(L, l;K)$ is the zero locus of $\WI_{ K}$.

By definition, $D_{\phi}$ is a elliptic type differential operator. On strip-like ends, $\WI_{\cC, K}(\phi)$ can be written as
$$
\WI_{\cC, K}(\phi)=\frac{\pat \phi}{\pat s}+J(\phi)\frac{\pat \phi}{\pat t}-\nabla_g \widetilde{H}_{\varpi}(\phi),
$$
and we can write
\begin{equation}\label{w-linearization-strip}
D_\phi \xi=\nabla_{s}\xi+J\nabla_t\xi-\nabla_{\xi}\nabla\widetilde{H}_{\varpi}(\phi).
\end{equation}

The operator $D_\phi$ on these parts can be written as the form $D_\phi=\pat_s-A(s)$. For each ends $\epsilon_\zeta(\Sigma^+)(\epsilon_\zeta(\Sigma^-))$, if we let $s\to+\infty(-\infty)$, then $\phi\to l$ for some $l\in S_{W}(L_{\zeta,0},L_{\zeta,1})$. On the other hand, we get a limit operator $A_l$ of $A(s)$.  We study these limit operators first.

For any $l\in S_{W}(L_0,L_1)$, define
\begin{equation}
w^{1,2}(l)=\{\xi\in W^{1,2}([0,1],l^*TM)\;|\;\xi(0)\in T_{l(0)}L_0, \xi(1)\in T_{l(1)}L_1 \},
\end{equation}
and
\begin{equation}
l^2(l)=L^2 ([0,1],l^*TM).
\end{equation}

\begin{lm}\label{w:soliton-iso}
Assume that $L_0$ intersects $L_1$ transversally  at time $1$ along $l$. Then the operator
$$
A_{l}:w^{1,2}(l)\mapsto l^2(l),
$$
defined by
$$
A_{l}\xi=-J{\nabla}_t \xi+\nabla_{\xi}(\nabla\widetilde{H}_{\varpi})
$$
is an isomorphism.
\end{lm}
\begin{proof}
For any $\eta\in T_{l(0)}L_0$, choose a smooth curve $\lambda$ lying in $L_0$ with $\lambda(0)=l(0)$, $\dot\lambda(0)=\eta$. Denote $\xi(t)=\partial_s \phi_t(\lambda(s))|_{s=0}$, where $\phi_t$ is the time-$t$ map of the flow generated by the equation (\ref{perturb-hal-r}).  Note that $\phi_t(\lambda(0))=l(t)$.

We have
\begin{equation}
J\partial_t(\phi_t(\lambda(s)))-\nabla\widetilde{H}_{\lambda}(\phi_t(\lambda(s)))=0
\end{equation}

 Taking the derivative $\nabla_s$ to the above equation at $s=0$, we have
\begin{equation}\label{w-linearization-1}
-J\nabla_t\xi+\nabla_{\xi}(\nabla\widetilde{H}_{\lambda})=0.
\end{equation}
By the assumption that $L_0$ intersects $L_1$ transversally  at time $1$ along $l$,  $\xi(1)=(\phi_1)_*\eta\not\in T_{l(1)}L_1$. Therefore there is no solution of this equation in $w^{1,2}(l)$, and then the operator $A_l$ is injective.

For $\xi,\xi' \in w^{1,2}(l)$
\begin{align}
\int_0^1g(\xi',-J\nabla_t \xi+\nabla_{\xi}(\nabla\widetilde{H}_{\varpi}))dt=&-g(\xi',J\xi)|_0^1+\int_0^1g(\nabla_t\xi', J\xi)dt\notag\\
+&\int_0^1 \text{Hess}\widetilde{H}_{\varpi}(\xi,\xi')dt\notag\\
=&\int_0^1g(\xi,-J\nabla_t \xi'+\nabla_{\xi'}(\nabla\widetilde{H}_{\varpi}))dt.
\end{align}
The last equality holds because, as $L_0$ and $L_1$ are Lagrangian submanifolds, $g(\xi',J\xi)=0$ at both end points. Thus, the operator is injective and self-adjoint, so it is an isomorphism.
\end{proof}

Now Lemma \ref{w:soliton-iso} implies the following result.
\begin{thm}If for any $\zeta\in \cZ$, $L_{\zeta,0}$ and $L_{\zeta,1}$ intersect transversally at time $1$ along $l_\zeta$, then the operator $D_\phi$ is a Fredholm operator.
\end{thm}

\subsection{Index and orientation}\label{index-orientation section}
\subsubsection*{Landau-Ginzburg branes}\label{lift}\

Define $\mathcal{G}M$ to be the Lagrangian Grassmannian bundle of $TM$. Namely, we have the fibration
$$
\xymatrix{
\mathcal{G}_p M \ar[r]^{} &    \mathcal{G}M \ar[d]^{} \\
                        &  M }
$$
where the fiber over each point $p\in M$ is the Grassmannian of Lagrangian subspaces of $T_pM$.

We then get an exact sequence of homotopy group
\begin{equation}\label{fibration}
\cdots\to \pi_2(M) \to \pi_1(\mathcal{G}_p M)\to \pi_1(\mathcal{G} M)\to\pi_1(M)\to *.
\end{equation}
Note that
\begin{equation}\label{fibration-maslov}
\pi_1(\mathcal{G}_p M)\cong H_1(\mathcal{G}_p M)\cong \mathbb{Z},
\end{equation}

 where the isomorphism is given by the Maslov index.

The sequence
\begin{equation}\label{fibration-homology}
\cdots\to \pi_2(M) \to H_1(\mathcal{G}_p M)\to H_1(\mathcal{G} M)\to H _1(M)\to *,
\end{equation}
is also exact since the abelization is right exact (see \cite{Sh1}, page 77). Let  $\widetilde{\mathcal{G}}M$ to be the universal Abelian covering space of $\mathcal{G}M$. For a Lefschetz thimble $L$, we can have a natural inclusion $i:L\mapsto \mathcal{G}M$.

We also need the notion of Pin structure in formulating the orientation theory of moduli spaces. In general, a Pin structure on a vector bundle $F\mapsto B$ is a principle $\text{Pin}_n$ bundle $P^{\#}=P^{\#}(F)$ with an isomorphism $P^{\#}\times_{\text{Pin}_n} \R^n\cong F$.

Since any Lefschetz thimble $L$ is simply connected, there is a natural Pin structure $P$ on  the bundle $TL$. We call it  the
Pin structure of $L$.
\begin{df}
A \textit{Landau-Ginzburg brane} $L^{\#}=(L,i^{\#}, P)$ is a Lefschetz thimble $L$ with a chosen Pin structure $P$ on it, as well as a lift $i^{\#}:L\mapsto \widetilde{\mathcal{G}}M$ of the inclusion $i:L\mapsto \mathcal{G}M$, such that the diagram
$$
\xymatrix{
&\widetilde{\mathcal{G}}M\ar[d]^{\pi} \\
L\ar[r]^i\ar[ur]^{i^{\#}} &\mathcal{G}M
}
$$
commutes.
\end{df}

 A Landau-Ginzburg brane is an anchored Lagrangian brane defined in \cite{fooo3} (or ref. \cite{Sh1}).

\subsubsection*{Grading datum}\

In \cite{Sh1}, Sherdian introduced a $\pmb{G}$-graded $A_{\infty}$ category. This $\pmb{G}$-graded $A_{\infty}$ category can be regarded as a generalization of ordinary $A_{\infty}$ category. In particular, the grading datum $\pmb{G}$ can be regarded as a generalization of  $\Z$ grading, which is most frequently used on the study of Fukaya categories, see \cite{Si3}, Chapters 9. We will adopt the setting of \cite{Sh1}, section 2.

\begin{df}
An \textit{unsigned grading datum} $\pmb{G}$ is an abelian group $Y$ together
  with a morphism $f:\Z\mapsto Y$. We will use the shorthand $\pmb{G} = \{ \Z \mapsto Y \}$. If we write the cokernel of $f$ by $X$, we also write this datum as
  $$
\begin{CD}
 \Z@>f>>Y@>g>>X.
\end{CD}
$$
Suppose we have two unsigned grading data $\pmb{G_1} = \{ \Z \mapsto Y_1 \},\pmb{G_2} = \{ \Z \mapsto Y_2 \}$, a morphism $\pmb{p} : \pmb{G_1}\mapsto \pmb{G_2}$ is a morphism $p:Y_1\mapsto Y_2$ that makes the following diagram commute:
$$
\begin{CD}
\Z@>>>Y_1\\
@VVidV@VVpV\\
\Z@>>>Y_2.
\end{CD}
$$
\end{df}
\begin{df}
We define the sign grading datum, $\pmb{G}_{\sigma}:= \{ \Z \mapsto \Z_2 \}$.
\end{df}
\begin{df}\label{signed-grading}
We define a \textit{grading datum} $(\pmb{G}, \pmb{\sigma})$ to be an unsigned grading
datum $\pmb{G}$, together with a sign morphism $\pmb{\sigma}$, which is a morphism of unsigned grading data from $\pmb{G}$ to $\pmb{G}_{\pmb{\sigma}}$. We denote by $\sigma$ the coresponding morphism $Y\to \Z_2$ in defining $\pmb{G}$.
\end{df}

If we take the direct sum of $g:Y \mapsto X$ with $\Z  \mapsto \Z$, we get a new unsigned datum
$$
\begin{CD}
 \Z@>(f, 0)>>Y\oplus \Z@>g\oplus id>>X\oplus \Z.
\end{CD}
$$
We can also define a sign morphism for this unsigned datum: if $\sigma$ is the morphism $Y\mapsto \Z_2$ in defining the sign morphism $\pmb{G}\mapsto \pmb{G}_{\sigma}$, we define $\widetilde{\sigma}(y,k)=\sigma(y)$. In this way, we can define a new grading datum from a given grading datum $\pmb{G}$, we call it $\pmb{G}\oplus \Z$.

\begin{df}
Suppose we have a grading datum $(\pmb{G}, \pmb{\sigma})$ with the unsigned grading datum $\Z\mapsto Y$. A \textit{$Y$-graded Abelian group (vector space)} is an Abelian group (vector space) $V$, together with a collection of Abelian groups(vector spaces) $V_y$, indexed by $y \in Y$, and an isomorphism
$$
V\cong \bigoplus_{y\in Y}V_y.
$$
\end{df}

\begin{df}
A \textit{$\pmb{G}$-graded algebra} is an algebra assigned with $Y$-grading, that is an  algebra whose multiplication respects the $Y$-grading.  For a   $\pmb{G}$-graded algebra $\pmb{R}$, a \textit{$\pmb{G}$-graded module (bimodule)} $A$ on $\pmb{R}$ is a $R$ module (bimodule) on a $\pmb{G}$-graded Abelian group $V$, that the multiplication respects the $Y$-grading.
\end{df}

\subsubsection*{Grading datum in Landau-Ginzburg model}\

We will define a grading datum from our geometric data. By  (\ref{fibration-maslov}) and (\ref{fibration-homology}) , We can get an unsigned datum $\pmb{G}_0$:
\begin{equation}\label{homology-datum}
\begin{CD}
H_1(\mathcal{G}M_p)\cong\Z@>f_1>>H_1(\mathcal{G}M)@>g_1>>H_1(M).
\end{CD}
\end{equation}
To define a sign morphism $\pmb{\sigma}$, we firstly define a real vector bundle $\mathscr{L} \mapsto \mathcal{G}M$, whose fibre over a point is identified with the Lagrangian subspace at that point. The first Stiefel-Whitney class $\omega_1(\mathscr{L})$ of this bundle is an element in $H^1(\mathcal{G}M)$. Thus we can define a map $\sigma_1:H_1(\mathcal{G}M)\mapsto \Z_2$ by pairing with this element $\omega_1(\mathscr{L})$. Using this, we can get a morphism $\pmb{G_0}\mapsto\pmb{G}_{\sigma}$, still denoted by $\pmb{G_0}$. By direct summing with $\Z$ we get a new grading datum $\pmb{G}=\pmb{G_0}\oplus \Z$. We will alway use the notation $\pmb{G}$ to denote  this grading datum.

\begin{rem}
The grading datum $\pmb{G_0}$ is the one used  in \cite{Sh1}. We ``add'' some information encoded by $\Z$ in our grading datum, which actually encodes the speed $\mathsf{k}$ of the Hamiltonian chords.
\end{rem}
\begin{rem}
Here is an example for these setting. Define $D$ to be the union of hyperplanes $\{Z_i=0 \}$ for $i=0,1,\cdots,n$ in $\CP^n$.    For any hypersurface in $\CP^n$, its part in $\CP^n\setminus D$   can be written as
$$N:=\{(z_1,\cdots,z_n)\in (\C^*)^n\;|\;W(z_1,\cdots,z_n)=0\}$$ for some $W$, where $z_i=Z_i/Z_0$. Sheridan (\cite{Sh1}) has studied the Fukaya category of $N$ in terms of `relative Fukaya category' In defining the grading datum there,  he used the exact sequence
\begin{equation}\label{grading-datum}
\begin{CD}
H_1(\mathcal{G}N_p)\cong\Z@>f_1>>H_1(\mathcal{G}N)@>g_1>>H_1(N)@>>>0,
\end{CD}
\end{equation}
where $\mathcal{G}N$ is the bundle of Lagrangian subspaces of $TN$. Let $M=\CP^n\setminus D$, denote by $h$ the restriction of the Fubini-Study metric on $M$. The inclusion $N\to M$ induces the commutative diagram
$$
\begin{CD}
\pi_2(N)@>2c_1(TN)>>H_1(\mathcal{G}N_p)@>>>H_1(\mathcal{G}N)@>>>H_1(N)@>>>0\\
@VVV@VVaV@VVbV@VVcV\\
\pi_2(M)@>2c_1(TM)>>H_1(\mathcal{G}M_p)@>>>H_1(\mathcal{G}M)@>>>H_1(M)@>>>0
\end{CD}
$$
As $\CP^n\setminus D\cong (\C^*)^n$, $2c_1(TM)=0$. By the same argument as in Lemma 3.24 of \cite{Sh1}, we have $2c_1(TN)=0$, so the diagram is reduced to
$$
\begin{CD}
0@>>>H_1(\mathcal{G}N_p)@>>>H_1(\mathcal{G}N)@>>>H_1(N)@>>>0\\
@.@VVaV@VVbV@VVcV\\
0@>>>H_1(\mathcal{G}M_p)@>>>H_1(\mathcal{G}M)@>>>H_1(M)@>>>0.
\end{CD}
$$

We can expect to get some implication of ``correspondence'' between the affine space and one of its hypersurfaces. This can be expected to be used in studying the relation between the Floer theory of LG model $(M,W,h)$ and the Fukaya category on $N$.
\end{rem}

\begin{rem}
However, any $\pmb{G_0}$-graded vector space or $\pmb{G_0}$-grade algebra can be also regarded a $\pmb{G}$-graded one, as we can directly assign $(y,0)$ for an element which has $\pmb{G_0}$ grading $y$. In this sense, the rich examples in section 2.2 in \cite{Sh1} can be regarded as examples for $\pmb{G_0}$ grading.
\end{rem}

 \subsubsection*{Grading for Hamiltonian chords}\

In section \ref{section-weq-act}, we considered the Hamiltonian chord equation (\ref{perturb-hal-r}):
$$
J(\phi)\frac{\pat \phi}{\pat t}-\nabla_g \widetilde{H}(\phi)=0
$$
It generates a flow $\phi_{\widetilde{H},t}$.  For $x\in L_1$, $(\phi_{\widetilde{H},t})^{-1}_{*}({TL_1}_x)$ gives a path in $\mathcal{G}M$ as $t$ goes from $0$ to $1$. We can lift it to a path in $\widetilde{\mathcal{G}}M$. Using this lift, we can start from a brane $L_1^{\#}$ and get a brane at the time $t=1$, denoted by $(\phi^{\#}_{\widetilde{H}})^{-1}(L_1^{\#})$.

Suppose we have a pair of branes $(L_0^{\#},L_1^{\#})$. By our settings in section \ref{sec-3}, we know that  $\phi_{\widetilde{H},1}(L_0)$ intersects $L_1$ transversally
, so $\phi_{\widetilde{H},1}^{-1}(L_1)$ intersects $L_0$ transversally, too.  Note that the group $H_1(\mathcal{G}M)$ have a covering action on $\widetilde{\mathcal{G}}M$. For $y\in H_1(\mathcal{G}M)$, denote the action by $\tau_{y}$. We will assign a grading in $\pmb{G}$ for each Hamiltonian chord $l$ between $L_0$ and $L_1$. First we choose an element $y$ in $H_1(\mathcal{G}M)$, which is the unique element that $l$ lifts to a path from $L_0^{\#}$ to $\tau_y L_1^{\#}$ with zero relative Maslov index.

To be more precise, the chord $l$ correspond to a intersection point $p_0=l(0)$ of $L_0$ and a new Lagrangian submanifold $\widetilde{L_1} =(\phi_{\widetilde{H},1})^{-1}(L_1)$.  We connect $T_{p_0}(L_0)$ and $T_{p_0}\widetilde{L_1}$ by a path $h_l$ in $\mathcal{G}_{p_0}M$ of zero relative Maslov index and lift it to a path $\widetilde{h_l}$ with $\widetilde{h_l}(0)\in L_0^{\#}$.  We choose a $y$ in $H_1(\mathcal{G}M)$, so that $\widetilde{h_l}(1)\in (\psi^{\#}_{\widetilde{H}})^{-1}\tau_{-y}(L_1^{\#})$.

 On the other hand, we recall that, each Hamiltonnian has a speed parameter $\mathsf{k}$, which is used to define the space $S_W$ in Definition \ref{chord-space-df}. For any $\pmb{G}$-graded algebra $\pmb{R}$, we get a $\pmb{G}$-graded module  $CF(L_0^{\#},L_1^{\#})$, freely generated by elements in $S_W(L_0,L_1)$. For the element $l$, the grading is defined to be $\mathsf{y}=(y,\mathsf{k})$.

\subsubsection*{Fredholm index and orientation}\

Suppose we have an element $  \phi \in  \B^{d+1}(\cC, L, l)$. For any $a>0$, we denote $\Sigma^{a,+}=[a,+\infty)\times[0,1],\Sigma^{a,-}=(-\infty,-a]\times[0,1]$ , for simplicity. By definition, we have $\phi\circ \ve_{\xi^{\pm}}\to \text{ some }l^{\pm}$, as $s\to {\pm}\infty$, where the sign depends on wether it is an incoming or outgoing end.

Choose a large $a$ and make perturbation for $\phi$ to guarantee that $\phi\circ\ve_{\xi^{\pm}}(s,t)=l^{\pm}(t)$, on these strip-like ends. Thus, on an end, we have $A(s)=A_l$ for some chord $l$.

 Choose a function

 $$
\begin{cases}
\rho_0(s)=0&  s<0\\
0\leq \rho_0(s) \leq 1 & 0 \leq s \leq 1\\
\rho_0(s)=1 &s \geq 1.
\end{cases}
$$

 Let
 $$
 \widetilde{\phi}(s,t)=\psi_{\widetilde{H},\rho_0(s-a)t}(\phi(s,t)),\Psi(\xi)(s,t)=D\psi_{\widetilde{H},\rho_0(s-a)t}(\xi)(s,t),
 $$
 on the outgoing end, and let
$$
 \widetilde{\phi}(s,t)=\psi_{\widetilde{H},\rho_0(s+a)t}(\phi(s,t)),\Psi(\xi)(s,t)=D\psi_{\widetilde{H},\rho_0(a+s)t}(\xi)(s,t),
 $$
 on the incoming end. We also require that $\widetilde{\phi}=\phi$ and $\Psi$ is identity on other part of the disc.

Thus we get an isomorphism $\Psi$ from some Banach space $\widetilde{\E_\cC}(\phi)$ to $\E_\cC|_{\phi}$, where $\widetilde{\E_\cC}(\phi)$ is some subspace of $\widetilde{\phi}(s,t)^*TM$.

 Let $\widetilde{D_{\phi}}=\Psi^{-1}\circ D_{\phi} \circ\Psi$. Note that
 $$
 \widetilde{D_{\phi}}=\Psi^{-1}\circ\big(\pat_s+A_l\big)\circ\Psi \xi=\pat_s\xi+J \pat_t\xi
 $$
 on each ends.

 Define a trivial Hermitian vector bundle over the upper half plane $\R\times\R_{\geq 0}$ whose fibre are $T_{p_0}M$. Note that, for each Hamiltonian chord $l$, we associate an intersection point $p$, together with a path $h_l\in \mathcal{G}M_p$. Using this, we define the Lagrangian boundary condition on $\R$ given by $h_l(\rho_0(s))$.  If the disc $r$ has $d+1$ boundary punctures, we get $d+1$ Fredholm operators $D_0,\cdots,D_d$, we call them orientation operators.




 We can get an elliptic operator $\hat{D}$ by gluing these orientation operators to $\widetilde{D_{\phi}}$. By the standard index gluing formula, we know



\begin{align*}
\text{i}(\hat{D})=&\text{i}(\widetilde{D_{\phi}})-\text{i}(D_0)+\text{i}(D_1)+\cdots+\text{i}(D_{d+1})+n\\
=&\text{i}(D_{\phi})-\text{i}(D_0)+\text{i}(D_1)+\cdots+\text{i}(D_{d+1})+n,
\end{align*}

 \begin{align}
 \text{det}(\hat{D})\cong& \text{det}(\widetilde{D_{\phi}})\otimes\text{det}(D_0)^{\vee}\otimes\text{det}(D_1)\otimes\cdots\otimes\text{det}(D_{d+1})\notag\\
 \cong &\text{det}(D_{\phi})\otimes\text{det}(D_0)^{\vee}\otimes\text{det}(D_1)\cdots\otimes\text{det}(D_{d+1}).\label{orientation-bundle}
 \end{align}

Now $\hat{D}$ can be regarded as the Cauchy-Riemann operator plus a zero order term, on a bundle pair $(D^2,E,F)$.

 The bundle $F$ on $S=\pat  D^2$ can be regarded as a map $F:S\to \mathcal{G}M$, this define a class $[\rho]\in H_1(\mathcal{G}M)=Y(M)$.  The bundle pair $(E,F)$ also has a Maslov index, denoted by $\mu(D^2,E,F)$. By Lemma 3.3 in \cite{Sh1}, we have
 $$
 [\rho]=f_1(\mu(D^2,E,F)),
 $$
 where $f_1$ is the map in (\ref{grading-datum}).

Now we consider the boundary condition for $\hat{D}$.  Suppose that the $d+1$ Lefschetz thimbles involved in the boundary condition are lifted to Landau-Ginzburg branes $L^{\#}_0,\cdots,L^{\#}_d$, and $l_j\in S_W(L_{j-1},L_j)$ for $j=1,\cdots,d$, $l_0\in S_W(L_0,L_d)$ are corresponding Hamiltonian chords. Suppose that the grading for $l_j$ is $\mathsf{y}_j=(y_j,k_j)$ for $j=0,\cdots,d$.

At an incoming strip-like end with Hamiltonian chord $l_j$, the brane $L_{j-1}^{\#}$ becomes $(\phi_{\widetilde{H}})^{-1}\tau_{-y}(L_j^{\#})$, after passing the chord $l_j$, and then becomes $\tau_{-y}(L_j^{\#})$ as $s$ goes from $+\infty$ to $1$, along the boundary piece $\{s,1\}$, by the definition of $\Psi$.
So if we lift $F$ to a map $F^{\#}:[0,1]\to \mathcal{G}M$, then $F^{\#}(1)=\tau_{(-y_1-\cdots-y_d+y_0)}F^{\#}(0)$, this means
$$
 f_1(\mu(D^2,E,F))=[\rho]=-y_1-\cdots-y_d+y_0.
$$
 So
 \begin{equation}\label{fredholm-index}
 f_1(i(\hat{D})-n)=f_1(i(D_{\phi}))=-y_1-\cdots-y_d+y_0,
 \end{equation}
 where the first equality is obtained by the index formula for $\hat{D}$.

\subsection{Exponential decay}\

\begin{thm}\label{sec5:exponential}
For any $\phi\in \W^{d+1}(\cC, L, l;K)$,  at any strip-like ends $\epsilon_{\zeta}(\Sigma^{\pm})$, we have
$$
|\partial_s (\phi\circ\epsilon_{\zeta}(s,t))|<ce^{-\delta|s|},
$$
for some positive numbers $\delta$ and $c$. In particular, for any $\phi\in \tM(l^-,l^+; \varpi)$ for some given $l^-,l^+\in S_W(L_0,L_1)$,  there are positive numbers $\delta$ and $c$, such that
\begin{equation}
|\partial_s \phi(s,t)|<ce^{-\delta|t|}.
\end{equation}
Furthermore, if a sequence $\{\phi_i\}\in \W^{d+1}(\cC, L, l;K)$ converging to some $\phi$ locally, and is uniform on some strip-like end $\epsilon_{\zeta}(\Sigma^+)$ ($\epsilon_{\zeta}(\Sigma^-)$), we have
$$
|\partial_s (\phi_i\circ\epsilon_{\zeta}(s,t))|<ce^{-\delta|s|},
$$
for some $(\delta,c)$ independent of $i$.
\end{thm}

\begin{proof} We only prove the special case for $\phi\in \tM(l^-,l^+; \varpi)$, and the general case is similar. The proof here is similar to the proof of \cite[Theorem 4]{f1}.

Let $\xi=\partial_s \phi$. Taking the covariant derivative $\nabla_s$ of  (\ref{perturb-lg-r}), we have

\begin{equation}\label{w-linearization-2}
\nabla_s \xi+J\nabla_t\xi-\nabla_{\xi}(\nabla_g \widetilde{H}_{\varpi})=0.
\end{equation}
Define
\begin{equation}
f(s)=\frac{1}{2}\int_0^1|\xi(s,t)|^2 dt.
\end{equation}
We have
\begin{equation*}
\frac{df}{ds}=\int_0^1g(\nabla_s \xi,\xi)dt,
\end{equation*}
and
\begin{align*}
\frac{d^2f}{ds^2}=&\int_0^1g(\nabla_s \xi,\nabla_s \xi)dt+\int_0^1g(\nabla_s \nabla_s \xi,\xi)dt\\
=&\int_0^1|\nabla_s\xi|^2dt-\int_0^1g(\nabla_s(J\nabla_t\xi-\nabla_{\xi}(\nabla_g\widetilde{H}_{\varpi})),\xi)dt\notag\\
=&\int_0^1 |\nabla_s\xi|^2dt+\int_0^1g(-J\nabla_t\nabla_s\xi-JR(\xi,\pat_t \phi)\xi,\xi) dt+\nabla_s g(\nabla_\xi(\nabla_g \widetilde{H}_{\varpi}),\xi)\\
-&g(\nabla_\xi (\nabla_g \widetilde{H}_{\varpi}),\nabla_s \xi )dt\notag\\
=&2\int_0^1 |\nabla_s\xi|^2dt+g((\nabla_s\xi)^{\bot}, J\xi)|^1_0+\int^1_0 [-g(\nabla_s\xi, \nabla_\xi(\nabla_g H))+R(J\xi, \xi,\xi,\pat_t\phi)\\
+&\nabla_\xi(\nabla d \widetilde{H}_{\varpi}(\xi,\xi) )-\nabla d\widetilde{H}_{\varpi}(\xi,\nabla_s\xi)] dt\\
=& 2\int_0^1 |\nabla_s\xi|^2dt+g(\mathrm{I\!I}(\xi,\xi), J\xi)|^1_0+\int^1_0 [R(J\xi, \xi,\xi,\pat_t\phi)+\nabla^2 d \widetilde{H}_{\varpi}(\xi,\xi,\xi)]dt,
\end{align*}
Where $\mathrm{I\!I}$ is the second fundamental form of $L_0\cup L_1$ in $M$. To obtain the above equality, we used the following facts:
\begin{align*}
&g(\nabla_X(\nabla_g H), Y)=(\nabla_X (dH))(Y)=\nabla dH(X,Y);\\
&\nabla_X (\nabla d H(Y, Z))=(\nabla_X (\nabla dH))(Y,Z)-\nabla dH (\nabla_X Y, Z)-\nabla dH(Y, \nabla_X Z),
\end{align*}
for any  $H\in C^\infty(M)$ and vector fields $X,Y,Z\in \Gamma(TM)$.

Using the negative gradient flow of $\im(W)$, we can move a small neighborhood of $l^+(0)\subset L_0$ to the neighborhood of $l^+(1)\in M$. Furthermore, a local smooth deformation can deform it into the neighborhood of $l^+(1)$ in $L_1$. Hence we get a family of second fundamental form $\tilde{\mathrm{I\!I}}(t)$. So we have
\begin{align*}
\frac{d^2f}{ds^2}=&2\int_0^1 |\nabla_s\xi|^2dt+\int^1_0 [\pat_t g(\tilde{\mathrm{I\!I}}(t)(\xi,\xi), J\xi)+R(J\xi, \xi,\xi,\pat_t\phi)+\nabla^2 d \widetilde{H}_{\varpi}(\xi,\xi,\xi)]dt
\end{align*}

By the $C^0$ estimate, $\phi$ is a bounded map. Hence $\nabla_t\tilde{\mathrm{I\!I}}(t), \nabla^2 d \widetilde{H}_{\varpi}(\phi)$ are bounded operators. By (\ref{w-linearization-2}) and Lemma \ref{w:soliton-iso}  there is a $\delta>0$, such that
\begin{equation}\label{proof:exponential}
\int_0^1 |\nabla_s \xi|^2dt\ge \delta^2\int_0^1|\xi|^2dt,
\end{equation}
if $|s|$ is large enough. On the other hand, $\xi$ and its derivatives converge to zero as $|s|\to +\infty$ by Lemma \ref{higher convergence}.  so for any $s>s_0$ for some $s_0$, we have
\begin{equation}
\frac{d^2f}{ds^2}\ge \delta^2 f(s).
\end{equation}
This implies the exponential decay
$$
f(s)\le c e^{-\delta s},\forall s\ge s_0.
$$
 By using a standard maximum principle to the second order elliptic equation induced by (\ref{w-linearization-2}),  we can get the exponential decay of $\partial_s \phi$.

If $\{\phi_k\}$ is a sequence in$\tM(l^-,l^+; \varpi)$ converging to $\phi$ uniformly on $\Sigma^+$($\Sigma^-$),  then (\ref{proof:exponential}) holds if $\phi$ is replaced by $\phi_k$, and the constant is independent of $k$. Thus we get the exponential decay of $\partial_s \phi_k$.
\end{proof}

\subsection{Transversality}\label{subsection8.4}

\subsubsection*{Transversality on stripes}\

When considering on the area $\Sigma$, for two Hamiltonian chords $l^-,l^+$ from  Lefschetz thimble $L_0$ to $L_1$, the Banach manifold $\B(\cC,\{l^-,l^+\})$ is exactly
$$
\mathcal{P}^{1,p}(l^-,l^+)=\{\phi\in W^{1,p}_{loc}(\Sigma,M)|\phi(s,0)\in L_0,\phi(s,1)\in L_1,\lim_{s\to \pm\infty}\phi(s,)=l^{\pm}\}.
$$
We define the fiber bundle $\LL^p$ over $\mathcal{P}^{1,p}(l^-,l^+)$ by setting the fiber $\LL^{p}(\phi)=L^{p}(\Sigma,\phi^*TM)$ for $\phi\in \mathcal{P}^{1,p}(l^-,l^+)$.

In the beginning of Section \ref{section-weq-act}, we have defined the Banach space
$$
\mathcal{V}^{\epsilon}_{per}=\{\varpi\in \mathcal{V}_{per}| \;||\varpi||_{\epsilon}<+\infty\},
$$
where the norm
$$
||\varpi||_{\epsilon}=\sum_{k=0}^{\infty} \epsilon_k \sup_{(x,t)\in M\times [0,1]} |\nabla^k \varpi(x,t)|,
$$
and $\{\epsilon_k\}$ is a sequence of positive real numbers converging rapidlyto zero.

For a positive $\delta_0$ smaller enough, we have defined an open subset $\mathcal{V}^{\epsilon,\delta_0}_{per}$ in Definition \ref{perturb-space-stripe}, so that then for any $\varpi\in \mathcal{V}^{\epsilon}_{per}(\delta_0)$, Lemma \ref{perturb-fini-rad} and Theorem \ref{perturb-hal-equi}, holds.  For such a perturbation function $\varpi$, we have $S_{W}=S_{W,\varpi}$.

Define the map
$$
\mathcal{F}: \mathcal{P}^{1,p}(l^-,l^+)\times \mathcal{V}_{\per}^{\epsilon}(\delta_0) \to \LL^{p}
$$
by
$$
\mathcal{F}(\phi,\varpi)=\pat_s \phi+J\pat_t \phi-\nabla_g(\widetilde{H}_\varpi)(\phi).
$$
We define the universal moduli space
\begin{align*}
\tM(l^-,l^+)=\{(\phi, \varpi)\in \mathcal{P}^{1,p}(l^-,l^+)\times \mathcal{V}_{\per}^{\epsilon,\delta_0} |\mathcal{F}(\phi,\varpi)=0\}.
\end{align*}
This moduli space depends on the metric $g$ which is given by the K\"ahler form $\omega_\varphi=\omega+\sqrt{-1}\pat\bpat \varphi,\varphi\in \mathcal{H}$.  Note that by Theorem \ref{transversal} and Theorem \ref{perturb-hal-equi}, we can choose generic $\varphi\in \mathcal{H}$ such that $S_{\re(W)}$ has finitely many trajectories.

The linearization operator of  $\mathcal{F}$ at the zero point $(\phi,\varpi)\in \tM(l^-,l^+;g)$ is given by
$$
D\mathcal{F}_{(\phi,\varpi)}(\xi,\chi)=\nabla_s \xi+J (\phi)\nabla_t \xi-\nabla_{\xi}\nabla_g\widetilde{H}_\varpi(\phi)-\nabla_g \chi.
$$
Applying the same argument as in \cite{Oh0}, we have the following conclusion.

\begin{thm}
If $\phi$ is a non-constant solution of (\ref{perturb-lg}), then the set
$$
\Sigma_{\phi}=\{(s,t)\in \Sigma\;|\; \frac{\pat \phi}{\pat s}\neq 0, \frac{\pat \phi}{\pat t}\neq 0, \phi(s,t)\notin \phi((\R\setminus \{s\})\times \{t\})\}
$$
 is open and dense in $\Sigma$.
\end{thm}

\begin{prop}\label{surjection-stripe}
$\tM(l^-,l^+)$ is an infinite dimensional Banach manifold.
\end{prop}

\begin{proof} If we can show that $D\mathcal{F}$ is surjective at zero point of $\mathcal{F}$, then by applying the implicit functional theorem we can obtain the conclusion.

To prove that $D\mathcal{F}$ is surjective at $(\phi,\varpi)\in  \tM(l^-,l^+)$, we use the contradiction method.

Suppose that it is not surjective at the zero point $(\phi,\varpi)$, then there is a nonzero section $\gamma\in  L^{q}(\Sigma,{\phi}^*TM) $ such that
\begin{equation}\label{sec5: equa-tran-1}
\int_{\Sigma} g(\nabla_s \xi+J \nabla_t \xi-\nabla_{\xi}(\nabla_g \widetilde{H}),\gamma)=0 \;\text{and}\; \int_{\Sigma} g(\nabla_g \chi,\gamma)=0,
\end{equation}
for any pair $(\xi,\chi)$ in the tangent space. The first equality shows that $\gamma$ is a weak solution of the dual equation $D\mathcal{F}_{(\phi,\varpi)}^*\gamma=0$. By the elliptic regularity, we know that $\gamma$ is smooth.

Choose a point $p_0=(s_0,t_0)\in \Sigma$, that $\phi^{-1}(\phi(p_0))$ consists only of finite points. We denote them by $p_0,\cdots,p_m$.  As $p_0\in \Sigma$, we can choose a small neighborhood $U$ of $\phi(p_0)$ such that $\phi^{-1}(U)$ consists of disjoint open sets $V_0,\cdots,V_m$ satisfying the conditions that $p_0\in V_0$, $\phi: V_0 \to U_0=\phi(V_0)\subset U$ being an embedding and the distance $d(\Pro_2 (V_0), \Pro_2(\phi^{-1}(U_0)\setminus V_0))>0$, where $\Pro_2: \Sigma=\R\times [0,1] \to [0,1]$ is the projection.

If we can find a function $\chi_0$ on $M$ with $\nabla_g \chi_0$ supported on $U$, and

\begin{equation}\label{sec5:ineq-1}
\int_{V_0}g(\nabla_g\chi_0,\gamma)dsdt>0,
\end{equation}
then we can choose a function $\alpha:[0,1]\to [0,1]$ which is one on $\Pro_2(V_0)$ and zero on $\Pro_2(\phi^{-1}(V)\setminus V_0)$. Let $\chi=\alpha \chi_0$, then we have
$$
\int_{\Sigma} g(\nabla_g\chi, \gamma)dsdt>0,
$$
which contradicts with the second equality of (\ref{sec5: equa-tran-1}) and the theorem is proved.

In order to find such a $\chi_0$, we introduce  two smooth auxiliary functions $\rho_1$ and $\rho_2$ as follows. Let $\rho_1\in C^\infty(\R)$ be a positive even function satisfying
$$
\begin{cases}
\rho_1(s)=0,&   s\le -1\\
0\le \rho_1(s)\le 1,&  -1\le s\le -\frac{1}{2}\\
\rho_1(s)=1,&-\frac{1}{2}\le s\le 0\\
\end{cases}
$$
and let $\rho_2\in C^\infty(\R)$ be an increasing positive function satisfying
$$
\begin{cases}
\rho_2(s)=0,&  s\le -1\\
\rho_2(s)=1,& s \ge 1\\
0\le \rho_2(s)\le 1,&-1\le s\le 1\\
\rho_2'(0)=\frac{1}{2}.
\end{cases}
$$
Note that $\rho_1$ is the same one as in the proof of Theorem \ref{transversal}.

We have two cases. \

Case 1: $g(\gamma(p_0), \xi)=0, \forall \xi\in \im (d\phi(p_0))$.\

In this case, we must have $n\ge 2$. Since $\phi$ is an embedding from $V_0$ to $U_0=\phi(V_0)\subset U$, we can choose local coordinate system near $\phi(p_0)$ such that $\phi$ can be identified with the standard embedding $(x_1,y_1)\mapsto (x_1,y_1,0,\cdots,0)^T$ near $p_0=(0,0)$. Assume that $\gamma(0)=(0,\cdots,0,c)^T$ for some $c>0$.

Let $\chi_{\epsilon}(x_1\cdots,y_{n})=\rho_1(x_1/\epsilon)\cdots\rho_1(y_{n}/\epsilon) \rho_2(y_{n}/\epsilon)$ for any $\epsilon>0$. $\chi_{\epsilon}$ is defined on $M$ if $\epsilon$ is small enough.

We have
\begin{align*}
&\int_{V_0}g(\nabla_g \chi_{\epsilon}(\phi(s,t)),\gamma(s,t))dsdt=\int_{V_0} d\chi_\epsilon(\phi(s,t))(\gamma(s,t))dsdt\\
=&\int_{[-\epsilon,\epsilon]^{2}}  (\rho_1'(x/\epsilon)\rho_1(y/\epsilon)/\epsilon,\rho_1'(y/\epsilon)\rho_1(x/\epsilon)/\epsilon,0,\cdots,0,\frac{1}{2\epsilon})^T \cdot \gamma dxdy\\
=&\int_{[-\epsilon,\epsilon]^{2}} (\rho_1'(x/\epsilon)\rho_1(y/\epsilon)/\epsilon,\rho_1'(y/\epsilon)\rho_1(x/\epsilon)/\epsilon,0,\cdots,0,0)^T \cdot \gamma dxdy\\
+&\int_{[-\epsilon,\epsilon]^{2}} (0,\cdots,0,\frac{1}{2\epsilon})^T \cdot\gamma dxdy\\
=&2c\epsilon+O(\epsilon^2)>0,
\end{align*}
for sufficiently small $\epsilon$.

\noindent Case 2: $g(\gamma(p_0), \xi)>0$ for some $\xi \in \im (d\phi(p_0))$.\

As done in Case 1,  we can choose a local coordinate system $(x_1,y_1,\cdots, y_n)$ in the neighborhood $U_0$ of $p_0=(0,0)$ and identify $\phi$ with the standard embedding $(x_1,y_1)\mapsto (x_1,y_1,0,\cdots,0)^T$ such that $\frac{\xi(0)}{|\xi(0)|_g}=e_1:=(1,0,\cdots,0)^T$. Let $a=e_1^T\cdot \gamma(0)>0$.

For small $\epsilon>0$, let $\chi_{\epsilon}=\rho_2(x_1/\epsilon)\rho_1(x_2/\epsilon)\cdots\rho_1(x_{2n})$.

We have
\begin{align*}
&\int_{V_0}g(\nabla_g \chi_{\epsilon},\gamma)dsdt=\int_{V_0} d \chi_\epsilon(\gamma(s,t))dsdt\\
=&\int_{[-\epsilon,\epsilon]^{2}}  (\rho_2'(x/\epsilon)\rho_1(y/\epsilon)/\epsilon,\rho_1'(y/\epsilon)\rho_2(x/\epsilon)/\epsilon,0,\cdots,0,0)^T\cdot\gamma dxdy\\
=&a \int^1_{-1}\rho_1(y)dy \epsilon+O(\epsilon^2)>0,
\end{align*}
for sufficiently small $\epsilon$.

Note that in each case the function $\chi_\epsilon$ satisfies $\|\nabla^k\chi\|<C(1/\epsilon)^k$. So if we choose $b$ and $\epsilon$ small enough, then we find that the function $\chi_0:=b\chi_\epsilon \in \mathcal{V}_{\per}^\epsilon$ and satisfies (\ref{sec5:ineq-1}).

\end{proof}

\begin{thm}\label{moduli-smooth-stripe}
 For generic $\varpi\in \mathcal{V}_{\per}^\epsilon(\delta_0)$, the moduli space $\tM(l^-,l^+;\varpi)$ is a smooth manifold with dimension $\ind(D_\phi)$ for any $\phi \in \tM(l^-,l^+; \varpi)$.
\end{thm}

\begin{proof} Define the projection $\pi: \tM(l^-,l^+)\subset \mathcal{P}^{1,p}(l^-,l^+)\times \mathcal{V}_{\per}^{\epsilon,\delta_0}\to \mathcal{V}_{\per}^{\epsilon,\delta_0}$. Then a routine argument (ref. \cite[Proof of Theorem 3.1.2]{MS}) shows that $\pi$ is Fredholm and $\ind(d\pi_{(\phi,\varpi)})=\ind D_\phi$. By Sard-Smale theorem (ref. \cite[Proposition 2.24]{Ck}), the regular value of $\pi$ is generic in $\mathcal{V}_{\per}^{\epsilon,\delta_0}$. If taking a regular $\varpi$, $D_{\phi}$ is surjective.

\end{proof}

\subsubsection*{Transversality in general}\label{subsec-trans-general}\

Fixed a $\varpi$ so that $\tM(l^-,l^+; \varpi)$ is smooth for any pair $(l^-,l^+)$ of Hamiltonian chords.
In subsection \ref{perturb-witten-section}, we have defined the space of consistent admissible perturbation family in $\Omega^{0,1}_{\cC/{\overline{\cR}}}(\cC, \mathcal{H})$. Choose a positive sequence $\epsilon=\{\epsilon_k\}$ such that $\epsilon_k\to 0$ rapidly enough. We define a norm on the disc $\cC$:
$$
\|K\|_{\cC,\epsilon}=\sum_{k=0}^{\infty} ||\pat^k K||_{C^0(M)}\epsilon_k,
$$
and a uniform norm on the moduli space $\overline{\cR}$:
$$
\|K\|_{\epsilon}=\sup_{\cC\in \overline{\cR}}\|K\|_{\cC}.
$$

We define $\mathcal{T}^{\delta_0}$ to be the space of consistent admissible $K$, with $\|K\|_\epsilon<\delta_0$.

In the beginning of Section (\ref{Fredholm theory section}) we defined the extended map $WI_K$ on $\B^{d+1}(L, l)$. If we also take the variation of $K$ into consideration, we can define a universal Witten map $WI$ by letting   $WI(\phi,\cC,K)=\WI_K(\phi,\cC)=\WI_{\cC, K}(\phi)$. Its linearization at $(\cC,\phi,K)$ can be written as

$$
D_{\cC,\phi,K}(\gamma,\xi,\kappa)=D_{\cC}\gamma+D_{\phi}\xi+D_{K}\kappa.
$$

An element $\chi$ of the cokernel of $D_{\cC,\phi,K}$ is an element in $L^q(\cC, \Omega^{0,1} \otimes \phi^* TM)$ such that $D_{\cC,\phi,K}^*\chi=0$ (here $q$ satisfies $\frac{1}{p}+\frac{1}{q}=1$). So $D_{\phi}^*\chi=0$, and by elliptic estimate, we know that $\chi$ is smooth. Moreover, we have
$$
\int_r \langle \chi,D_{K}\kappa \rangle=0,
$$
for any $\kappa\in T_K \mathcal{T}^{\delta_0}$. However, by definition, $D_{K}\kappa$ is the $(0,1)$ part of the  vector-field-valued form on the disc, induced from $\kappa$. By similar arguments as in the proof of Proposition \ref{surjection-stripe}, we get $\chi\equiv 0$. So we have
\begin{lm}
$D_{\cC,\phi,K}$ is surjective.
\end{lm}

Define
$$
\W^{d+1}_{\delta_0}(L,l):=\{(\phi,\cC,K)| \phi \in \W^{d+1}(\cC,L,l; K), \cC\in \overline{\cR},||K||_\epsilon<\delta_0\}.
$$
It is the zero locus of the universal Witten map $WI$.

By the above lemma, we have

\begin{prop}
$\W^{d+1}_{\delta_0}(L,l)$ is  an infinite dimensional Banach manifold.
\end{prop}

\begin{thm}\label{moduli-smooth-general}

For generic $K \in \mathcal{T}^{\delta_0}$, $\W^{d+1}(L,l; K)$ is an oriented smooth manifold. The dimension $m$ of $\W^{d+1}(L,l; K)$ satisfies
\begin{equation}\label{dimension}
f_1(m-(d-2))=-y_1-\cdots-y_d+y_0.
\end{equation}
\end{thm}

 \begin{proof}
 For smoothness, we can use the projection $\pi_{WI}$ from the zero locus of the universal Witten map $WI$ to $\B^{d+1}(L, l)$. By the Sard-Smale theorem (see  \cite[Proposition 2.24]{Ck} for a reference), the regular value of $\pi_{WI}$ is generic. If we choose a regular $K$, then $DWI_K$ is surjective by a routine argument as in ref. \cite[Proof of Theorem 3.1.2]{MS}.

Note that $DWI_K(\xi,\gamma)=D_{\phi}\xi+D_{\cC}\gamma$, so we have
$$
\text{i}(DWI_K)=\text{i}(D_{\phi})+\text{i}(D_{\cC}).
$$
As $D_{\cC}$ encodes the variation of disc configurations, so it has index $d+1-3=d-2$. Then (\ref{dimension})  follows from (\ref{fredholm-index}).

The orientation comes from the orientation bundle relation(\ref{orientation-bundle}). Here we can use a general theory as in 12b of \cite{Si3}. We can also use an easy oberservation on the bundle $F$ we defined on the boundary of the disc which is use in the definition of $\hat{D}$ in (\ref{orientation-bundle}). As all Lefschetz thimbles are simply connected, $F$ is trival. Thus we can get an natrual orientation for $\text{det}(\hat{D})$ by 8.1.4 of \cite{fooo2}.
 \end{proof}

 \begin{df} \label{rigid}
 We call such a $(\cC, \phi) \in \W^{d+1}(L,l; K)$ to be a \textit{rigid solution},  if the extended index $\text{i}(DWI_K)$ at $(\cC,\phi)$ is zero. For completeness, we also consider such $\tM(l^-,l^+; \varpi)$ of index 1and then the quotient space $\W (l^-,l^+; \varpi)=\tM(l^-,l^+; \varpi)/\R$. Any element $[\phi]\in \W (l^-,l^+; \varpi)$ is also called a \textit{rigid solution}.
 \end{df}

\section{Fukaya category of Landau-Ginzburg model}
\subsection{$A_{\infty}$ relation with $\pmb{G}$ grading}\label{subsec-9.1}\

In subsection \ref{index-orientation section}, we have introduced the notions of Grading datum, $\pmb{G}$-graded vector space and $\pmb{G}$-graded module.  Furthermore, we give a brief introduction of  $\pmb{G}$-graded $A_{\infty}$ structure here. These notions are almost the same as in  \cite{Sh1}, section 2.

Recall that the grading datum is made up from a unsigned datum, which is an abelian group $Y$ together
  with a morphism $f:\Z\mapsto Y$.
\begin{df}
For $\pmb{G}$-graded bimodules $A,B$, a homomorphism $\mu$ between $A^{\otimes_{\pmb{R}}s}$ and $B$ is called \textit{of degree $r$}, if the grading of $\mu(a_1, \cdots, a_s)$ is $r$ plus the sum of the gradings of $a_1,\dots,a_s$ for any $(a_1,\cdots, a_s)\in A^{\otimes_{\pmb{R}}s}$. We denote the subspace of such morphisms by $\hom_{\pmb{R}-\text{bi-mod}}(A^{\otimes_{\pmb{R}}s},B)_{r}$.
\end{df}

We can then define a $(\pmb{G},\Z)$ graded bimodule, whose degree $(\mathsf{y}, s)$ part is
$$
CC_c^{\mathsf{y}}(A,B)^s=\hom_{\pmb{R}-\text{bi-mod}}(A^{\otimes_{\pmb{R}}s},B)_{\mathsf{y}-f(s)}.
$$

We write
$$
CC_{\pmb{G}}^{\mathsf{y}}(A,B):=\prod_{s>0}CC_c^{\mathsf{y}}(A,B)^s, \;\; CC_{\pmb{G}}^{*}(A,B):=\bigoplus_{{\mathsf{y}}\in Y}CC_{\pmb{G}}^{\mathsf{y}}(A,B).
$$
We also denote
$$
CC_{\pmb{G}}^s(A,B):= CC_{\pmb{G}}^{f(s)}(A,B),
$$
for simplicity.

\begin{df}
The \textit{Gerstenhaber product}
\begin{align*}
 CC_{\pmb{G}}^{*}(A,B)\otimes_{\pmb{R}}&CC_{\pmb{G}}^{*}(A,A)\mapsto CC_{\pmb{G}}^{*}(A,B)\\
 \mu_1 \otimes \mu_2 &\to\mu_1 \circ \mu_2,
\end{align*}
is defined by
 \begin{align}\label{Gerstenhaber-relation}
 &(\mu_1 \circ \mu_2)^d(a_1,\cdots,a_d) \nonumber\\
 =&\sum_{i+j+k=d}(-1)^{\dag}\mu_1^{i+k+1}(a_{1}\cdots a_{i},\mu_2^j(a_{i+1},\cdots,a_{i+j}),a_{i+j+1},\cdots,a_{i+j+k}).
 \end{align}

 Here is an explanation of the notations.  We denote by $\mu_i^{s}$ the component of $\mu_i$ with $s$ inputs. We then put $\dag=(\sigma(\mu_2)+1)((\sigma(a_1)+\cdots+\sigma(a_i)-i)$\footnote{Here we assume $\mu_1$ and $\mu_2$ are of homogeneous degree, and extend by linearity.}, where the $\Z_2$ number $\sigma(\mu)$ is decided as follows: suppose the map $\mu$ have a $(\pmb{G},\Z)$ grading $(\mathsf{y},s)$, the number $\sigma(\mu)$ is defined to be $\sigma(\mathsf{y})$, where $\sigma$ is the morphism $Y\mapsto \Z_2$ in defining the sign morphism $\pmb{\sigma}$ in Definition \ref{signed-grading}. Finally, if the element $a$ has grading $\mathsf{y}$, we denote $\sigma(\mathsf{y})$ by $\sigma(a)$ for simpliticy.
\end{df}
\begin{df}
Let $A=\oplus_{y\in Y}A_{y}$ be a $\pmb{G}$-graded module
over a $\pmb{G}$-graded algebra $\pmb{R}$. If there is a $\mu$ in $CC_c^2(A,A)$, so that $\mu\circ\mu=0$, we call it a \textit{$\pmb{G}$-graded $A_{\infty}$ structure on $A$ over $\pmb{R}$}.
 \end{df}

In component form, a  $\pmb{G}$-graded $A_{\infty}$ structure is a sequence $\mu_1, \mu_2, \cdots $ of elements in $CC_c^2(A,A)$, so that $\mu^d$ has $d$ inputs, and
\begin{equation}\label{a-infty-relation}
\sum_{i+j+k=d}(-1)^{\dag}\mu^{i+k+1}(a_{1}\cdots a_{i},\mu^j(a_{i+1},\cdots,a_{i+j}),a_{i+j+1},\cdots,a_{i+j+k})=0.
\end{equation}

If $d=1$, (\ref{a-infty-relation}) gives
$$
\mu^1\circ \mu^1=0.
$$

Thus we get a Floer chain map $\mu^1:A\mapsto A$ of degree $f(1)$, that is, $A_y$ is mapped to $A_{y+f(1)}$, satisfying $\mu^1\circ \mu^1=0$. We can define the  Floer homology groups by putting
\begin{equation}\label{hf}
HF_y(A)=\text{Ker}( \mu^1|_{A_y})/\im ( \mu^1|_{A_{y-f(1)}}).
\end{equation}

\subsection{Floer homology and Fukaya category of Landau-Ginzburg model}\label{subsection9.2}
\subsubsection*{Solution counting}\

Recall that in subsection \ref{subsec-trans-general}, we can choose a consistent admissible perturbation$K$ with respect to $\varpi\in V^\epsilon_{\per}$, so that all moduli spaces $\W^{d+1}(L, l;K)$ involved are smooth. We fix a $K$ now.

For Landau-Ginzburg branes $L_0^{\#},\cdots,L_d^{\#}$ and Hamiltonian chords $l^0,l^1,\cdots,l^d$, so that $l^i\in S_W(L_{i-1},L_i), i=1,\cdots,d$ and $l^0\in S_W(L_0,L_d)$, we can define a number
$$
\langle l^0,l^1,\cdots,l^d\rangle,
$$
by counting the rigid solutions in $\W^{d+1}(L, l;K)$ associated to these $\{l^i\}$ and $L_i$. This counting is summing with sign, where the sign is determined by the orientation of moduli spaces.

\subsubsection*{The Fukaya category of LG model}\

The Fukaya category $\text{Fuk}(M,h,W)$ of a Landau-Ginzburg model $(M,h,W)$ consists of the following data:

\begin{enumerate}
\item A set $\text{Ob}(\text{Fuk}(M,h,W))$, consisting of all Landau-Ginzburg branes.
\item For each pair $(L_0^{\#},L_1^{\#})$ of Landau-Ginzburg branes, a morphism space defined by $\text{Hom}(L_0^{\#},L_1^{\#})=CF(L_0^{\#},L_1^{\#})$.
\item Composition maps
$$
\mu^d:\text{Hom}(L_0^{\#},L_1^{\#})\otimes\cdots\otimes \text{Hom}(L_{d-1}^{\#},L_d^{\#}) \mapsto \text{Hom}(L_0^{\#},L_d^{\#}),
$$
defined by
$$
\mu^d(l^1,\cdots,l^d)=\sum_{l\in S_W}(-1)^{\sigma(\mathsf{y}_0)-(\sigma(\mathsf{y}_1)+\cdots+\sigma(\mathsf{y}_d)}\langle l,l^1,\cdots,l^d\rangle l,
$$
where $\mathsf{y}^i$ is the grading of $l^i$ and $\sigma$ is the sign morphism $H^1(\mathcal{G}M)\oplus \Z\mapsto \Z_2$ defined in subsection \ref{index-orientation section}.
\end{enumerate}

Using all these $CF(L_0^{\#},L_1^{\#})$, we can define a new $\pmb{G}$-graded module $CF(W)$, by
$$
CF(W)=\bigoplus_{L_0^{\#},\;L_1^{\#}:\text{ LG branes}} CF(L_0^{\#},L_1^{\#}).
$$

Clearly, it is generated by those Hamiltonian chords between all LG branes.  We denote this set by $S_W^{\#}$. On the other hand, using all these $\mu^d$ we can get an element $\mu\in CC_{\pmb{G}}^2(CF(W),CF(W))$ whose $CC_c^2(CF(W),CF(W))^d$ components are $\mu^d$.

To see that $\mu^d\in CC_c^2(CF(W),CF(W))^d$, consider $l^0,\cdots,l^d$, where each $l^i$ has grading $\mathsf{y}_i=(y_i,\mathsf{k}_i)$.  As the moduli space involved is of dimension zero, we have
$$
f_1(2-s)=y_0-y_1-\cdots-y_d,
$$
by (\ref{dimension}). As $f(s)=(f_1(s),0)$, we have
$$
f(2-s)=(y_0-y_1-\cdots-y_d,0).
$$
However, the moduli space is empty unless
$\mathsf{k}_0-\mathsf{k}_1-\cdots-\mathsf{k}_d=0$, as pointed out at the end of subsection \ref{sec5-cano-sect}. This means
$$
f(2-s)=(\mathsf{y}_0-\mathsf{y}_1-\cdots-\mathsf{y}_d).
$$

By Theorem \ref{gromov:st}, $\mu^d$ satisfies the $A_{\infty}$ relation (\ref{a-infty-relation}). Thus, by definition, $\mu$ is a $\pmb{G}$-graded $A_{\infty}$ structure on $CF(W)$. The first term $\mu^1$ gives the Floer homology $(\ref{hf})$ of the Landau-Ginzburg  model.

\bibliographystyle{amsplain}

\begin{thebibliography}{ADKMV}
\bibitem{Ab1}{M. Abouzaid}, {\it On homological mirror symmetry for toric varieties, Ph.D. Thesis at the University of Chicago (2007)}.

\bibitem{Ab2}\bysame,  {\it Homogeneous coordinate rings and mirror symmetry for toric varieties}, Geom. Topol. 10 (2006), 1097-1156.

\bibitem{Ab3}\bysame, {\it Morse homology, tropical geometry, and homological mirror symmetry for toric varieties}, Selecta. Math. (N.S.) 15 (2009), no. 2, 189-270.

\bibitem{Ab4} \bysame, {\it Family Floer cohomology and mirror symmetry}, Proceedings of the International Congress of Mathematicians---Seoul 2014. Vol. II, 813-836, 2014; \textit{The family Floer functor is faithful}, J. Eur. Math. Soc. 19 (2017), no. 7, 2139-2217.


\bibitem{AI} {M. Abouzaid, I. Smith}, {\it Homological mirror symmetry for the 4-torus}, Duke Math. J. 152 (2010), no. 3, 373-440.

\bibitem{Au}{D. Auroux}, {\it Special Lagrangian fibrations, wall-crossing, and mirror symmetry}, Surveys in differential geometry. Vol. XIII. Geometry, analysis, and algebraic geometry: forty years of the Journal of Differential Geometry, 1-47, Surv. Differ. Geom., 13 (2009).

\bibitem{Au2} \bysame, \textit{Speculations on homological mirror symmetry for hypersurfaces in $(\mathbb{C}^\ast)^n$}, arXiv preprint arXiv:1705.06667 (2017).

\bibitem{AAK}{M. Abouzaid, D. Auroux, and L. Katzarkov}, {\it Lagrangian fibrations on blowups of toric varieties and mirror symmetry for hypersurfaces}, Publ. Math. Inst. Hautes \'Etudes Sci. 123 (2016), 199-282.

\bibitem{BH}A. Banyaga and D. Hurtubise, {\it Lecture on Morse homology}, Springer Science+Business media, B. V. ,  2004.


\bibitem{CGMS}K. Cieliebak, A.R. Gaio, I. Mundet i Riera, and D.A. Salamon, {\it The symplectic vortex equations and invariants of Hamiltonian group actions}, J. Symplectic Geom. 1 (3) (2002), 543-645.

\bibitem{Ck}K. C. Chang, \it{Methods in nonlinear analysis}, Springer Monographs in Mathematics, Springer-Verlag Berlin Heidelberg, 2005.

\bibitem{CIR} A. Chiodo, H. Iritani, and Y. Ruan, {\it Landau-Ginzburg/Calabi-Yau correspondence, global mirror symmetry and Orlov equivalence}, Publ. Math. Inst. Hautes \'Etudes Sci. 119 (2014), 127-216.

\bibitem{Fa}H. Fan, \it{Schroedinger equations, deformation theory and tt*-geometry}, Preprint at arXiv: 1107. 1290v1 [math-ph].

\bibitem{FJR1}{H. Fan, T. Jarvis, and Y. Ruan}, {\it Geometry and analysis of spin equations},  Comm. Pure Applied Math.
61(2008), 715-788.

\bibitem{FJR2}\bysame, {\it The Witten equation, mirror symmetry and quantum singularity theory},  Ann. of Math.(2), 178.1 (2013): 1-106.

\bibitem{FJR3}\bysame, {\it The Witten equation and its virtual fundamental cycle},  arXiv:0712.4025, appear in AMS book "Quantization of Singularity Theory, 266 pages, 2019.

\bibitem{f1}{A. Floer}, {\it The unregularized gradient flow of the symplectic action},
    Comm. Pure Appl. Math., 1988, 41(6): 775-813.



\bibitem{fooo2}{K. Fukaya, Y. G. Oh, H.Ohta, and K. Ono}, {\it Lagrangian Floer theory on compact toric manifolds I}, arXiv preprint arXiv:0802.1703 (2008).

\bibitem{fooo3}{K. Fukaya, Y. G. Oh, H.Ohta, and K. Ono} {\it Anchored Lagrangian submanifolds and their Floer theory}(English summary) Mirror symmetry and tropical geometry, Contemp. Math., 527, Amer. Math. Soc., Providence, RI, 2010. 15-54.

\bibitem{Fu1}{K. Fukaya}, {\it Floer homology for families---a progress report}, Integrable systems, topology, and physics (Tokyo, 2000), 33-68, Contemp. Math., 309 (2002).

\bibitem{Fu2}\bysame, {\it Mirror symmetry of abelian varieties and multi-theta functions}, J. Algebraic Geom. 11 (2002), no. 3, 393-512

\bibitem{GMW}{D.Gaiotto, G. Moore, and E. Witten}, {\it Algebra of the infrared: string field theoretic structures in massive N=(2, 2) field theory in two dimensions}, arXiv:1506.04087 [hep-th], 2015.

\bibitem{GKR}{M. Gross, L. Katzarkov, and H. Ruddat}, {\it Towards mirror symmetry for varieties of general type}, Adv. Math. 308 (2017), 208-275.

\bibitem{h1}{A. Haydys}, {\it Fukaya-Seidel category and gauge theory}, J. Symplectic Geom. 13 (2015), no. 1, 151-207.

\bibitem{HV}{K. Hori, A. Iqbal, and  C. Vafa}, {\it D-branes and mirror symmetry},  arXiv preprint hep-th/0005247, 2000.

\bibitem{ay} {A. Kapustin, Y. Li}, {\it D-branes in Landau-Ginzburg models and algebraic geometry},  J.  High Energy Phys. 2003(12): 005.

\bibitem{KKP}{L, Katzarkov, M. Kontsevich, and T. Pantev}, {\it Hodge theoretic aspects of mirror symmetry}, From Hodge theory to integrability and TQFT $tt^\ast$-geometry, 87-174, Proc. Sympos. Pure Math., 78, 2008.

\bibitem{k1}{M. Kontsevich}, {\it Homological algebra of mirror symmetry}, arXiv:alg-geom/9411018, 1994.

\bibitem{k2}\bysame, {\it Lectures at ENS, Paris, Spring 1998}, notes taken by J. Bellaiche, J.-F. Dat, I. Marin, G. Racinet and H. Randriambololona, unpublished.

\bibitem{KKS}{M. Kapranov, M. Kontsevich, and Y. Soibelman}, {\it Algebra of the infrared and secondary polytopes.} arXiv preprint arXiv:1408.2673 (2014).

\bibitem{KS}{M. Kontsevich, Maxim, and Y. Soibelman}, {\it Homological mirror symmetry and torus fibrations}, Symplectic geometry and mirror symmetry (Seoul, 2000), 203-263, World Sci. Publ. (2001).


\bibitem{MS}{D. McDuff, D. Salamon}, {\it J-holomorphic curves and symplectic topology},  Math. Soc. Coll. Publ., 52. Amer. Math. Soc., Providence, R.I., 2004. preprint hep-th/0005247, 2000.

\bibitem{Oh0}{Y. G. Oh}, {\it Floer cohomology of Lagrangian intersections and pseudo-holomorphic disks. I}, Comm. Pure Appl. Math. 46 (1993), no. 7, 949-993. Addendum, Comm. Pure Appl. Math. 48 (1995), no. 11, 1299-1302.

\bibitem{Oh1}\bysame, {\it Symplectic Topology and Floer Homology: Volume 1: Symplectic Geometry and
Pseudoholomorphic Curves} New Mathematical Monographs Book  28, Cambridge University Press, 2015.

\bibitem{Oh2}\bysame, {\it Symplectic Topology and Floer Homology: Volume 2, Floer Homology and its Applications} New Mathematical Monographs Book 29, Cambridge University Press, 2015.

\bibitem{Oh3}\bysame, {\it Removal of boundary singularities of pseudo-holomorphic curves with Lagrangian boundary conditions}, Commun. Pure Appl. Math. 45 (1992), 121-139.

\bibitem{or1}{D. Orlov}, {\it Triangulated categories of singularities and D-branes in Landau-Ginzburg models}, Proc. Steklov Inst. Math. 2004, no. 3 (246), 227--248.

\bibitem{or2}\bysame, {\it Matrix factorizations for nonaffine LG-models},  Math. Ann., 2012, 353(1): 95-108.

\bibitem{or3}\bysame, {\it Orlov, Dmitri Derived categories of coherent sheaves and triangulated categories of singularities}, Algebra, arithmetic, and geometry: in honor of Yu. I. Manin. Vol. II, 503-531, Progr. Math., 270, Birkh\"auser, 2009.

\bibitem{PZ} {A. Polishchuk and E. Zaslow}, {\it Categorical mirror symmetry in the elliptic curve}, Winter School on Mirror Symmetry, Vector Bundles and Lagrangian Submanifolds (Cambridge, MA, 1999), 275-295, AMS/IP Stud. Adv. Math. 23 (2001).



\bibitem{Ru}{Y. Ruan}, {\it Lefschetz thimbles and LG-Floer theory}, Personal communication, 2013.


\bibitem{Sh1} {N. Sheridan}, {\it Homological mirror symmetry for Calabi-Yau hypersurfaces in projective space}, Invent. Math. 199 (2015), no. 1, 1-186.

\bibitem{Sh2} \bysame, {\it On the Fukaya category of a Fano hypersurface in projective space}, Publ. Math. Inst. Hautes \'Etudes Sci. 124 (2016), 165-317.

\bibitem{Si1}{P. Seidel}, {\it Vanishing cycles and mutation}, European congress of mathematics. Birkh$\ddot{\text{a}}$user Basel, 2001: 65-85.

\bibitem{Si2}\bysame, {\it More about vanishing cycles and mutation}, arXiv preprint math/0010032, 2000.

\bibitem{Si3}\bysame, {\it Fukaya categories and Picard-Lefschetz theory}, European Mathematical Society, 2008.

\bibitem{Si4}\bysame, {\it Fukaya $A_\infty$-structures associated to Lefschetz fibrations. I.}, J. Symp. Geom., 2012, 10(3): 325-388.

\bibitem{Si5}\bysame, {\it Fukaya $A_\infty$-structures associated to Lefschetz fibrations. II.}, Algebra, geometry, and physics in the 21st century, 295-364, Progr. Math., 324, Birkh\"auser/Springer, 2017.

\bibitem{Si6}\bysame, {\it Fukaya $A_\infty$-structures associated to Lefschetz fibrations.  II 1/2.}, Adv. Theor. Math. Phys. 20 (2016), no. 4, 883-944.

\bibitem{Si7}\bysame,{\it Fukaya $A_\infty$-structures associated to Lefschetz fibrations.  III.}, preprint arXiv:1608.04012, 2016.

\bibitem{Si8}\bysame, {\it Fukaya $A_\infty$-structures associated to Lefschetz fibrations. IV.}, preprint arXiv:1709.06018, 2017.

\bibitem{Si9}\bysame, {\it Fukaya $A_\infty$-structures associated to Lefschetz fibrations. IV 1/2.}, preprint arXiv:1802.09461, 2018.

\bibitem{Si10}\bysame, {\it Fukaya $A_\infty$-structures associated to Lefschetz fibrations. VI.}, preprint arXiv:1810.07119, 2018.

\bibitem{Si11}\bysame,{\it Homological mirror symmetry for the quartic surface}, Mem. Amer. Math. Soc. 236 (2015), no. 1116, vi+129 pp.

\bibitem{Si-ICM}\bysame, {\it Fukaya categories and deformations}, Proceedings of the International Congress of Mathematicians, Vol. II (Beijing, 2002), 351-360 (2002).

\bibitem{Sy}{Z. Sylvan}, {\it On partially wrapped Fukaya categories} preprint arXiv:1604.02540, 2016.

\bibitem{Tu} {J. Tu}, \textit{On the reconstruction problem in mirror symmetry}, Adv. Math. 256 (2014), 449-478.

\end{thebibliography}

\end{document}